\documentclass[]{scrartcl}

\usepackage[utf8]{luainputenc}
\usepackage[USenglish]{babel}
\usepackage{csquotes}

\usepackage[a4paper,top=27mm,bottom=20mm,inner=25mm,outer=20mm]{geometry}

\usepackage{graphicx}
\usepackage{verbatim}
\usepackage{stmaryrd}
\usepackage{amsmath}
\usepackage{comment}
\usepackage{listings}
\usepackage[table,xcdraw]{xcolor}
\usepackage{color}
\usepackage{amsthm}
\usepackage{times}
\usepackage[hidelinks]{hyperref}
\usepackage{multirow}
\usepackage{makecell}
\usepackage{booktabs}
\usepackage{array}
\usepackage{enumitem}
\usepackage{amssymb}
\usepackage{subcaption}
\usepackage{siunitx}
\usepackage{bm}
\usepackage[normalem]{ulem}
\usepackage{adjustbox}

\usepackage[T1]{fontenc}
\usepackage{newpxtext,newpxmath}

\usepackage[%
  backend=bibtex,bibencoding=ascii,
  style=numeric-comp,
  giveninits=true, uniquename=init, %abbreviate first names
  natbib=true,
  url=false,
  doi=true,
  isbn=false,
  backref=false,
  maxnames=99,
]{biblatex}
\definecolor{dkgreen}{rgb}{0,0.6,0}
\definecolor{gray}{rgb}{0.5,0.5,0.5}
\definecolor{mauve}{rgb}{0.58,0,0.82}

\newcommand*{\ldblbrace}{\{\mskip-5mu \left \{}
\newcommand*{\rdblbrace}{ \}\mskip-5mu \right \}}
\newcommand{\mean}[1]{\left\ldblbrace #1 \right\rdblbrace}

\newcommand{\mtx}[1]{\boldsymbol{#1}}
\renewcommand{\vec}[1]{\boldsymbol{#1}}

\newcommand{\norm}[1]{\left\lVert #1 \right\rVert}
\newcommand{\jump}[1]{\left\llbracket #1  \right\rrbracket}
\newtheorem{theorem}{Theorem}[section]
\newtheorem{corollary}[theorem]{Corollary}
\newtheorem{lemma}[theorem]{Lemma}
\theoremstyle{definition}
\newtheorem{remark}[theorem]{Remark}

\definecolor{suggcolor}{rgb}{0.0,0.5,0.0}

\newcommand{\orcid}[1]{ORCID:~\href{https://orcid.org/#1}{#1}}
\usepackage{authblk}

\definecolor{dkgreen}{rgb}{0,0.6,0}
\definecolor{gray}{rgb}{0.5,0.5,0.5}
\definecolor{mauve}{rgb}{0.58,0,0.82}

\newenvironment{keywords}{\par\textbf{Key words.}}{\par}
\newenvironment{AMS}{\par\textbf{AMS subject classification.}}{\par}

\title{Entropy-Stable and Well-Balanced Discontinuous Galerkin Methods for the Compressible Euler Equations in Vector-Invariant Form}

\author[1]{Marco Artiano\thanks{\orcid{0009-0009-5872-702X}}}
\affil[1]{Institute of Mathematics, Johannes Gutenberg University Mainz, Germany}

\author[2]{Kieran Ricardo\thanks{\orcid{0009-0000-4544-3049}}}
\affil[2]{Mathematical Sciences Institute, Australian National University, Canberra, Australia}

\author[3]{Oswald Knoth}
\affil[3]{Leibniz Institute for Tropospheric Research (TROPOS), Leipzig, Germany}

\author[4]{Peter Spichtinger\thanks{\orcid{0000-0003-4008-4977}}}
\affil[4]{Institute for Atmospheric Physics, Johannes Gutenberg University Mainz, Germany}

\author[1]{Hendrik Ranocha\thanks{\orcid{0000-0002-3456-2277}}}

\begin{document}
\maketitle

\begin{abstract}
We develop structure-preserving methods for the compressible Euler equations with gravity in vector-invariant form and potential temperature as a prognostic variable within the flux-differencing discontinuous Galerkin spectral element method (DGSEM) framework.
By discretizing the nonconservative terms as symmetric and antisymmetric products, we derive two-point numerical fluxes that conserve both the thermodynamic entropy and the total energy.
Moreover, we design an entropy-stable numerical flux that is well-balanced for both isothermal and isentropic background states.
All properties are shown to carry over to the high-order DGSEM on general curvilinear meshes.
Several numerical examples confirm the theoretical findings and show the robustness and accuracy of the scheme for use in modern dynamical cores for atmospheric flows.
\end{abstract}


\newpage
\section{Introduction}
\label{sec:introduction}
In recent decades, increasing attention has been devoted to the development of modern dynamical cores, driven by advances in high-performance computing and by the need for more accurate, efficient and robust numerical methods for atmospheric flows.
Many operational and research cores are based on spectral-element methods, which combine high-order accuracy with good scalability on modern CPU and GPU architectures (see, e.g.,~\cite{giraldo2013numa,CAM-SE2018,choi2016,yatunin_clima2016,kawai2025,baldaufbridge2026,ricardo2024thermodynamicconsistencystructurepreservationsummation}); we refer to~\cite{marras2016review} for a review of element-based Galerkin methods for numerical weather prediction.

Among the high-order spectral element methods, discontinuous Galerkin (DG) methods have gained significant attention due to their ability to handle arbitrary geometries, their high-order accuracy, and their excellent scalability for parallel computing on CPU and GPU hardware~\cite{wang2013high,NUMA,waterhouse2026gpuperformanceentropystablediscontinuous,candelaresi2026massivelyparallelnumericalsimulations}.
Since the pioneering work of Carpenter et al.~\cite{Carpenter2014,CarpenterBook}, Fisher et al.~\cite{FISHER2013518} and Gassner et al.~\cite{gassner2016split}, the flux-differencing discontinuous Galerkin spectral element method (DGSEM) has emerged as a promising approach for solving the compressible Euler equations.
Combining entropy-conservative/stable fluxes~\cite{Tadmor1987,Tadmor2003} with a flux-differencing volume integral and suitable interface fluxes yields provably entropy-stable (or entropy-conservative) and well-balanced schemes.
This increases robustness~\cite{sjogreen2018high,klose2020assessing,rojas2021robustness,Souza,WARUSZEWSKI2022111507,ricardo2024thermodynamicconsistencystructurepreservationsummation,LIU2025114095} and removes the need for ad hoc stabilization techniques, such as artificial viscosity~\cite{ULLRICH2018427,CiCP-27-5}, filtering~\cite{GIRALDO20083849} and dealiasing~\cite{MENGALDO201556,winters2018comparative} to control spurious oscillations and robustness issues in underresolved simulations.
It also allows convergence analyses going beyond arguments relying on linearization~\cite{lukacova2022convergence,ranocha2026convergence,chen2026error}.

Various formulations of the compressible Euler equations are currently being explored for the development of modern dynamical cores, differing in the choice of the thermodynamic variables, such as potential temperature, entropy, internal or total energy, as well as in the form of the momentum equation, e.g., the conservative or the vector-invariant form, and it is still unclear whether one formulation offers advantages over the others~\cite{GIRFOGLIO2025106510,artiano2025structurepreservinghighordermethodscompressible,baldaufbridge2026,ricardo2024thermodynamicconsistencystructurepreservationsummation, GIRALDO20083849}.

In particular, the vector-invariant form, for which energy conservation can be established only through integration by parts, is widely used in modern dynamical cores.
For example, Taylor et al.~\cite{taylor2019} proposed an energy-consistent discretization that uses the vector-invariant form for the horizontal momentum and an advective form for the vertical transport, with the potential temperature as a prognostic variable.
The CliMA dynamical core~\cite{yatunin_clima2016} discretizes the vector-invariant form with the total energy as a prognostic variable, following a hybrid approach that couples a spectral-element method with staggered finite differences in the vertical direction.

In the DG framework, several structure-preserving discretizations have been proposed for various formulations for atmospheric flows.
Ricardo et al.~\cite{ricardo2024thermodynamicconsistencystructurepreservationsummation} proposed a DG method for the moist compressible Euler equations in vector-invariant form that is entropy-conserving and energy-conserving/stable, with the thermodynamic entropy as a prognostic variable.
Waruszewski et al.~\cite{WARUSZEWSKI2022111507} extended the entropy-conservative flux of Chandrashekar~\cite{chandrashekar2013} to the compressible Euler equations with gravity in the total energy formulation;
they proved that their scheme is well-balanced for an isothermal background state for gravity aligned to the mesh (see also~\cite{KAPPELI2014199, ChandrashekarWB, ChandrashekarWBDG, MICHELDANSAC2025106853, LIU2025114095} for different approaches to well-balanced schemes).
Artiano et al.~\cite{artiano2025structurepreservinghighordermethodscompressible} proposed structure-preserving, entropy- and energy-conservative/stable, and well-balanced schemes for different sets of equations, using either potential temperature or the sum of kinetic and internal energy;
they also constructed a scheme that is well-balanced for an isentropic background state and extended it to arbitrary curvilinear meshes.

In this work, we focus on the spatial discretization of the compressible Euler equations in vector-invariant form, with potential temperature as the thermodynamic variable.
To the authors' knowledge, no provably entropy-stable scheme has been proposed for the vector-invariant form so far.
Leveraging the recent framework for nonconservative terms introduced in~\cite{artiano2026affordable}, we derive new two-point fluxes that guarantee conservation of total energy and thermodynamic entropy.
We further construct an entropy-stable numerical flux that guarantees the positivity of the pressure, and we prove that the resulting schemes are well-balanced for two families of hydrostatic background states, namely isentropic and isothermal background states.
The scheme is first presented in the finite volume framework and then extended to the DGSEM on arbitrary curvilinear meshes, following the work in~\cite{artiano2026affordable}.

The paper is organized as follows.
In Section~\ref{sec:vector_invariant_form}, we introduce the vector-invariant formulation of the compressible Euler equations and derive total-energy- and entropy-conservative two-point fluxes for finite volume methods.
Entropy stability is investigated in Section~\ref{sec:entropy_stability}, and pressure positivity for strong-stability-preserving (SSP) Runge--Kutta methods is proved in Section~\ref{sec:positivity}.
The well-balanced schemes for isentropic and isothermal background states and their derivation are presented in Section~\ref{sec:well_balanced}.
The novel entropy-stable and well-balanced numerical flux is presented in Section~\ref{sec:riemann_solver}.
Section~\ref{sec:dgsem} extends all conservation properties, the well-balanced schemes, and the entropy stability to the DGSEM on general curvilinear meshes.
Numerical results are reported in Section~\ref{sec:numerical_results}, and a summary with conclusions is given in Section~\ref{sec:summary_conclusions}.
\section{Vector-invariant form of the compressible Euler equations}
\label{sec:vector_invariant_form}
The compressible Euler equations in vector-invariant form can be written as~\cite{gassmann2013}

\begin{equation}
  \begin{aligned}
&\frac{\partial \varrho}{\partial t}
+
\nabla \cdot (\varrho \vec{V})
=0,\\
&\frac{\partial \vec{V}}{\partial t}
+
(\nabla \times \vec{V}) \times \vec{V}
+
\nabla\left(\frac{\norm{\vec{V}}^2}{2}\right)
+
c_p \theta \nabla \Pi
+
 \nabla \phi
=\vec{0},\\
&\frac{\partial (\varrho\theta)}{\partial t}
+
\nabla \cdot (\varrho\theta \vec{V})
=0,
  \end{aligned}
  \label{eq:vector_invariant_form}
\end{equation}
where $\varrho$ is the density, $\vec{V}$ the velocity, $\theta$ the potential temperature,
$p = p_0 \left ( \varrho \theta R / p_0 \right )^{\gamma}$ the pressure, $p_0 = \SI{100000}{Pa}$ the reference pressure,
$\Pi = \left (\varrho \theta R / p_0 \right )^{R/{c_v}}$ the Exner pressure, and $\phi = \phi(x,y,z)$ the geopotential.
We assume an ideal gas equation of state, i.e., $p = \varrho R T$ with the gas constant $R = c_p - c_v$ and the temperature $T$.
The (constant) specific heats at constant volume and pressure are denoted by $c_v$ and $c_p$, respectively, and $\gamma = c_p / c_v$ is the ratio of specific heats.

For smooth solutions, the system~\eqref{eq:vector_invariant_form} admits additional conservation laws of the form
\begin{equation}
\partial_t U(\vec{u}) + \nabla \cdot \vec{F} (\vec{u}) = 0,
\label{eq:entropy_conservation_law}
\end{equation}
where $U$ is a mathematical entropy, $\vec{F}$ the associated entropy flux, and $\vec{u} = (\varrho, \vec{V}, \varrho \theta)$ are the prognostic variables\footnote{Technically, we also include the geopotential $\phi$ in $\vec{u}$ with trivial evolution equation $\partial_t \phi = 0$. To simplify the notation, we omit the explicit dependence on $\phi$ in the mathematical entropy.}.
In particular, both the thermodynamic entropy and the total energy act as mathematical entropies for the system~\eqref{eq:vector_invariant_form}, and the corresponding entropy fluxes can be derived from the conservation laws.
In the following, we consider the total energy
\begin{equation}
U_{\varrho E} = \frac{p}{\gamma-1} + \frac{1}{2} \varrho \norm{\vec{V}}^2 + \varrho \phi, \qquad \vec{F}_{\varrho E} = \left ( U_{\varrho E} + p \right ) \vec{V},
\end{equation}
and the thermodynamic entropy
\begin{equation}
U_{\varrho s} = -\frac{\varrho s}{\gamma-1}, \qquad s = \log\left(\frac{p}{\varrho^\gamma}\right), \quad \vec{F}_{\varrho s} = U_{\varrho s} \vec{V},
\end{equation}
where $s$ is the specific entropy.
Next, we derive entropy-conservative (EC) and total-energy-conservative (TEC) numerical fluxes for the system~\eqref{eq:vector_invariant_form}.
Such fluxes are the first building blocks in the construction of many entropy-stable schemes~\cite{fjordholm2012arbitrarily,chandrashekar2013,WINTERS2017274,WARUSZEWSKI2022111507} based on Tadmor's approach~\cite{Tadmor1987,Tadmor2003}.

\subsection{Entropy analysis}
\label{sec:entropy_analysis}
Because of the particular structure of the vector-invariant form, we restrict the analysis to the two-dimensional case for simplicity, i.e., $\vec{V} = (v_1, v_3) = (u, w)$.
The extension to three dimensions follows along the same lines. The system~\eqref{eq:vector_invariant_form} in two-dimensional Cartesian coordinates reads\footnote{Traditionally, $z$ denotes the vertical coordinate in atmospheric flows, hence the indices $1$ and $3$.}
\begin{equation}
  \begin{aligned}
&\frac{\partial \varrho}{\partial t}
+ \frac{\partial (\varrho v_1)}{\partial x}
+ \frac{\partial (\varrho v_3)}{\partial z}
=0,\\
&\frac{\partial v_1}{\partial t}
+ v_3\left(\frac{\partial v_1}{\partial z}-\frac{\partial v_3}{\partial x}\right)
+ \frac{\partial}{\partial x}\left(\frac{v_1^2+v_3^2}{2}\right)
+ c_p \theta \frac{\partial \Pi}{\partial x}
+ \frac{\partial \phi}{\partial x}
=0,\\
&\frac{\partial v_3}{\partial t}
+ v_1\left(\frac{\partial v_3}{\partial x}-\frac{\partial v_1}{\partial z}\right)
+ \frac{\partial}{\partial z}\left(\frac{v_1^2+v_3^2}{2}\right)
+ c_p \theta \frac{\partial \Pi}{\partial z}
+ \frac{\partial \phi}{\partial z}
=0,\\
&\frac{\partial (\varrho\theta)}{\partial t}
+ \frac{\partial (\varrho\theta v_1)}{\partial x}
+ \frac{\partial (\varrho\theta v_3)}{\partial z}
=0.
  \end{aligned}
  \label{eq:vector_invariant_form_2d}
\end{equation}
With the set of prognostic variables $\vec{u} = (\varrho, v_1, v_3, \varrho \theta)$, the system can be cast in the compact form
\begin{equation}
\partial_t \vec{u} + \sum_{j} \partial_{x_j}\vec{f}^j(\vec{u}) + \sum_{j} \vec{H}^j(\vec{u}) \partial_{x_j}\vec{g}^j(\vec{u}) = \vec{0},
\label{eq:nonconservative_form}
\end{equation}
where $\vec{f}^j$ represents the conservative fluxes and $\vec{H}^j \partial_{x_j} \vec{g}^j$ the nonconservative products, namely the vorticity contributions, the pressure term $c_p \theta \partial_{x_j} \Pi$ and the geopotential term $\partial_{x_j} \phi$, with the following definitions
\begin{equation}
    \label{eq:vector_invariant_form_fluxes}
 \vec{f}^{1} =
 \begin{pmatrix}
    \varrho v_1 \\
    K \\
    0 \\
    \varrho \theta v_1
\end{pmatrix}, \quad
\vec{f}^{3} =
\begin{pmatrix}
    \varrho v_3 \\
    0 \\
    K \\
    \varrho \theta v_3
\end{pmatrix},
 \end{equation}

 \begin{equation}
    \label{eq:vector_invariant_form_nonconservative}
\vec{H}^{1} =
 \begin{pmatrix}
    0    & 0          & 0 \\
    -v_3 & c_p \theta & 1 \\
    v_1  & 0          & 0 \\
    0    & 0          & 0
\end{pmatrix}, \quad
\vec{g}^{1} =
\begin{pmatrix}
v_3 \\
\Pi \\
\phi
\end{pmatrix}, \quad
\vec{H}^{3} =
\begin{pmatrix}
    0    & 0          & 0 \\
    v_3  & 0          & 0 \\
    -v_1 & c_p \theta & 1 \\
    0    & 0          & 0
\end{pmatrix}, \quad
 \vec{g}^{3} =
\begin{pmatrix}
v_1 \\
\Pi \\
\phi
\end{pmatrix},
\end{equation}
where $K = \frac{1}{2} \|\vec{V}\|^2$.

We introduce some standard two-point notation that we will use throughout the study:
\begin{equation}
\label{eq:properties_discrete_mean}
\begin{aligned}
	\jump{a} &= a_{+} - a_{-},                 &\qquad& (\text{jump})\\
	\mean{a} &= \tfrac{1}{2}(a_{+} + a_{-}),       &\qquad& (\text{arithmetic mean})\\
	\mean{a}_{\mathrm{log}} &= \jump{a}/\jump{\log a}, &\qquad& (\text{logarithmic mean, cf.~\cite{ismail2009affordable}})\\
	\mean{a}_\gamma &= \tfrac{\gamma-1}{\gamma}\jump{a^{\gamma}}/\jump{a^{\gamma-1}}, &\qquad& (\text{Stolarsky mean, cf.~\cite{Winters2020}})\\
    \mean{a}_{\mathrm{geo}} &= \sqrt{a_{+} a_-}, &\qquad& (\text{geometric mean}),
\end{aligned}
\end{equation}
with the discrete product rules $\jump{ab} = \mean{a}\jump{b} + \mean{b}\jump{a}$, $\jump{a^2} = 2\jump{a}\mean{a}$.

To derive the entropy-conservative/stable methods, we follow the approach presented in~\cite{artiano2026affordable}, which is based on a specialization of the fluctuation form of the nonconservative terms (see, e.g.,~\cite{Castro2013,RENAC20191}).
Among the forms proposed therein, we adopt the third, which is the most practical, as it offers enough flexibility to derive entropy-conservative fluxes for most nonconservative hyperbolic systems, while keeping the construction of novel entropy-stable/conservative schemes straightforward.
Following~\cite{artiano2026affordable}, a semidiscrete finite volume (FV) method for~\eqref{eq:nonconservative_form} on a uniform Cartesian grid reads
\begin{equation}
\partial_t \vec{u}_i
+ \sum_{j} \frac{1}{\Delta x_j}\left (\vec{f}^{\mathrm{num},j}_{i+1/2} - \vec{f}^{\mathrm{num},j}_{i-1/2} \right )
+ \sum_{j} \frac{1}{2 \Delta x_j} \left ( \vec{H}^{\mathrm{num},j}_{i+1/2}\jump{\vec{g}^j}_{i+1/2} + \vec{H}^{\mathrm{num},j}_{i-1/2}\jump{\vec{g}^j}_{i-1/2} \right ) = \vec{0}.
\label{eq:fv_nonconservative_form}
\end{equation}
The numerical counterparts $\vec{f}^{\mathrm{num},j}$, $\vec{H}^{\mathrm{num},j}$, and $\vec{g}^j$ follow naturally from~\eqref{eq:vector_invariant_form_fluxes} and~\eqref{eq:vector_invariant_form_nonconservative}, respectively.
Note that the gradient of the potential is treated as a nonconservative term, although it could be discretized in divergence form.
This choice is made purely for convenience, since it simplifies the derivation of the numerical fluxes.
Nevertheless, the derivation can be carried out equivalently by treating the geopotential as a conservative term with central numerical flux.

The scheme~\eqref{eq:fv_nonconservative_form} is uniquely determined once the numerical fluxes are specified.
Thus, we are looking for numerical fluxes that conserve the entropy functionals $U_{\varrho s}$ and $U_{\varrho E}$.
It is important to note that, when expressed in terms of the variables $(\varrho, v_1, v_3, \varrho \theta)$, these functionals lose strict convexity.
The thermodynamic entropy remains convex but not strictly~\cite{artiano2025structurepreservinghighordermethodscompressible}, whereas the total energy is not convex, since the Hessian is indefinite if the velocity is non-zero.
Indeed, considering the vector-invariant form in one dimension, the Hessians of the two mathematical entropies read
\begin{equation}
	 U_{\varrho E}''(\vec{u}) = \begin{bmatrix}
	    0        & v & 0           \\
		v                             & \varrho & 0                                         \\
		0 & 0 & \gamma R \frac{\Pi}{\varrho \theta}
	\end{bmatrix}, \quad  U_{\varrho s}''(\vec{u}) = \frac{\gamma}{\gamma-1}\begin{bmatrix}
		\frac{1}{\varrho}         & 0 & -\frac{1}{\varrho \theta}            \\
		0                              & 0 & 0                                         \\
		-\frac{1}{\varrho \theta} & 0 & \frac{ \varrho}{(\varrho \theta)^2}
	\end{bmatrix},
\end{equation}
with eigenvalues
\begin{equation}
         \lambda_{\varrho E, 1,2} = \frac{\varrho \pm \sqrt{\varrho^2 + 4 v^2}}{2}, \quad
         \lambda_{\varrho E, 3} = \gamma R \frac{\Pi}{\varrho \theta},
\end{equation}
\begin{equation}
	 \lambda_{\varrho s, 1} = \lambda_{\varrho s, 2} = 0, \quad \lambda_{\varrho s, 3} = \frac{\gamma}{\gamma-1}\left ( \frac{1}{\varrho} + \frac{\varrho}{(\varrho \theta)^2} \right).
\end{equation}
The loss of strict convexity implies that Barth's theorem~\cite{Barth1999} cannot be applied to uniquely determine a dissipation operator guaranteeing entropy stability~\cite{WARUSZEWSKI2022111507, WINTERS2017274}.
However, a generalized Tadmor condition remains both necessary and sufficient for entropy conservation/stability, regardless of whether a mathematical entropy is convex (see the algebraic characterization in~\cite{artiano2026affordable}), so that entropy-stable schemes can still be obtained by enforcing this condition.
Moreover, we show that it can be satisfied in a simple way for any discretization of the compressible Euler equations that uses the density and the potential temperature (in conservative form) as prognostic variables, independently of how the momentum equation is discretized.

For completeness, we report here the relevant result for the semidiscretization~\eqref{eq:fv_nonconservative_form}, which corresponds to the third form of~\cite{artiano2026affordable} and extends the classical result of Tadmor~\cite{Tadmor1987}.
\begin{theorem}[Theorem~3.10 of~\cite{artiano2026affordable}]
\label{th:tadmor_form3}
    The FV scheme~\eqref{eq:fv_nonconservative_form} is entropy-conservative/stable for a given entropy pair $(U, \vec{F})$ if and only if
\begin{equation}
\label{eq:TadmorDef_form3}
	\jump{\vec{\omega}} \cdot \vec{f}^{\mathrm{num},j} -\mean{\vec{\omega}} \cdot \vec{H}^{\mathrm{num},j}\jump{\vec{g}^j} \leq \jump{\psi^j}
\end{equation}
for all $j$, where $\vec{\omega} = U'$ are the entropy variables and $\psi^j = \vec{\omega} \cdot \vec{f}^j - F^j$ is the entropy potential.
\end{theorem}

In the next two sections, we follow an analysis similar to that of~\cite{artiano2025structurepreservinghighordermethodscompressible} to derive the total energy and the thermodynamic entropy numerical fluxes, using the generalized Tadmor condition~\eqref{eq:TadmorDef_form3}.

\subsection{Total energy conservation}
\label{sec:tec}
We first use the total energy $U_{\varrho E}$ as the mathematical entropy.
Since the analysis decouples in the two spatial directions, we carry out the derivations in the $x_1$-direction; the $x_3$-direction follows analogously.
The entropy variables and the entropy potential in the $x_1$-direction associated with $(\varrho, v_1, v_3, \varrho \theta)$ are
\begin{equation}
\vec{\omega}_{\varrho E} = \left ( K + \phi,\; \varrho v_1,\; \varrho v_3,\; c_p \Pi \right ), \qquad
\psi^1_{\varrho E} = \varrho K v_1, \qquad K = \tfrac{1}{2}(v_1^2 + v_3^2).
\label{eq:tec_entropy_variables}
\end{equation}
Inserting~\eqref{eq:tec_entropy_variables} into the condition~\eqref{eq:TadmorDef_form3} with equality, we obtain
\begin{multline}
\jump{K}f^{\varrho,1} + \jump{\phi}f^{\varrho,1} + \jump{\varrho v_1} f^{v_1,1} + c_p\jump{\Pi} f^{\varrho \theta,1}
\\
+\mean{\varrho v_1} v_3^{\mathrm{num},1} \jump{v_3} - \mean{\varrho v_3} v_1^{\mathrm{num},1} \jump{v_3}
-\mean{\varrho v_1}\left ( c_p\,\theta^{\mathrm{num},1} \jump{\Pi} + \jump{\phi} \right )
= \jump{\varrho v_1 K},
\label{eq:tec_condition}
\end{multline}
where $\theta^{\mathrm{num},1}$ denotes the discretization of $\theta$ in the nonconservative pressure product $c_p\theta\partial_x\Pi$, and the term $\mean{\varrho v_1}\jump{\phi}$ stems from the nonconservative geopotential gradient. Expanding the right-hand side with the discrete chain rule and matching the coefficients of the independent jumps $\jump{K}$, $\jump{\phi}$, $\jump{\varrho v_1}$, and $\jump{\Pi}$ yields the following result.
    \begin{theorem}
    \label{th:tec}
    A numerical flux for the compressible Euler equations in vector-invariant form~\eqref{eq:vector_invariant_form_2d} satisfying
    \begin{equation}
        \label{eq:tec_fluxes}
        \begin{aligned}
        f^{\varrho,j} &= \mean{\varrho v_j}, \\
        f^{v_j,j} &= \mean{K}, \\
        f^{\varrho \theta, j} &= \mean{\varrho v_j}\, \theta^{\mathrm{num}, j},\\
        v_i^{\mathrm{num},j} &= \frac{\mean{\varrho v_i}}{\mean{\varrho}},
        \end{aligned}
    \end{equation}
    where $\theta^{\mathrm{num},j}$ is any consistent discretization of $\theta$, is total-energy-conservative (TEC).
\end{theorem}
\begin{remark}
  Note that the condition to derive a TEC flux for $v_i^{\mathrm{num},j}$ in the $x_1$-direction decouples into
  \begin{equation}
  \mean{\varrho v_1} v_3^{\mathrm{num}} = \mean{\varrho v_3} v_1^{\mathrm{num}},
  \end{equation}
  which can be satisfied in many different ways, e.g., by $v_i^{\mathrm{num},j} = \mean{\varrho v_i} / \varrho^{\mathrm{num}}$, where $\varrho^{\mathrm{num}}$ is any consistent discretization of the density.
  We adopt the simplest choice $\varrho^{\mathrm{num}} = \mean{\varrho}$, since we found no reason to introduce a more expensive mean value.
\end{remark}

\begin{remark}
The scheme~\eqref{eq:tec_fluxes} uses a different momentum discretization from the energy-conservative scheme proposed by Ricardo et al.~\cite{ricardo2024thermodynamicconsistencystructurepreservationsummation}:
here the skew-symmetric vorticity term is discretized by means of nonconservative split forms~\cite{artiano2026affordable}, and the scheme of Ricardo et al.\ is recovered when the vorticity term $ \left ( \nabla \times \vec{V} \right ) \times \vec{V}$ is discretized pointwise (i.e., evaluated at the right interface as $\frac{\vec{H}_i}{2} \jump{\vec{g}}_{i+1/2}$) and the potential temperature numerical flux is the central one.
The freedom to choose the potential-temperature numerical flux in the present generalization, in contrast to the work of~\cite{ricardo2024thermodynamicconsistencystructurepreservationsummation}, plays a key role when deriving a method that conserves the entropy- as well as the total energy and guarantees preservation of isothermal and isentropic background states, as shown later in Section~\ref{sec:well_balanced}.
\end{remark}

\subsection{Thermodynamic entropy conservation}
\label{sec:ec}
We now look for numerical fluxes that conserve the thermodynamic entropy $U_{\varrho s}$. The associated entropy variables are
\begin{equation}
\vec{\omega}_{\varrho s} = \left (-\frac{\log(p/\varrho^\gamma) - \gamma}{\gamma-1},\; 0,\; 0,\; - \frac{\gamma}{\gamma-1} \frac{\varrho}{\varrho \theta} \right ),
\label{eq:ec_entropy_variables}
\end{equation}
and the entropy potential vanishes, $\psi^j_{\varrho s} = 0$.
The velocity components of the entropy variables $\vec{\omega}_{\varrho s}$ vanish and the mass and potential-temperature equations coincide with those of the conservative formulation~\cite{artiano2025structurepreservinghighordermethodscompressible}.
Thus, the condition~\eqref{eq:TadmorDef_form3} reduces to the same entropy conservation condition as in~\cite[Eq.~(26)]{artiano2025structurepreservinghighordermethodscompressible}, which reads
\begin{equation}
\jump{\log \left (1/\theta \right )} f^{\varrho,j} - \jump{1/\theta}\, f^{\varrho \theta,j} = 0.
\label{eq:ec_condition}
\end{equation}
The condition~\eqref{eq:ec_condition} involves only the density and potential-temperature fluxes, and the density flux remains a degree of freedom. Solving for the potential-temperature flux to avoid the influence of the pressure in the density flux (see~\cite{artiano2025structurepreservinghighordermethodscompressible}) gives the following result.
\begin{theorem}
\label{th:ec}
A numerical flux for the compressible Euler equations in vector-invariant form~\eqref{eq:vector_invariant_form_2d} satisfying
\begin{equation}
    \begin{aligned}
\label{eq:ec_fluxes}
f^{\varrho,j} &= \text{any consistent flux of } \varrho v_j,\\
f^{\varrho \theta,j} &= \frac{1}{\mean{1/\theta}_{\mathrm{log}}}\, f^{\varrho,j},
    \end{aligned}
\end{equation}
is entropy-conservative (EC).
\end{theorem}

\begin{remark}
This result coincides with the one obtained in~\cite{artiano2025structurepreservinghighordermethodscompressible}. Clearly, the choice of formulation and of momentum discretization plays no role, and the result holds equally for the conservative variables with momentum $\varrho v_j$.
\end{remark}

The degrees of freedom in the EC and TEC numerical fluxes can be combined to construct a numerical flux that conserves the total energy and entropy. A full characterization of such a numerical flux is given in Section~\ref{sec:well_balanced}, where we prove well-balancedness properties.

\section{Entropy stability}
\label{sec:entropy_stability}
Usually, entropy-stable methods are obtained by adding proper dissipation to an EC flux, so that~\eqref{eq:ec_condition} is satisfied as an inequality.
Here, we propose a simple modification of the EC flux~\eqref{eq:ec_fluxes} that yields an entropy-stable flux and also guarantees positivity of the pressure, as shown in Section~\ref{sec:positivity}.

\begin{theorem}
\label{th:es}
A numerical flux for the compressible Euler equations in vector-invariant form~\eqref{eq:vector_invariant_form_2d} satisfying
\begin{equation}
    \begin{aligned}
f^{\varrho\theta,j} &= f^{\varrho,j}\,\theta^{\mathrm{num}, j}, \\
\theta^{\mathrm{num},j} &= \frac{1}{\mean{1/\theta}_{\mathrm{log}}} - \frac{\theta^{\mathrm{min}}}{2\,\mean{\theta}}\jump{\theta}\,\operatorname{sign}\!\left(f^{\varrho,j}\right),
\label{eq:es_flux}
    \end{aligned}
\end{equation}
where $f^{\varrho,j}$ is any consistent numerical flux of the density, is entropy-stable (ES). Moreover, $\theta^{\mathrm{num},j} \in~[\theta_{\min},\theta_{\max}]$ with $\theta^{\max}_{i+1/2} = \max(\theta_i,\theta_{i+1})$ and $\theta^{\mathrm{min}}_{i+1/2} = \min(\theta_i,\theta_{i+1})$.
\end{theorem}
\begin{proof}
The entropy stability condition~\eqref{eq:ec_condition} becomes
\begin{equation}
\jump{\log \left (1/\theta \right )} f^{\varrho,j} - \jump{1/\theta}\, f^{\varrho \theta,j} \leq 0.
\end{equation}
Substituting the entropy-stable flux~\eqref{eq:es_flux} and using $\mean{1/\theta}_{\mathrm{log}} = \jump{1/\theta}/\jump{\log(1/\theta)}$, we obtain
\begin{multline}
f^{\varrho, j} \left ( \jump{\log(1/\theta)} - \jump{1/\theta}\frac{1}{\mean{1/\theta}_{\mathrm{log}}} + \jump{1/\theta}\frac{\theta^{\mathrm{min}}}{2\,\mean{\theta}}\jump{\theta}\,\operatorname{sign}\!\left(f^{\varrho,j}\right) \right )
\\
=
- \frac{\theta^{\mathrm{min}}}{2\,\mean{\theta}}\frac{\jump{\theta}^2}{\mean{\theta}_{\mathrm{geo}}^2} | f^{\varrho,j} | \leq 0,
\end{multline}
where we used
\begin{equation}
\jump{\frac{1}{\theta}} = - \frac{\jump{\theta}}{\mean{\theta}^2_{\mathrm{geo}}}.
\end{equation}
Now we show that the numerical flux is bounded by the minimum and maximum values at the interface.
First, we notice that the inverse of the logarithmic mean of $1/\theta$ is bounded by the harmonic and geometric means~\cite{MeansGenerated}
\begin{equation}
    \label{eq:log_bounds}
\frac{2 \theta^{\mathrm{min}} \theta^{\mathrm{max}}}{\theta^{\mathrm{max}} + \theta^{\mathrm{min}}} \leq \frac{1}{\mean{1/ \theta}_{\mathrm{log}}} \leq \sqrt{\theta^{\mathrm{min}} \theta^{\mathrm{max}}}.
\end{equation}
Moreover, we can rewrite $\theta^{\mathrm{num}} = \frac{1}{\mean{1/\theta}_{\mathrm{log}}} - S$, where $S = \frac{\theta^{\mathrm{min}}}{2 \mean{\theta}} \jump{\theta} \mathrm{sign}(f^{\varrho,j})$ and therefore
\begin{equation}
\frac{1}{\mean{1/\theta}_{\mathrm{log}}} - |S| \leq \theta^{\mathrm{num}} \leq \frac{1}{\mean{1/\theta}_{\mathrm{log}}} + |S|,
\end{equation}
where
\begin{equation}
|S| = \frac{\theta^{\mathrm{min}} ( \theta^{\mathrm{max}} - \theta^{\mathrm{min}})}{\theta^{\mathrm{max}} + \theta^{\mathrm{min}}} \geq 0.
\end{equation}
For the lower bound, using~\eqref{eq:log_bounds} we obtain
\begin{equation}
\frac{1}{\mean{1/\theta}_{\mathrm{log}}} - |S| \geq \frac{2 \theta^{\mathrm{min}} \theta^{\mathrm{max}}}{\theta^{\mathrm{max}} + \theta^{\mathrm{min}}} - \frac{\theta^{\mathrm{min}} ( \theta^{\mathrm{max}} - \theta^{\mathrm{min}})}{\theta^{\mathrm{max}} + \theta^{\mathrm{min}}} = \theta^{\mathrm{min}}.
\end{equation}
For the upper bound, we have
\begin{equation}
\frac{1}{\mean{1/\theta}_{\mathrm{log}}} + |S| \leq \sqrt{\theta^{\mathrm{max}} \theta^{\mathrm{min}}} + \frac{\theta^{\mathrm{min}} ( \theta^{\mathrm{max}} - \theta^{\mathrm{min}})}{\theta^{\mathrm{max}} + \theta^{\mathrm{min}}}.
\end{equation}
We have to prove the following
\begin{equation}
\sqrt{\theta^{\mathrm{max}} \theta^{\mathrm{min}}} +  \frac{\theta^{\mathrm{min}} ( \theta^{\mathrm{max}} - \theta^{\mathrm{min}})}{\theta^{\mathrm{max}} + \theta^{\mathrm{min}}} \leq \theta^{\mathrm{max}},
\end{equation}
which is equivalent to proving the following
\begin{equation}
\sqrt{a_{i+1} a_i} \leq \frac{a_{i+1}^2 + a_i^2}{a_{i+1} + a_i} = 2\mean{a} - \mean{a}_{\mathrm{har}},
\end{equation}
where
\begin{equation}
\mean{a}_{\mathrm{har}} = 2 \frac{a_i a_{i+1}}{a_{i+1}+a_i}
\end{equation}
is the harmonic mean.
Therefore, the proof boils down to the following inequality
\begin{equation}
\mean{a}_{\mathrm{geo}} \leq 2 \mean{a} - \mean{a}_{\mathrm{har}},
\end{equation}
which is satisfied by the strictly increasing monotonicity of the mean values generated by the integral~\cite{MeansGenerated}.
\end{proof}

\begin{remark}
    The simpler numerical flux for the potential temperature $\theta$
    \begin{equation}
        \theta^{\mathrm{num}} = \frac{1}{\mean{1/\theta}_{\mathrm{log}}} - \frac{\jump{\theta}}{2}\operatorname{sign}\!\left(f^{\varrho,j}\right)
    \end{equation}
    is also entropy-stable, but it does not guarantee that $\theta^{\mathrm{num}}$ remains bounded in the interval $[\theta^{\mathrm{min}}, \theta^{\mathrm{max}}]$.
    \end{remark}
\begin{remark}
The result above shows that, whenever the density and the potential temperature (in conservative form) are chosen as prognostic variables, entropy stability can be achieved simply by upwinding the potential temperature flux according to the inverse logarithmic mean law, independently of whether the conservative or the vector-invariant form is used.
A concrete choice of the density flux, which constitutes a degree of freedom here as it does for the entropy-conservative flux, is given in Section~\ref{sec:well_balanced}.
\end{remark}

\section{On the positivity of the density and pressure}
\label{sec:positivity}
In this section, we investigate the positivity of the pressure and density for a fully-discrete explicit scheme.
Although the main focus of the present study is on the spatial discretization, operational atmospheric models typically require (vertically) implicit time integrators (see, e.g.,~\cite{taylor2019,BALDAUF2021110635,giraldo2013numa,GIRALDO2024113275}).
The corresponding positivity analysis for vertically implicit schemes is left for future work.

Density and pressure must remain positive to avoid nonphysical states and instabilities~\cite{DERIGS2017624}. Since SSP Runge--Kutta methods are convex combinations of explicit Euler steps~\cite{Gottliebs2011} and our set of interest is also convex, it suffices to study the positivity of a single explicit Euler step applied to the first-order FV scheme.
The analysis can be extended to DG methods with similar arguments (see~\cite{ZHANG20108918,zhang2011maximum}).

We consider the convex set of admissible states
\begin{equation}
G = \left \{ \vec{u} = (\varrho, v, \varrho \theta) \;\middle|\; \varrho > 0,\; p = p_0 \left (\frac{\varrho \theta R}{p_0} \right )^\gamma > 0 \right \},
\end{equation}
where, for simplicity, we have considered the one-dimensional case, since the velocity components do not play a role in the positivity analysis of the density and the pressure.
\begin{theorem}
\label{th:positivity}
  Assume that all $\vec{u}_i^n \in G$.
  Consider an explicit Euler discretization of the FV scheme~\eqref{eq:fv_nonconservative_form} with flux $f^{\varrho \theta} = f^{\varrho} \theta^{\mathrm{num}}$, where $f^{\varrho} = \left (\mean{\varrho} -\frac{\jump{\varrho}}{2} \mathrm{sign}(v^{\mathrm{num}}) \right ) v^{\mathrm{num}}$, $v^{\mathrm{num}}$ is a consistent interface velocity, and $\theta^{\mathrm{num}} \in [\theta^{\min},\theta^{\max}]$ with $\theta^{\min} = \min(\theta_i,\theta_{i+1})$, and $\theta^{\max} = \max(\theta_i,\theta_{i+1})$.
  Then, the resulting scheme preserves the positivity of $\varrho\theta$, and hence of the pressure, under the Courant--Friedrichs--Lewy (CFL) condition
  \begin{equation}
      \Delta t < \frac{\Delta x}{2 \lambda}, \qquad
      \lambda \geq \max\left( v^{\mathrm{num},+}_{i+1/2}\,\frac{\theta^{\mathrm{max}}_{i+1/2}}{\theta_i},\;
      |v^{\mathrm{num},-}_{i-1/2}|\,\frac{\theta^{\mathrm{max}}_{i-1/2}}{\theta_i} \right),
      \label{eq:positivity_cfl}
  \end{equation}
  where $v^{\mathrm{num},+} = \max(v^{\mathrm{num}}, 0)$ and $v^{\mathrm{num},-} = \min(v^{\mathrm{num}}, 0)$ denote the positive and negative parts of the interface velocity.
\end{theorem}
\begin{proof}
The explicit Euler step of the first-order FV scheme for $\varrho\theta$ in the
$x$-direction reads
\begin{equation}
(\varrho\theta)_i^{n+1}
= (\varrho\theta)_i^{n}
- \frac{\Delta t}{\Delta x}\left( f^{\varrho\theta}_{i+1/2} - f^{\varrho\theta}_{i-1/2}\right).
\end{equation}
With the upwind density flux
$f^{\varrho}_{i+1/2} = v^{\mathrm{num},+}_{i+1/2}\varrho_i + v^{\mathrm{num},-}_{i+1/2}\varrho_{i+1}$, the update can be split into three contributions
\begin{equation}
\begin{aligned}
(\varrho\theta)_i^{n+1}
= \varrho_i\theta_i
&- \frac{\Delta t}{\Delta x}\,\varrho_i
\left( v^{\mathrm{num},+}_{i+1/2}\,\theta^{\mathrm{num}}_{i+1/2}
- v^{\mathrm{num},-}_{i-1/2}\,\theta^{\mathrm{num}}_{i-1/2}\right)\\
&- \frac{\Delta t}{\Delta x}
\left( v^{\mathrm{num},-}_{i+1/2}\,\varrho_{i+1}\,\theta^{\mathrm{num}}_{i+1/2}
- v^{\mathrm{num},+}_{i-1/2}\,\varrho_{i-1}\,\theta^{\mathrm{num}}_{i-1/2}\right),
\end{aligned}
\end{equation}
where we dropped the superscript $\mathrm{num}$ on the velocity for simplicity.
Collecting the different contributions yields
\begin{equation}
\begin{aligned}
(\varrho\theta)_i^{n+1}
= \underbrace{\varrho_i\theta_i
\left( 1 - \frac{\Delta t}{\Delta x}
\left[ v^{\mathrm{num},+}_{i+1/2}\,\frac{\theta^{\mathrm{num}}_{i+1/2}}{\theta_i}
+ |v^{\mathrm{num},-}_{i-1/2}|\,\frac{\theta^{\mathrm{num}}_{i-1/2}}{\theta_i}\right]\right)}_{(\mathrm I)}
&+ \underbrace{\frac{\Delta t}{\Delta x}\,|v^{\mathrm{num},-}_{i+1/2}|\,\varrho_{i+1}\,\theta^{\mathrm{num}}_{i+1/2}}_{(\mathrm{II})}\\
&+ \underbrace{\frac{\Delta t}{\Delta x}\,v^{\mathrm{num},+}_{i-1/2}\,\varrho_{i-1}\,\theta^{\mathrm{num}}_{i-1/2}}_{(\mathrm{III})}.
\end{aligned}
\end{equation}
We show that the first term is strictly positive and that the remaining two terms are non-negative.

The second and third terms are products of non-negative factors. Indeed, the velocity parts satisfy $|v^{\mathrm{num},-}_{i+1/2}| \ge 0$ and $v^{\mathrm{num},+}_{i-1/2}\ge 0$, the densities $\varrho_{i\pm 1} > 0$ by assumption, while
$\theta^{\mathrm{num}}_{i\pm 1/2}\in[\theta^{\min},\theta^{\max}]$ holds by construction.

The first term is the product of $\varrho_i\theta_i > 0$ with the factor in parentheses, which is non-negative under the CFL condition~\eqref{eq:positivity_cfl}:
\begin{equation}
\frac{\Delta t}{\Delta x}
\left( v^{\mathrm{num},+}_{i+1/2}\,\frac{\theta^{\mathrm{num}}_{i+1/2}}{\theta_i}
+ |v^{\mathrm{num},-}_{i-1/2}|\frac{\theta^{\mathrm{num}}_{i-1/2}}{\theta_i}\right)
\le \frac{\Delta t}{\Delta x}2\lambda < 1.
\end{equation}
Since $p = p_0(R \varrho\theta / p_0)^\gamma$, the positivity of $\varrho\theta$ implies $p_i^{n+1} > 0$ for all $n \in \mathbb{N}$.
\end{proof}

The positivity of the density follows from the same argument with $\theta_i = 1$; therefore, we state only the result.
\begin{corollary}
\label{th:positivity_density}
  Assume that all $\vec{u}_i^n \in G$.
  Consider an explicit Euler discretization of the FV scheme~\eqref{eq:fv_nonconservative_form} with flux $f^{\varrho} = \left (\mean{\varrho} -\frac{\jump{\varrho}}{2} \mathrm{sign}(v^{\mathrm{num}}) \right ) v^{\mathrm{num}}$, where $v^{\mathrm{num}}$ is a consistent interface velocity.
  Then, the resulting scheme preserves the positivity of $\varrho$ under the CFL condition
\begin{equation}
    \Delta t < \frac{\Delta x}{2 \lambda}, \qquad
    \lambda = \max\left( v^{\mathrm{num},+}_{i+1/2},\;
    |v^{\mathrm{num},-}_{i-1/2}|\right ).
    \label{eq:positivity_cfl_density}
\end{equation}
\end{corollary}

\section{Well-balanced schemes}
\label{sec:well_balanced}
Hydrostatically balanced atmospheres at rest are steady states of~\eqref{eq:vector_invariant_form} characterized by $\vec{V} = \vec{0}$, which satisfy
\begin{equation}
c_p \theta \nabla \Pi + \nabla \phi = 0.
\label{eq:hydrostatic_balance}
\end{equation}
A scheme is well-balanced if it preserves~\eqref{eq:hydrostatic_balance} exactly at the discrete level. Therefore, a consistent and well-balanced scheme must produce no spurious contributions in the density and potential temperature equations at the steady state and must reduce, in the momentum equations, to a discrete analog of~\eqref{eq:hydrostatic_balance}. Indeed, for $\vec{V}=\vec{0}$ and consistent (and symmetric) numerical fluxes, the gradient of the kinetic energy $\mean{K}$ and the vorticity terms vanish, and the FV scheme~\eqref{eq:fv_nonconservative_form} reduces to
\begin{equation}
    \frac{c_p}{2\Delta x} \Bigl (
    \theta^{\mathrm{num}}_{i+1/2}\jump{\Pi}_{i+1/2} + \theta^{\mathrm{num}}_{i-1/2}\jump{\Pi}_{i-1/2}
    \Bigr ) + \frac{1}{2\Delta x} \Bigl (
    \jump{\phi}_{i+1/2} + \jump{\phi}_{i-1/2}
    \Bigr ) = 0.
    \label{eq:wb_semidiscrete}
\end{equation}
Equation~\eqref{eq:wb_semidiscrete} is a specific form of the well-balancedness condition for nonconservative hyperbolic systems, which has been derived in~\cite{artiano2026affordable}.

\begin{theorem}[Theorem~6.5 of~\cite{artiano2026affordable}]
\label{well_balanced_fluxes}
A finite volume scheme~\eqref{eq:fv_nonconservative_form} with symmetric numerical fluxes $\vec{f}^{\mathrm{num}}$ and $\vec{H}^{\mathrm{num}}$ satisfies the well-balancedness condition
\begin{equation}
\frac{1}{\Delta x_i} \left  ( \vec{f}^{\mathrm{num}}_{i+1/2} - \vec{f}^{\mathrm{num}}_{i-1/2}\right ) + \frac{\vec{H}^{\mathrm{num}}_{i+1/2} \jump{\vec{g}}_{i+1/2} + \vec{H}^{\mathrm{num}}_{i-1/2}  \jump{\vec{g}}_{i-1/2}}{2 \Delta x_i}  = \vec{0},
\end{equation}
if and only if
\begin{equation}
    \vec{f}^{\mathrm{num}} (\vec{u}_i, \vec{u}_{i+1})
    + \frac{1}{2}\vec{H}^{\mathrm{num}}(\vec{u}_i, \vec{u}_{i+1})
    \bigl(\vec{g}(\vec{u}_{i+1}) - \vec{g}(\vec{u}_{i})\bigr)
    = \vec{f}(\vec{u}_i).
\end{equation}
\end{theorem}

Assuming that the discrete kinetic energy flux and the vorticity terms vanish at the steady state, the necessary and sufficient condition can be recast as follows.
\begin{lemma}
    \label{lemma:wb}
    Given a hydrostatic balance state, the vector-invariant form semidiscretization~\eqref{eq:fv_nonconservative_form} is well-balanced if
\begin{equation}
        \label{eq:wb_vector_invariant}
\begin{aligned}
    & f^{\varrho} = 0,\\
    & f^{v} = 0,\\
    & v^{\mathrm{num}} = 0,\\
    & c_p \theta^{\mathrm{num}}\jump{\Pi} + \jump{\phi}  = 0,\\
    & f^{\varrho \theta} = 0.
\end{aligned}
\end{equation}
\end{lemma}

Therefore, we exploit the only remaining degree of freedom, the choice of $\theta^{\mathrm{num}}$, to enforce well-balancedness. We now analyze isothermal and isentropic background states, and summarize the results in the following two lemmas.

\begin{lemma}
\label{lemma:wb_temperature}
    A numerical flux preserves the hydrostatic balance~\eqref{eq:hydrostatic_balance} prescribed by a constant background temperature $T$ if
    \begin{equation}
    \theta^{\mathrm{num}} = \frac{1}{\mean{1/\theta}_{\mathrm{log}}}.
    \label{eq:wb_temperature_flux}
    \end{equation}
\end{lemma}
\begin{proof}
For a constant temperature, $T = \theta\,\Pi = \text{const}$ and thus $\theta = T/\Pi$, so we can integrate the hydrostatic balance exactly:
\begin{equation}
c_p \theta \partial_x \Pi = - \partial_x \phi,
\end{equation}
which, upon substituting $\Pi = T/\theta$, can be written as
\begin{equation}
-\frac{1}{\theta} \partial_x \theta = - \frac{1}{c_p T} \partial_x \phi.
\end{equation}
A solution must be of the form
\begin{equation}
  -\log(\theta) + \log(\theta_0) = \frac{1}{c_p T}\left (-\phi + \phi_0 \right ),
\end{equation}
which implies that, for any two points, we have
\begin{equation}
\label{eq:condition_isothermal}
    c_p T \jump{\log \theta} = \jump{\phi}.
\end{equation}
The condition~\eqref{eq:wb_vector_invariant} at the isothermal background state, using~\eqref{eq:condition_isothermal}, reduces to
\begin{equation}
c_p T \theta^{\mathrm{num}}\jump{\frac{1}{\theta}} + c_p T \jump{\log \theta} = 0,
\end{equation}
which is satisfied if $\theta^{\mathrm{num}} = \frac{1}{\mean{1/\theta}_{\mathrm{log}}}$.
\end{proof}

\begin{lemma}
\label{lemma:wb_theta}
    A numerical flux preserves the hydrostatic balance~\eqref{eq:hydrostatic_balance} prescribed by a constant background potential temperature $\theta$ if
    \begin{equation}
    \theta^{\mathrm{num}} = \text{any consistent two-point function of $\theta_\pm$.}
    \label{eq:wb_theta_flux}
    \end{equation}
\end{lemma}
\begin{proof}
  For a constant potential temperature, the hydrostatic balance law~\eqref{eq:hydrostatic_balance} has the closed-form solution
  \begin{equation}
  \Pi = \Pi_0 - \frac{\phi}{c_p \theta}.
  \end{equation}
  Thus, substituting in~\eqref{eq:wb_vector_invariant} and using the consistency of the numerical flux $\theta^{\mathrm{num}}$, we obtain
  \begin{equation}
  \jump{\Pi} = - \frac{1}{c_p \theta}\jump{\phi}.
  \end{equation}
  The condition becomes
  \begin{equation}
  c_p \theta \jump{\Pi} -c_p \theta \jump{\Pi} = 0,
  \end{equation}
  which proves the claim.
\end{proof}

\begin{remark}
We highlight that the present result applies also to the scheme of Ricardo et al.~\cite{ricardo2024thermodynamicconsistencystructurepreservationsummation};
their potential-temperature numerical flux is the central one and therefore consistent, preserving the isentropic background state exactly.
\end{remark}

Since preservation of the isentropic background state is guaranteed by any consistent numerical flux for the potential temperature in the velocity component, and since the flux that enforces the isothermal balance coincides with the entropy-conservative one, a single numerical flux can be constructed to preserve both background states exactly while being total-energy- and entropy-conservative.
We summarize this result in the following theorem.
\begin{theorem}
\label{th:ec_tec_wb}
The numerical flux for the compressible Euler equations in vector-invariant form~\eqref{eq:vector_invariant_form_2d} given by
\begin{equation}
    \begin{aligned}
\label{eq:tec_ec_wb_fluxes}
    f^{\varrho,j} &= \mean{\varrho v_j}, \\
     f^{v_j,j} &= \mean{K}, \\
    f^{\varrho \theta, j} &= \mean{\varrho v_j}\, \theta^{\mathrm{num}},\\
    v_i^{\mathrm{num}} &= \frac{\mean{\varrho v_i}}{\mean{\varrho}},\\
    \theta^{\mathrm{num}} &=  \frac{1}{\mean{1/\theta}_{\mathrm{log}}}
    \end{aligned}
\end{equation}
is TEC, EC, and well-balanced for isothermal and isentropic background states.
\end{theorem}
\begin{proof}
It suffices to show that Lemma~\ref{lemma:wb} is satisfied, as well as the EC and TEC conditions in Theorems~\ref{th:ec} and~\ref{th:tec}.
\end{proof}

\begin{remark}
  We stress that the combination of all the properties achieved in Theorem~\ref{th:ec_tec_wb}, namely EC, TEC and simultaneous well-balancedness for isothermal and isentropic background states, is not possible to achieve with the conservative formulation of the Euler equations with potential temperature (see~\cite{artiano2025structurepreservinghighordermethodscompressible}).
\end{remark}

The numerical flux~\eqref{eq:tec_ec_wb_fluxes} represents the volume numerical flux that we employ in the volume kernel discretization of the DGSEM flux-differencing method described in Section~\ref{sec:dgsem}.
A characterization of a well-balanced and entropy-stable numerical flux is presented in the next section.

\section{An entropy-stable and well-balanced scheme}
\label{sec:riemann_solver}
So far, only the entropy- and total-energy-conservative fluxes have been fully specified, while the entropy stability and well-balancedness conditions leave some degrees of freedom.
We now fix these degrees of freedom and present the numerical flux used at the interfaces in the simulations, which guarantees entropy stability and well-balancedness.

\begin{theorem}
\label{th:es_wb}
The numerical flux for the compressible Euler equations in vector-invariant form~\eqref{eq:vector_invariant_form_2d} given by
\begin{equation}
    \begin{aligned}
\label{eq:es_tec_wb_fluxes}
    f^{\varrho,j} &= \left (\mean{\varrho} - \frac{\jump{\varrho}}{2} \mathrm{sign}(V^{\mathrm{num},j}) \right ) V^{\mathrm{num},j}, \\
      f^{v_i,j} & = \delta_{ij} \left ( \mean{K} - \frac{\mean{c}}{2 \mean{\varrho}} \jump{\varrho v_j} \right ) - \frac{\Bigl | \mean{v_j} \Bigr|}{2 \mean{\varrho}} \jump{\varrho v_i}, \\
    f^{\varrho \theta, j} &= f^{\varrho,j} \left ( \theta^{\mathrm{num}} - \frac{\theta^{\mathrm{min}}}{2 \mean{\theta}} \jump{\theta} \mathrm{sign}(V^{\mathrm{num},j}) \right ), \\
    v_i^{\mathrm{num}} &= \frac{\mean{\varrho v_i}}{\mean{\varrho}},\\
    V^{\mathrm{num},j} &= \mean{v_j} - c_p\theta^{\mathrm{num}}\frac{\jump{\Pi}}{2 \mean{c}} - \frac{\jump{\phi}}{2\mean{c}},\\
    \theta^{\mathrm{num}} &= \frac{1}{\mean{1/\theta}_{\mathrm{log}}},
    \end{aligned}
\end{equation}
where $c = \sqrt{\gamma R T}$ is the speed of sound, is entropy-stable and well-balanced for the isothermal and isentropic background states.
\end{theorem}
\begin{proof}
  For the well-balancedness property, it suffices to show that, for $\vec{V} = \vec{0}$, only the discrete hydrostatic balance terms~\eqref{eq:wb_semidiscrete} remain in the scheme~\eqref{eq:fv_nonconservative_form}, since $V^{\mathrm{num},j} = 0$, and that $\theta^{\mathrm{num}}$ satisfies the assumptions of Lemmas~\ref{lemma:wb_temperature} and~\ref{lemma:wb_theta}.
  Entropy stability follows from Theorem~\ref{th:es}.
\end{proof}

We further observe that this entropy-stable Riemann solver can also be used for the compressible Euler equations in conservative form, which read
\begin{equation}
  \label{eq:euler_conservative}
\begin{aligned}
&\partial_t \varrho + \nabla \cdot ( \varrho \vec{V}) = 0,\\
&\partial_t (\varrho \vec{V} )+ \nabla \cdot (\varrho \vec{V} \otimes \vec{V}) + \nabla p + \varrho \nabla \phi = \vec{0}, \\
&\partial_t (\varrho \theta) + \nabla \cdot (\varrho \theta \vec{V}) = 0,
\end{aligned}
\end{equation}
by changing the interface velocity as follows
\begin{equation}
V^{\mathrm{num},j} = \mean{v_j} - \frac{\jump{p}}{2 \mean{c} }\frac{1}{\varrho^\mathrm{num}} - \frac{\jump{\phi}}{2\mean{c}},
\end{equation}
which comes from solving the Rankine--Hugoniot condition in the incompressible limit~\cite{AControlVolumeModeloftheCompressibleEulerEquationswithaVerticalLagrangianCoordinate}; the geopotential term is added to guarantee well-balancedness.
The choice of $\varrho^{\mathrm{num}}$ is given by the Stolarsky or the logarithmic mean, depending on whether one wishes to preserve the isentropic or isothermal well-balanced property, respectively (see~\cite{artiano2025structurepreservinghighordermethodscompressible,WARUSZEWSKI2022111507}).
This stems from the fact that the momentum components do not enter the entropy stability analysis, and therefore any consistent discretization of the momentum equation may be used.
We summarize the result in the following theorem.
\begin{theorem}
\label{th:es_wb_cons}
The numerical flux for the compressible Euler equations in conservative form~\eqref{eq:euler_conservative} given by
\begin{equation}
    \begin{aligned}
\label{eq:es_wb_cons_fluxes}
    f^{\varrho,j} &= \left (\mean{\varrho} - \frac{\jump{\varrho}}{2} \mathrm{sign}(V^{\mathrm{num},j}) \right ) V^{\mathrm{num},j}, \\
     f^{\varrho v_i,j} &= \left (\mean{\varrho v_i} - \frac{\jump{\varrho v_i}}{2} \mathrm{sign}(V^{\mathrm{num},j}) \right ) V^{\mathrm{num},j} + \delta_{ij}P^{\mathrm{num},j}, \\
    f^{\varrho \theta, j} &= f^{\varrho,j} \left ( \theta^{\mathrm{num}} - \frac{\theta^{\mathrm{min}}}{2 \mean{\theta}} \jump{\theta} \mathrm{sign}(V^{\mathrm{num},j}) \right ), \\
    V^{\mathrm{num},j} &= \mean{v_j} - \frac{\jump{p}}{2 \mean{c} }\frac{1}{\varrho^\mathrm{num}} - \frac{\jump{\phi}}{2\mean{c}},\\
    P^{\mathrm{num},j} &= \mean{p} - \frac{1}{2} \mean{c} \mean{\varrho} \jump{v_j},\\
    \theta^{\mathrm{num}} &= \frac{1}{\mean{1/\theta}_{\mathrm{log}}},
    \end{aligned}
\end{equation}
where $c = \sqrt{\gamma R T}$ is the speed of sound, is entropy-stable and well-balanced for isothermal background states if $\varrho^{\mathrm{num}} = \mean{\varrho}_{\mathrm{log}}$ and for isentropic background states if $\varrho^{\mathrm{num}} = \mean{\varrho}_{\gamma}$.
\end{theorem}
\begin{proof}
The proof follows analogous steps to Theorem~\ref{th:es_wb}.
Entropy stability follows from Theorem~\ref{th:es}, since the momentum components do not enter the entropy stability analysis.
For the well-balanced property, it suffices to show that, for $\vec{V} = \vec{0}$, $V^{\mathrm{num},j} = 0$ and $P^{\mathrm{num},j} = \mean{p}$, so that, discretizing the geopotential term as $\varrho^{\mathrm{num}} \jump{\phi}$, only the discrete hydrostatic balance terms of Lemmas~4.4 and~4.5 of~\cite{artiano2025structurepreservinghighordermethodscompressible} remain for the isothermal and isentropic background states, respectively.
\end{proof}

\section{Extension to discontinuous Galerkin spectral element methods and curvilinear meshes}
\label{sec:dgsem}
The two-point numerical fluxes previously constructed are the building
blocks of a high-order flux-differencing DGSEM.
Since the entropy stability and well-balancedness analyses rely solely on the summation-by-parts (SBP) property, they carry over to any diagonal-norm SBP discretization (see~\cite{Carpenter2014,RENAC20191,WARUSZEWSKI2022111507}).
As for the finite volume scheme, we employ the specific extension of the fluctuation form to high-order SBP operators shown in~\cite{artiano2026affordable}, from which the structure-preserving properties follow directly.

On each $d$-dimensional element $E$, the solution is approximated by Lagrange polynomials $\{l_k\}$ of degree $p$ on the Gauss--Lobatto--Legendre nodes.
The corresponding quadrature defines the diagonal mass matrix $\mtx{M}$ and the derivative operators $\mtx{D}^{j}$, $j\in\{1,\dots,d\}$, with $(\mtx{D}^{j})_{ik}=l_k'(\xi^j_i)$ on the reference element $[-1,1]^d$.
Together with the boundary restriction $\mtx{R}$, the boundary mass matrix $\mtx{B}$, and the normal operators $\mtx{N}^{j}$ that contain the signs of the outward reference normals, they satisfy the SBP property~\cite{gassner2013skew,Hicken2016_SBP,ranocha2018generalised}
\begin{equation}
  \label{eq:sbp_property}
  \mtx{M}\mtx{D}^{j}+(\mtx{D}^{j})^{T}\mtx{M}=\mtx{R}^{T}\mtx{B}\mtx{N}^{j}\mtx{R},
\end{equation}
which is a discrete analog of integration by parts.

Curved elements are obtained through a mapping, e.g., in two dimensions~$\vec{x}(\xi^1, \xi^2) = \vec{x}(\xi,\eta)  = ( X(\xi, \eta), Y(\xi, \eta))$, following the approach of Kopriva~\cite{Kopriva2006,Kopriva2009}.
Under the mapping $\vec{x}(\xi,\eta)$, with Jacobian $J = X_{\xi}Y_\eta - X_\eta Y_{\xi}$, the
system~\eqref{eq:nonconservative_form} becomes
\begin{equation}
  \label{eq:curvilinear_continuous}
  J\,\partial_t \vec{u}
  + \sum_{j=1}^{d} \partial_{\xi_j}\, \tilde{\vec{f}}^{j}(\vec{u})
  + \sum_{j=1}^{d} \tilde{\vec{H}}(\vec{u})\,\partial_{\xi_j}\, \tilde{\vec{g}}^{j}(\vec{u})
  = \vec{0},
\end{equation}
with the contravariant fluxes
\begin{equation}
  \tilde{\vec{f}}^{j}=\hat{\vec{n}}^{\,j}\cdot(\vec{f}^{1},\vec{f}^{2}),\qquad
  \tilde{\vec{g}}^{j}=\hat{\vec{n}}^{\,j}\circ(\vec{g}^{1},\vec{g}^{2}),\qquad
  \tilde{\vec{H}}=(\vec{H}^{1},\vec{H}^{2}),
\end{equation}
where $\circ$ is the Hadamard product. The contravariant normals $\hat{\vec{n}}^{\,j}$,
\begin{equation}
\hat{\vec{n}}^{1} = \left (Y_\eta, -X_\eta \right)^T, \quad \hat{\vec{n}}^{2} = \left (-Y_\xi, X_\xi \right)^T,
\end{equation}
are built from the metric terms $\vec{X}_\xi,\vec{X}_\eta,\vec{Y}_\xi,\vec{Y}_\eta$ and satisfy the discrete metric identities~\cite{Kopriva2006,Kopriva2009}
\begin{equation}
  \label{metric_identities}
  \mtx{D}^{1}\,\vec{Y}_\eta - \mtx{D}^{2}\,\vec{Y}_\xi = \vec{0},
  \qquad
  \mtx{D}^{2}\,\vec{X}_\xi - \mtx{D}^{1}\,\vec{X}_\eta = \vec{0},
\end{equation}
so that the free-stream preservation property holds~\cite{Kopriva2006, VisbalGaitonde1999}.
In three dimensions, the metric terms are computed using the invariant curl form of Kopriva~\cite{Kopriva2006}, a modification of the conservative curl form of Thomas and Lombard~\cite{ThomasLombard}, which satisfies the discrete metric identities (see also~\cite{CREAN2018410} for an alternative construction of the metric terms and~\cite{NOLASCO2020104555} for optimized geometric metrics).

Given these definitions, the semidiscretization reads
\begin{equation}
  \label{eq:semi_high_order}
  J_i \partial_t \vec{u}_i + \vec{\mathrm{VOL}}_i + \vec{\mathrm{SURF}}_i = \vec{0},
\end{equation}
where the volume term collects the conservative volume numerical fluxes and the antisymmetric nonconservative product,
\begin{equation}
  \label{volume_terms_curvilinear}
  \vec{\mathrm{VOL}}_i := \sum_{j=1}^{d}\Bigl(
     \sum_{k=1}^{N} 2\,\mtx{D}^{j}_{ik}\,\tilde{\vec{f}}^{\,\mathrm{vol},j}(\vec{u}_i,\vec{u}_k)
     + \sum_{k=1}^{N} \mtx{D}^{j}_{ik}\,\tilde{\vec{H}}^{\,\mathrm{vol}}(\vec{u}_i,\vec{u}_k)\,
       \jump{\tilde{\vec{g}}^{\,j}}_{ik} \Bigr),
\end{equation}
and the surface term couples neighboring elements through the surface numerical fluxes,
\begin{equation}
  \label{surface_terms_curvilinear}
  \vec{\mathrm{SURF}} := \sum_{j=1}^{d}\mtx{M}^{-1}\mtx{R}^{T}\mtx{B}\Bigl(
     \tilde{\vec{f}}^{\,\mathrm{num},j}
     + \tfrac{1}{2}\tilde{\vec{H}}^{\,\mathrm{num}}\jump{\tilde{\vec{g}}^{\,j}}
   \Bigr)
   - \sum_{j=1}^{d}\mtx{M}^{-1}\mtx{R}^{T}\mtx{B}\,\mtx{N}^j \,\mtx{R}\,\tilde{\vec{f}}^{\,j},
\end{equation}
with the discrete contravariant fluxes
\begin{equation}
  \label{eq:contravariant_fluxes}
  \begin{aligned}
    \tilde{\vec{f}}^{\,\mathrm{num},j} &= \vec{f}^{\,\mathrm{num}}\bigl(\vec{u}_{-},\vec{u}_{+},\hat{\vec{n}}\bigr),\\
    \tilde{\vec{f}}^{\,\mathrm{vol},j}(\vec{u}_i, \vec{u}_k) &= \vec{f}^{\,\mathrm{vol}}\bigl(\vec{u}_i,\vec{u}_k,\mean{\hat{\vec{n}}^{\,j}}_{ik}\bigr)
      = \mean{\hat{\vec{n}}^{\,j}}_{ik}\cdot(\vec{f}^{\mathrm{vol},1},\vec{f}^{\mathrm{vol},2}),\\
    \tilde{\vec{g}}^{\,j} &= \mean{\hat{\vec{n}}^{\,j}}\circ(\vec{g}^{1},\vec{g}^{2}),
  \end{aligned}
\end{equation}
and $\tilde{\vec{H}}^{\mathrm{vol}}$ defined analogously.
Here $\vec{f}^{\,\mathrm{num}}(\cdot,\cdot,\vec{n})$ and $\vec{f}^{\,\mathrm{vol}}(\cdot,\cdot,\vec{n})$ denote two-point numerical fluxes in a general direction $\vec{n}$, i.e., consistent in the sense $\vec{f}(\vec{u},\vec{u},\vec{n}) = \vec{n}\cdot(\vec{f}^{1},\vec{f}^{2})$.
For the numerical flux $\vec{f}^{\,\mathrm{num}}\bigl(\vec{u}_{-},\vec{u}_{+},\hat{\vec{n}}\bigr)$ at the interfaces between elements at $\xi_j=\pm1$, $\vec{u}_{-}$ denotes the
interior state, $\vec{u}_{+}$ the exterior state, and $\hat{\vec{n}}$ the outward-pointing normal.

Since the diagonal-norm SBP operators mimic integration by parts at the discrete level, the algebraic Tadmor and well-balanced characterizations of the finite-volume analysis extend to high-order SBP flux-differencing discretizations in curvilinear coordinates or on unstructured grids~\cite{RENAC20191,WARUSZEWSKI2022111507, artiano2026affordable, chan2018discretely, chan2019skew}.
The structure-preserving properties of the finite volume scheme therefore carry over at arbitrary polynomial degree and, provided~\eqref{metric_identities} holds, on general curvilinear meshes: with the volume flux~\eqref{eq:tec_ec_wb_fluxes} and the surface flux~\eqref{eq:es_tec_wb_fluxes}, the scheme is entropy-stable and well-balanced.
Before stating the theorem, we extend the entropy-stable interface flux~\eqref{eq:es_tec_wb_fluxes} to a general direction $\vec{n}$,
\begin{equation}
    \label{eq:es_flux_normal}
    \begin{aligned}
    f^{\varrho,\mathrm{num}} &= \left ( \mean{\varrho} - \frac{\jump{\varrho}}{2}\mathrm{sign}(V^{\mathrm{num}}) \right ) V^{\mathrm{num}}, \\
    \vec{f}^{\,\vec{V},\mathrm{num}} &= \mean{K}\,\vec{n}
      - \frac{\mean{c}}{2\mean{\varrho}} \left ( \jump{\varrho \vec{V}}\cdot\vec{n} \right ) \vec{\nu}
      - \frac{\bigl| \mean{\vec{V}}\cdot\vec{n} \bigr|}{2\mean{\varrho}} \jump{\varrho \vec{V}}, \\
    f^{\varrho\theta,\mathrm{num}} &= f^{\varrho,\mathrm{num}} \left ( \theta^{\mathrm{num}} - \frac{\theta^{\mathrm{min}}}{2\mean{\theta}}\jump{\theta}\,\mathrm{sign}(V^{\mathrm{num}}) \right ), \\
    V^{\mathrm{num}} &= \mean{\vec{V}}\cdot\vec{n} - \norm{\vec{n}} \left ( c_p \theta^{\mathrm{num}} \frac{\jump{\Pi}}{2\mean{c}} + \frac{\jump{\phi}}{2\mean{c}} \right ), \\
    \theta^{\mathrm{num}} &= \frac{1}{\mean{1/\theta}_{\mathrm{log}}}, \\
    v_i^{\mathrm{num}} & = \frac{\mean{\varrho v_i}}{\mean{\varrho}},
    \end{aligned}
\end{equation}
where $\vec{\nu} = \vec{n}/\norm{\vec{n}}$.

\begin{theorem}
\label{th:dgsem_curvilinear}
Consider the semidiscretization~\eqref{eq:semi_high_order} with the volume and surface terms~\eqref{volume_terms_curvilinear} and~\eqref{surface_terms_curvilinear}.
If the volume fluxes $\vec{f}^{\mathrm{vol}},\vec{H}^{\mathrm{vol}}$ are as in Theorem~\ref{th:ec_tec_wb} and the surface fluxes $\vec{f}^{\mathrm{num}},\vec{H}^{\mathrm{num}}$ as in Theorem~\ref{th:es_wb}, $\mtx{M}$ and $\mtx{R}^{T}\mtx{B}\mtx{N}^{j}\mtx{R}$ are diagonal and the metric identities~\eqref{metric_identities} hold discretely, then the vector-invariant semidiscretization is entropy-stable and well-balanced on general curvilinear meshes with slip wall or periodic boundary conditions.
\end{theorem}
\begin{proof}
  The result is a specialization of Theorems~6.1 and~6.6 of~\cite{artiano2026affordable} to the vector-invariant fluxes of Theorems~\ref{th:ec_tec_wb} and~\ref{th:es_wb}.
  The well-balanced property follows from Lemmas~\ref{lemma:wb_temperature} and~\ref{lemma:wb_theta}.
  The slip wall boundary conditions are imposed through the interface entropy-stable flux~\eqref{eq:es_flux_normal}, evaluated at the outward-pointing face normal $\hat{\vec{n}}$, using a reflected exterior state
  \begin{equation}
    \label{eq:mirrored_state}
    \vec{u}_b = \left ( \varrho_-,\; \vec{V}_- - 2\left (\vec{V}_-\cdot\hat{\vec{\nu}}\right )\hat{\vec{\nu}},\;
      (\varrho\theta)_- \right ), \qquad \hat{\vec{\nu}} = \frac{\hat{\vec{n}}}{\norm{\hat{\vec{n}}}}.
\end{equation}
We now derive the condition for entropy stability at the boundary.
For the compressible Euler equations in conservative form, this condition has already been derived by Gassner et al.~\cite{gassner2018br1} (see also~\cite{hindenlang2020wall,svard2021boundary}, where slip wall boundary conditions are proved to be entropy stable when imposed weakly through the numerical flux).
Here, we show that this result extends also to slip wall boundary conditions for the compressible Euler equations in vector-invariant form.
We contract the semidiscretization~\eqref{eq:semi_high_order} with $\vec{\omega}^{T} \mtx{M}$.
The SBP property~\eqref{eq:sbp_property}, the metric identities~\eqref{metric_identities} and the entropy conservation of the volume fluxes then
yield the curvilinear counterpart of Eq.~(178) in~\cite{artiano2026affordable}, which reads
\begin{equation}
    \label{eq:entropy_rate_element}
    \frac{\mathrm{d}}{\mathrm{d} t}\sum_{i=1}^{N}\mtx{M}_{ii}\,J_i\,U(\vec{u}_i)
      = \sum_{j=1}^{d}\sum_{i,k=1}^{N}
          \left ( \mtx{R}^{T}\mtx{B}\,\mtx{N}^{j}\mtx{R} \right )_{i,k}
          \tilde\psi^{\,j}_{k}
      - \vec{\omega}^{T}\sum_{j=1}^{d}\mtx{R}^{T}\mtx{B}\left ( \tilde{\vec{f}}^{\,\mathrm{num},j}
        + \tfrac{1}{2}\tilde{\vec{H}}^{\,\mathrm{num}}
          \jump{\tilde{\vec{g}}^{\,j}} \right ),
\end{equation}
with $\tilde\psi^{\,j} = \hat{\vec{n}}^{\,j}\cdot(\psi^1, \psi^2)$ the contravariant entropy
potential.
The entropy rate of the element then reduces to
\begin{equation}
    \label{eq:entropy_rate_faces}
    \frac{\mathrm{d}}{\mathrm{d} t}\sum_{i=1}^{N}\mtx{M}_{ii}\,J_i\,U(\vec{u}_i)
      = -\vec{1}^{T}\sum_{j=1}^{d}\mtx{R}^{T}\mtx{B}\,\tilde{\vec{F}}^{j}\,,
    \qquad
    \tilde{F}^{j} := \vec{\omega}_{-} \cdot \left (
        \tilde{\vec{f}}^{\,\mathrm{num},j}
      + \frac{1}{2}\tilde{\vec{H}}^{\,\mathrm{num}}\,
        \jump{\tilde{\vec{g}}^{\,j}} \right ) - \tilde\psi^{\,j}_{-}\,,
\end{equation}
where $\tilde{\vec{F}}^{j}$ collects the values of $\tilde{F}^{j}$ at the face nodes, $\vec{u}_{-}$ denotes the value of the element itself, and $\vec{u}_{+}$ denotes the value of the neighboring element or the prescribed boundary state at the boundary of the domain.
For a face normal $\hat{\vec n}$, we write $\tilde{\vec g}=\hat{\vec n}\circ(\vec g^1,\vec g^2)$ and $\tilde{\psi}(\vec u,\hat{\vec n}) =\hat{\vec n}\cdot(\psi^1,\psi^2)$; thus, at each face node we have
\begin{equation}
    \label{eq:one_sided_face_term}
    \tilde{F}^{j} = \tilde{F}(\vec{u}_{-}, \vec{u}_{+}, \hat{\vec{n}}) := \vec{\omega}_{-} \cdot \left (
        \vec{f}^{\,\mathrm{num}}(\vec{u}_{-}, \vec{u}_{+}, \hat{\vec{n}})
      + \frac{1}{2}\tilde{\vec{H}}^{\,\mathrm{num}}\,\jump{\tilde{\vec{g}}} \right ) - \tilde\psi ( \vec{u}_-, \hat{\vec{n}})\,,
\end{equation}
where, with a slight abuse of notation, $\hat{\vec{n}} = \pm\hat{\vec{n}}^{\,j}$ on $\xi_j = \pm 1$ is the outward normal of the element and the face quantities satisfy
\begin{equation}
    \label{eq:telescoping_property}
    \vec{f}^{\,\mathrm{num}}(\vec{u}_{+}, \vec{u}_{-}, -\hat{\vec{n}}) = -\vec{f}^{\,\mathrm{num}}(\vec{u}_{-}, \vec{u}_{+}, \hat{\vec{n}}).
\end{equation}
Note that $\tilde{\vec{H}}^{\,\mathrm{num}}$ depends only on the two states, while the normal enters the nonconservative term through $\jump{\tilde{\vec{g}}}$, so both the jump and the normal change sign and $\tilde{\vec{H}}^{\mathrm{num}} \jump{\tilde{\vec{g}}}$ is invariant under a change of the normal direction.
At an interface, we have therefore
\begin{multline}
    \tilde{F}(\vec{u}_{+}, \vec{u}_{-}, -\hat{\vec{n}})
      = \vec{\omega}_{+} \cdot \left ( \vec{f}^{\,\mathrm{num}}(\vec{u}_{+}, \vec{u}_{-}, -\hat{\vec{n}})
      + \frac{1}{2}\tilde{\vec{H}}^{\,\mathrm{num}}\,\jump{\tilde{\vec{g}}} \right ) - \tilde\psi(\vec{u}_{+}, -\hat{\vec{n}})
      \\
      = \vec{\omega}_{+} \cdot \left ( -\vec{f}^{\,\mathrm{num}}(\vec{u}_{-}, \vec{u}_{+}, \hat{\vec{n}})
      + \frac{1}{2}\tilde{\vec{H}}^{\,\mathrm{num}}\,\jump{\tilde{\vec{g}}} \right ) + \tilde\psi_{+}\,.
\end{multline}
The two contributions of the interface then sum up to
\begin{equation}
    \label{eq:interface_pairing}
    \begin{aligned}
    \tilde{F}(\vec{u}_{-}, \vec{u}_{+}, \hat{\vec{n}}) + \tilde{F}(\vec{u}_{+}, \vec{u}_{-}, -\hat{\vec{n}})
      &= \left ( \vec{\omega}_{-} - \vec{\omega}_{+} \right ) \cdot \tilde{\vec{f}}^{\,\mathrm{num}}
      + \frac{1}{2} \left ( \vec{\omega}_{-} + \vec{\omega}_{+} \right ) \cdot \tilde{\vec{H}}^{\,\mathrm{num}}\jump{\tilde{\vec{g}}}
      - \tilde\psi_{-} + \tilde\psi_{+} \\
      &= - \left ( \jump{\vec{\omega}}\cdot\tilde{\vec{f}}^{\,\mathrm{num}}
        - \mean{\vec{\omega}}\cdot\tilde{\vec{H}}^{\,\mathrm{num}}\jump{\tilde{\vec{g}}}
        - \jump{\tilde\psi} \right ).
    \end{aligned}
\end{equation}
Summing up the contributions of all elements, the entropy rate reads
\begin{multline}
    \label{eq:entropy_rate_global}
    \frac{\mathrm{d}}{\mathrm{d} t}
      \sum_{\mathrm{elements}}\sum_{i=1}^{N}\mtx{M}_{ii}\,J_i\,U(\vec{u}_i)
      \\= \sum_{\substack{\mathrm{interior}\\ \mathrm{interfaces}}}\sum_{l}\mtx{B}_{ll}
        \left [
          \jump{\vec{\omega}}\cdot\tilde{\vec{f}}^{\mathrm{num}}
          - \mean{\vec{\omega}}\cdot\tilde{\vec{H}}^{\mathrm{num}}\jump{\tilde{\vec{g}}} -\jump{\tilde\psi}
        \right ]_l
      - \sum_{\substack{\mathrm{boundary}\\ \mathrm{interfaces}}}\sum_{l}\mtx{B}_{ll}\,\tilde{F}_{l}\,.
\end{multline}
Here, the sum over $l$ runs over all quadrature nodes on the interface.
The sum in square brackets is nonpositive by the entropy stability condition~\eqref{eq:TadmorDef_form3} and positivity of the weights $\mtx{B}_{ll}$.
At a boundary interface, the exterior state is the prescribed $\vec{u}_b$ and the condition for entropy stability reads
\begin{equation}
    \label{eq:wall_boundary_term}
    \tilde{F}_b = \vec{\omega}_{-} \cdot \left (
        \vec{f}^{\,\mathrm{num}}(\vec{u}_{-}, \vec{u}_{b}, \hat{\vec{n}})
      + \frac{1}{2}\vec{H}^{\,\mathrm{num}}(\vec{u}_{-}, \vec{u}_{b})\,
        \jump{\tilde{\vec{g}}} \right ) - \tilde\psi_{-} \geq 0
\end{equation}
at every boundary quadrature node. Note that condition~\eqref{eq:wall_boundary_term} generalizes the one derived in~\cite{gassner2018br1,svard2021boundary} by including the nonconservative contribution.

We now insert the reflection state~\eqref{eq:mirrored_state}.
Considering that $\tilde{\psi}_- = 0$ and that the only nonzero components of the entropy variables~\eqref{eq:ec_entropy_variables} are the density and potential-temperature ones, only $f^{\varrho,\mathrm{num}}$ and $f^{\varrho\theta,\mathrm{num}}$ enter the boundary term~\eqref{eq:wall_boundary_term}.
For the boundary state $\vec{u}_b$ in~\eqref{eq:mirrored_state}, $\mean{\vec{V}} \cdot \hat{\vec{n}} = 0$ and $\jump{\Pi} = \jump{\phi} = 0$, so $f^{\varrho,\mathrm{num}} = f^{\varrho\theta,\mathrm{num}} = 0$ and the boundary term vanishes.
Therefore, $\tilde{F}_b = 0$ for any choice of the momentum flux; the slip wall does not contribute to the entropy rate, and the semidiscretization is entropy stable up to the boundary.
\end{proof}

\begin{remark}
We emphasize that the semidiscretization of Theorem~\ref{th:dgsem_curvilinear} ensures entropy stability (and total energy conservation in the volume kernel) at the cost of a single nonlinear mean value, i.e., the inverse of the logarithmic mean.
In contrast, the schemes presented in~\cite{WARUSZEWSKI2022111507,artiano2025structurepreservinghighordermethodscompressible} require the evaluation of two nonlinear mean values.
Moreover, this scheme also preserves locally and globally two background states, while the schemes derived in~\cite{WARUSZEWSKI2022111507,artiano2025structurepreservinghighordermethodscompressible} can preserve only one given background state.
Furthermore, given the particular form of the nonconservative terms, which is antisymmetric, the computational cost of the flux-differencing DGSEM implementation can also be reduced for the nonconservative terms by computing only the upper or lower triangular part of the derivative operator~\cite{ranocha2023efficient, artiano2026affordable} (see also the implementation in our reproducibility repository~\cite{artiano2026vectorinvariantRepro}).
\end{remark}

\section{Numerical results}
\label{sec:numerical_results}
In this section, we validate the structure-preserving methods derived for the vector-invariant form~\eqref{eq:vector_invariant_form}, assessing the theoretical findings of the previous sections.
We have implemented our methods in Julia~\cite{bezanson2017julia} using Trixi.jl~\cite{ranocha2022adaptive,schlottkelakemper2021purely}.
Unless stated otherwise, the time integration is carried out with the third-order, four-stage SSP Runge--Kutta method~\cite{kraaijevanger1991contractivity} implemented in OrdinaryDiffEq.jl~\cite{rackauckas2017differentialequations}.
We generated the figures using Makie.jl~\cite{danisch2021makie}.
All source code required to reproduce the numerical results is available in our reproducibility repository~\cite{artiano2026vectorinvariantRepro}.

The section is organized as follows: we first assess the conservation of entropy and total energy for the vector-invariant formulation. We then validate the well-balanced schemes of Section~\ref{sec:well_balanced} and perform a convergence study on the inertia-gravity wave test case.
Subsequently, the methods are tested on linear hydrostatic and linear non-hydrostatic mountain waves, a breaking-wave test case, and a critical-layer test case.
Finally, the robustness of the presented method is tested on the baroclinic instability benchmark on a cubed-sphere mesh.
\subsection{Conservation of entropy and total energy}
\label{sec:results_conservation}
First, we analyze the conservation properties of the derived numerical fluxes in the absence of a geopotential term. We assess these properties on the three-dimensional inviscid Taylor--Green vortex, a canonical test for the transition to turbulence and the generation of under-resolved scales.
The initial condition is~\cite{gassner2016split}
\begin{equation}
    \begin{aligned}
            \varrho(\vec{x},0) &= 1, \\
            v_1(\vec{x},0) &= \sin(x) \cos(y) \cos(z), \\
            v_2(\vec{x},0) &= -\cos(x) \sin(y) \cos(z),\\
            v_3(\vec{x},0) &= 0, \\
            p(\vec{x},0) &= \frac{100}{\gamma} + \frac{(\cos(2x) + \cos(2y))(\cos(2z) + 2)}{16}.
    \end{aligned}
\end{equation}
The solution is computed on $[-\pi, \pi]^3$ using $4$ elements in each direction, with periodic boundary conditions and $\text{CFL} = 0.01$, up to the final time $T = 5$.
For this test case only, we employ the numerical fluxes of Theorem~\ref{th:ec_tec_wb} for both the volume and surface numerical fluxes.
Table~\ref{tab:tgv_conservation} reports, for the total-energy- and entropy-conservative (ETEC) flux~\eqref{eq:tec_ec_wb_fluxes} and polynomial degrees $p \in \{2,3,4,5\}$, the conservation error $|\int U - \int U_0|$ (computed using the quadrature induced by the mass matrix $\vec{M}$) and the semidiscrete entropy rates $\vec{\omega}^T \vec{M} \partial_t \vec{u}$ of the thermodynamic entropy $U_{\varrho s}$ and of the total energy $U_{\varrho E}$.
Both the thermodynamic entropy and the total energy are conserved up to machine precision for all polynomial degrees, the entropy rates remaining at the level of round-off, in agreement with the theoretical results.
\begin{table}[h!]
    \centering
    \caption{Maximum entropy error $| \int U_{\varrho s} - \int U_{\varrho s,0}|$ and
    entropy residual $\vec{\omega}_{\varrho s}^T \vec{M} \partial_t \vec{u}$, and total-energy
    error $|\int U_{\varrho E} - \int U_{\varrho E,0}|$ and total-energy residual
    $\vec{\omega}_{\varrho E}^T \vec{M} \partial_t \vec{u}$, for the Taylor--Green vortex for
    different polynomial degrees $p$ up to the final time $T = 5$.}
    \label{tab:tgv_conservation}
    \begin{tabular}{ccccc}
        \toprule
        & \multicolumn{2}{c}{Thermodynamic entropy} & \multicolumn{2}{c}{Total energy} \\
        \cmidrule(lr){2-3} \cmidrule(lr){4-5}
        $p$ & $|\int U_{\varrho s} - \int U_{\varrho s,0}|$ & $\vec{\omega}_{\varrho s}^T \vec{M} \partial_t \vec{u}$ & $|\int U_{\varrho E} - \int U_{\varrho E,0}|$ & $\vec{\omega}_{\varrho E}^T \vec{M} \partial_t \vec{u}$ \\
        \midrule
        2 & \num{1.91e-11} & \num{2.05e-18} & \num{5.55e-10} & \num{3.73e-15} \\
        3 & \num{2.57e-11} & \num{1.10e-17} & \num{5.24e-10} & \num{1.05e-15} \\
        4 & \num{3.23e-11} & \num{4.12e-18} & \num{6.53e-10} & \num{6.37e-16} \\
        5 & \num{3.86e-11} & \num{2.99e-18} & \num{7.86e-10} & \num{1.29e-15} \\
        \bottomrule
    \end{tabular}
\end{table}

\subsection{Well-balancedness}
\label{sec:results_balance}
We assess the well-balancing properties of Lemmas~\ref{lemma:wb_temperature} and~\ref{lemma:wb_theta} on a curvilinear mesh, considering two canonical hydrostatic background states: an isothermal atmosphere ($T_0 = \SI{250}{K}$) and a constant-potential-temperature atmosphere ($\theta_0 = \SI{300}{K}$).
In both cases, the velocity field is initially at rest, and the density and pressure satisfy the hydrostatic balance.
The domain is discretized into $16 \times 16$ elements with polynomial degree~$3$, and the time step is fixed to $\Delta t = \SI{0.025}{s}$. We employ the smooth curvilinear mapping~\cite{artiano2025structurepreservinghighordermethodscompressible}
\begin{equation}
x = 500(1 + \xi + 0.1 \sin(\pi \xi) \sin(\pi \eta)), \quad z = 500(1 + \eta + 0.1 \sin(\pi \xi) \sin(\pi \eta)),
\end{equation}
which produces the warped mesh shown in Figure~\ref{fig:wb_mesh}. The simulations are run up to ten days ($T = \SI{864000}{s}$).
\begin{figure}[tbp]
    \centering
    \includegraphics[width = 0.55\textwidth]{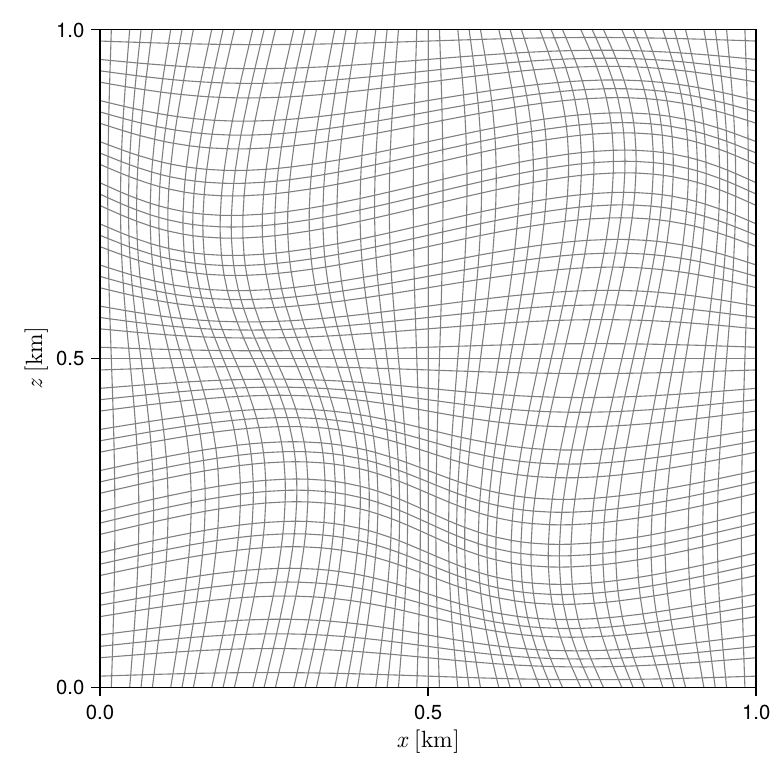}
    \caption{Warped mesh used for the well-balancedness test cases.}
    \label{fig:wb_mesh}
\end{figure}
Figure~\ref{fig:balance} reports the $L_2$ error of the horizontal and vertical velocity over the ten days, both in double precision (\texttt{Float64} in Julia) and double-double precision (using MultiFloats.jl \cite{MultiFloatsJL}).
% see https://github.com/dzhang314/MultiFloats.jl/issues/88 for the reference
For the stably stratified isothermal atmosphere ($\mathrm{d}\theta/\mathrm{d}z > 0$, so that $\mathcal{N}^2 > 0$), the hydrostatic balance is preserved up to machine precision throughout, in both precisions.
The constant potential temperature background state is instead \emph{neutrally} stratified ($\mathrm{d}\theta/\mathrm{d}z = 0$, $\mathcal{N}^2 = 0$): the hydrostatic rest state is metastable.
Therefore, a physical instability is triggered, which causes the growth of perturbations in double precision, growing to $\mathcal{O}(10^{-4})$ over ten days.
Repeating the simulation in double-double precision lowers the round-off floor by more than ten orders of magnitude, and the well-balancedness error remains at machine precision, confirming, as already observed in~\cite{artiano2025structurepreservinghighordermethodscompressible}, that this growth is entirely due to floating-point round-off and not to a deficiency of the well-balanced scheme.
\begin{figure}[tbp]
    \centering
    \includegraphics[width = \textwidth]{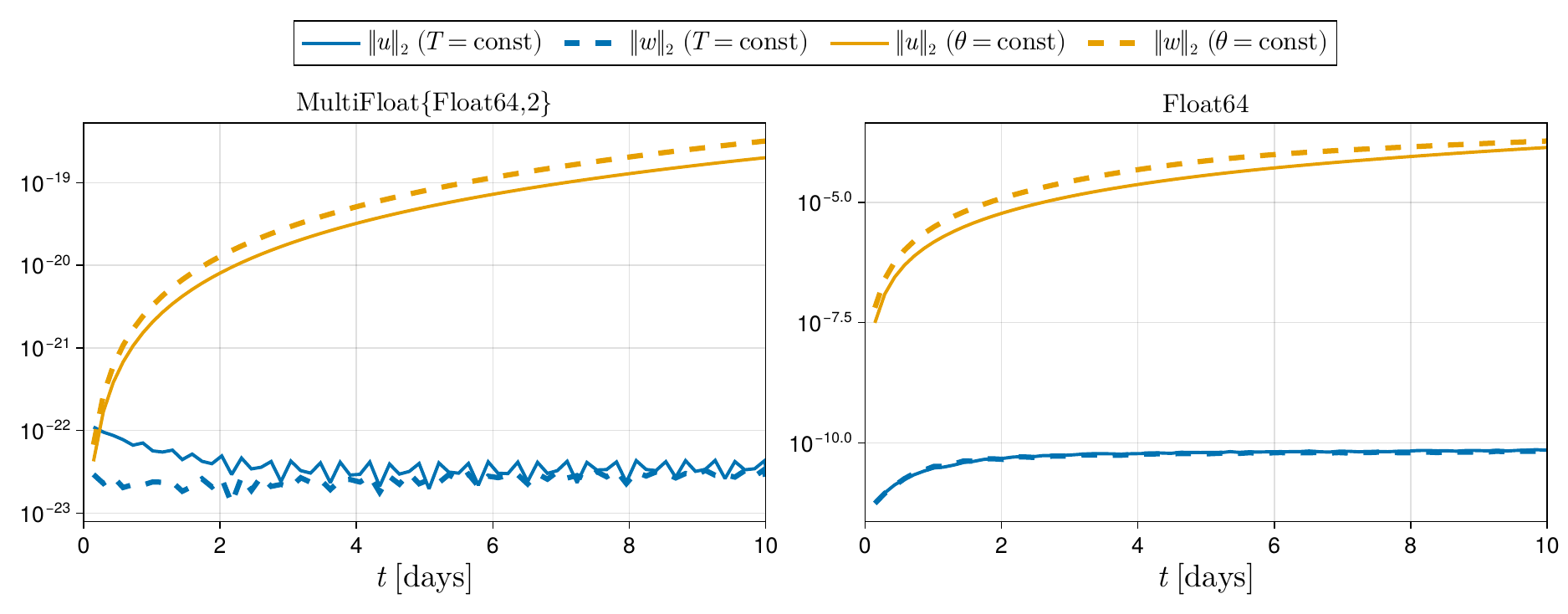}
    \caption{$L_2$ error of the horizontal ($\|u\|_2$) and vertical ($\|w\|_2$) velocity for the isothermal ($T = \mathrm{const}$) and constant-potential-temperature ($\theta = \mathrm{const}$) background states over ten days, in double-double precision (\texttt{MultiFloat\{Float64,2\}}, left) and double precision (\texttt{Float64}, right), for the vector-invariant formulation on a curvilinear mesh.}
    \label{fig:balance}
\end{figure}

\subsection{Inertia-gravity waves}
\label{sec:results_igw}
The high-order convergence of the DGSEM is assessed using the linearized solution of Baldauf and Brdar~\cite{baldauf2013} for inertia-gravity waves in a channel. The atmosphere is initialized with an isothermal background state at $T_0 = \SI{250}{K}$, perturbed by a warm bubble
\begin{equation}
    T' = \Delta T \sin \left (\frac{ \pi z}{H} \right ) \exp \left (-\frac{(x - x_c)^2}{a^2}\right ),
\end{equation}
with $H = \SI{10}{km}$, channel length $L = \SI{300}{km}$, $x_c = \SI{100}{km}$, $a = \SI{5}{km}$, and geopotential $\phi = g z$, $g = \SI{9.81}{m/s^2}$.
The solution is computed up to $t = \SI{1800}{s}$ with $\Delta T = \SI{1.0e-3}{K}$ to avoid nonlinear saturation effects~\cite{blaise2016,BALDAUF2021110635,WARUSZEWSKI2022111507}.
The domain is discretized with $N_x \times N_z$ elements at a fixed ratio $\Delta x / \Delta z = 3$ and $\text{CFL} = 0.1$.
Figure~\ref{fig:igw} reports the $L_2$ errors of the horizontal and vertical velocity, of the temperature and of the pressure as a function of the mesh size $\Delta x$, for polynomial degrees $p \in \{2,3,4\}$.
The structure-preserving scheme attains the expected high-order convergence rates.
As reported in~\cite{blaise2016,BALDAUF2021110635,WARUSZEWSKI2022111507}, saturation errors due to nonlinear effects appear at polynomial degree~$4$, and are particularly pronounced in the horizontal velocity component.
\begin{figure}[!htb]
    \centering
    \includegraphics[width = \textwidth]{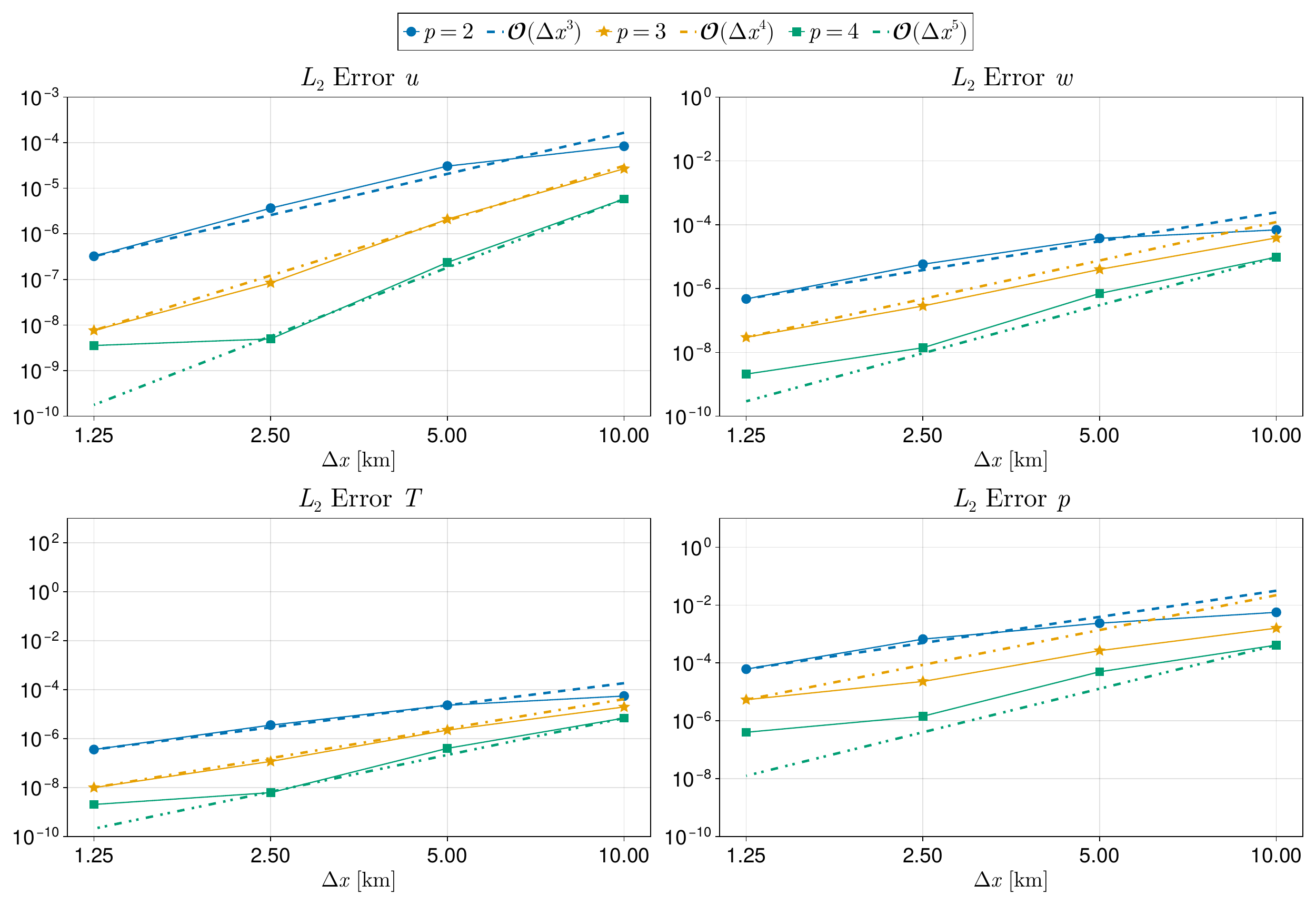}
    \caption{$L_2$ error of the horizontal velocity (top left), vertical velocity (top right), temperature (bottom left) and pressure (bottom right) as a function of the mesh size, for polynomial degrees $p \in \{2,3,4\}$, for the vector-invariant formulation.}
    \label{fig:igw}
\end{figure}

The convergence study is repeated on a curvilinear mesh, using the transformation of~\cite{WARUSZEWSKI2022111507}.
As shown in Figure~\ref{fig:igw_warped}, the observed convergence rates agree with those obtained in~\cite{WARUSZEWSKI2022111507,artiano2025structurepreservinghighordermethodscompressible}.
\begin{figure}[!htb]
    \centering
    \includegraphics[width = \textwidth]{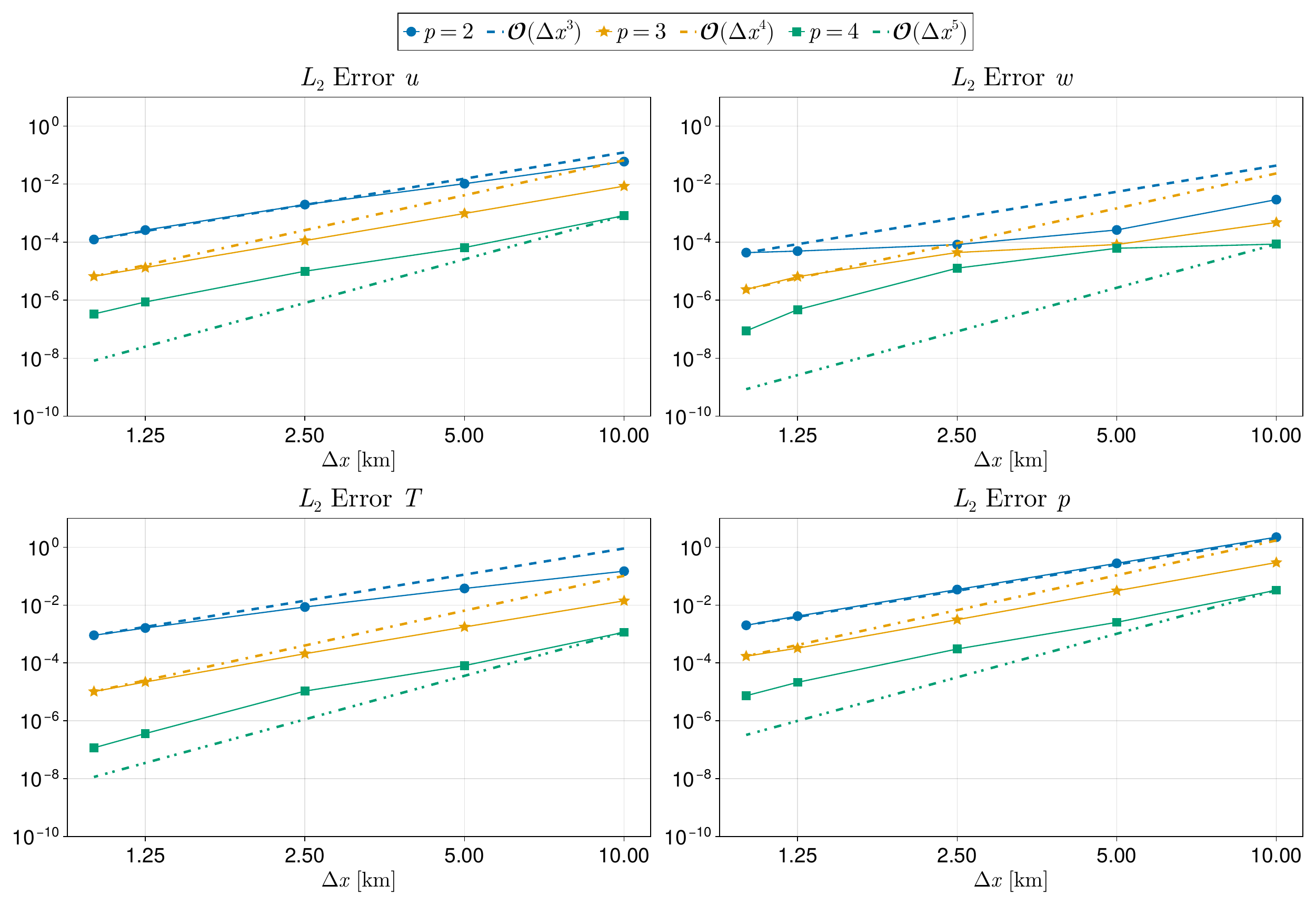}
    \caption{$L_2$ error of the horizontal velocity (top left), vertical velocity (top right), temperature (bottom left) and pressure (bottom right) as a function of the mesh size, for polynomial degrees $p \in \{2,3,4\}$, for the vector-invariant formulation on the warped mesh.}
    \label{fig:igw_warped}
\end{figure}

Finally, Figure~\ref{fig:igw_contour} shows the vertical velocity $w$ at the final time $t = \SI{1800}{s}$, computed on a $90 \times 12$ mesh with polynomial degree $p = 3$ and $\text{CFL} = 1$.
\begin{figure}[tbp]
    \centering
    \includegraphics[width = \textwidth]{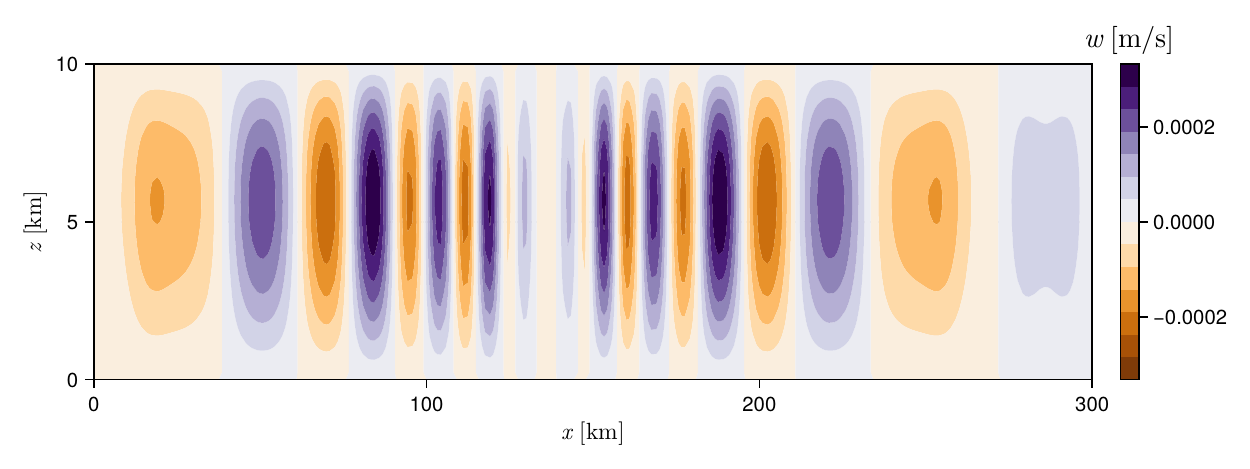}
    \caption{Vertical velocity $w$ of the inertia-gravity wave at the final time $t = \SI{1800}{s}$, computed with $90 \times 12$ elements and polynomial degree $p = 3$, for the vector-invariant formulation.}
    \label{fig:igw_contour}
\end{figure}

\subsection{Numerical results with orography}
\label{sec:results_orography}
We now present numerical results for four idealized flows over orography.
To avoid spurious reflections at the domain boundaries, non-reflecting boundary conditions are prescribed at the top and lateral boundaries through Rayleigh damping layers~\cite{ORLANDO2023115124,GIRFOGLIO2025106510}. On the right-hand side of the momentum and potential-temperature equations, we add the source terms
\begin{equation}
    \begin{aligned}
    \vec{q}_{\vec{V}} &= (\vec{V} - \vec{V}^0)( S_{\vec{V}}  + S_{h_1} + S_{h_2}),\\
    q_{\varrho \theta} &= \varrho (\theta - \theta^0)( S_{\vec{V}}  + S_{h_1} + S_{h_2}),
    \end{aligned}
\end{equation}
where the Rayleigh damping profiles $S_{\vec{V}}$, $S_{h_1}$, and $S_{h_2}$ are defined as
\begin{equation}
\begin{aligned}
    & S_{\vec{V}}= \begin{cases}0, & z<z_B, \\
    -\alpha \sin ^2\left[\frac{\pi}{2}\frac{z-z_B}{z_T-z_B}\right], & z \geq z_B,\end{cases}\\
    & S_{h_1}= \begin{cases}0, & x<x_{B}, \\
    -\alpha \sin ^2\left[\frac{\pi}{2}\frac{x-x_{B}}{x_T-x_{B}}\right], & x \geq x_{B},\end{cases}\\
    & S_{h_2}= \begin{cases}0, & x>-x_{B}, \\
    -\alpha \sin ^2\left[\frac{\pi}{2}\frac{x+x_{B}}{x_B-x_T}\right], & x \leq -x_{B}.\end{cases}
    \end{aligned}
\end{equation}
Moreover, for the following test cases in this subsection, the time integration is done with the fourth-order, nine-stage, low-storage FSAL explicit Runge--Kutta method with adaptive time-stepping developed in~\cite{Ranocha2022Optimal}, as implemented in OrdinaryDiffEq.jl~\cite{rackauckas2017differentialequations}.
\subsubsection{Linear hydrostatic mountain}
\label{sec:results_hydrostatic}
The first benchmark is the linear hydrostatic mountain~\cite{ACompressibleModelfortheSimulationofMoistMountainWaves,GIRALDO20083849}, with the \textit{versiera di Agnesi} profile
\begin{equation}
    h(x) = \frac{h_c}{1 + (x/a_c)^2}.
    \label{eq:versiera}
\end{equation}
The domain is $[-120, 120] \times [0, 30]~\mathrm{km}^2$, with $h_c = \SI{1}{m}$ and $a_c = \SI{10}{km}$. The atmosphere has a constant background temperature $T_0 = \SI{250}{K}$ and a constant mean flow $\overline{u} = \SI{20}{m/s}$.
The Exner pressure is $\overline{\Pi} = \exp(-\mathcal{N}^2 z / g)$, with $\mathcal{N} = g/\sqrt{c_p T_0}$ denoting the Brunt--V\"ais\"al\"a frequency. Since $\mathcal{N} a_c / \overline{u} > 1$, the flow is in the hydrostatic regime~\cite{GIRALDO20083849}.
A free-slip condition is used at the bottom, and non-reflecting conditions at the top and lateral boundaries, with $\alpha = 0.035$, $z_B = \SI{15}{km}$, $z_T = \SI{30}{km}$, $x_B = \SI{60}{km}$, $x_T = \SI{120}{km}$. We use $100 \times 60$ elements and polynomial degree $p = 3$.
The quantity of interest is the analytical momentum flux~\cite{Smith1979TheIO}
\begin{equation}
    m = \int_{-\infty}^{\infty} \overline{\varrho}(z)\, u'(x,z)\, w'(x,z)\, \mathrm{d}x,
    \qquad
    m^H = - \frac{\pi}{4}\overline{\varrho}_s \overline{u}_s \mathcal{N} h_c^2,
\end{equation}
where $u'$, $w'$ are the velocity perturbations and $\overline{\varrho}_s$, $\overline{u}_s$ the surface background density and velocity.
The normalized momentum flux $m/m^H$ should approach~$1$.
Figure~\ref{fig:hydrostatic_momentum} shows the evolution of the normalized momentum flux at $T = \SI{2.5}{h}, \SI{5}{h}, \SI{7.5}{h}, \SI{10}{h}, \SI{12.5}{h}$, which agrees with the results obtained in~\cite{artiano2025structurepreservinghighordermethodscompressible,GIRALDO20083849, ORLANDO2023115124}.
\begin{figure}[tbp]
    \centering
    \includegraphics[width = 0.7\textwidth]{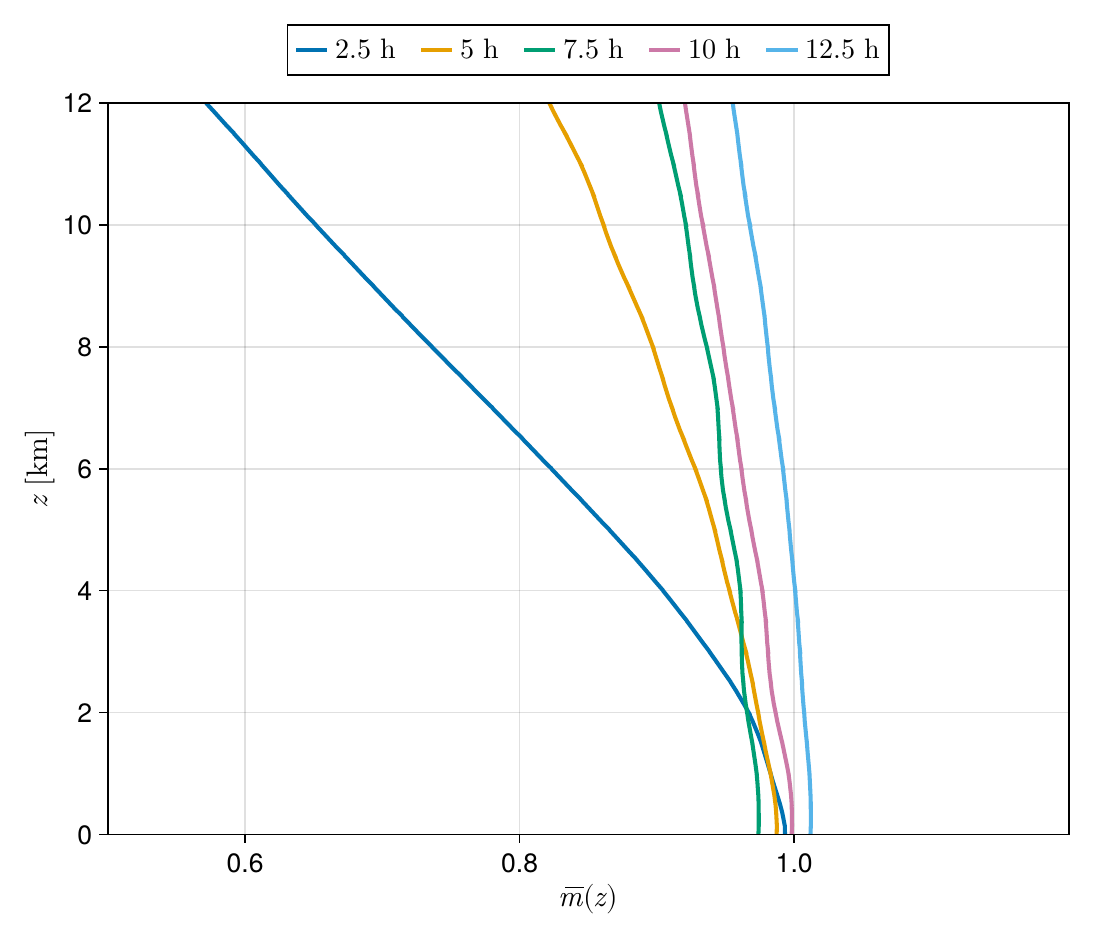}
    \caption{Evolution of the normalized momentum flux for the linear hydrostatic mountain with $100 \times 60$ elements and $p = 3$, for the vector-invariant formulation.}
    \label{fig:hydrostatic_momentum}
\end{figure}
Figure~\ref{fig:hydrostatic_contour} shows the horizontal and vertical velocity components at the final time $T = \SI{12.5}{h}$, analogous to the contour obtained in~\cite{GIRALDO20083849,artiano2025structurepreservinghighordermethodscompressible}.
\begin{figure}[tbp]
    \centering
    \includegraphics[width = \textwidth]{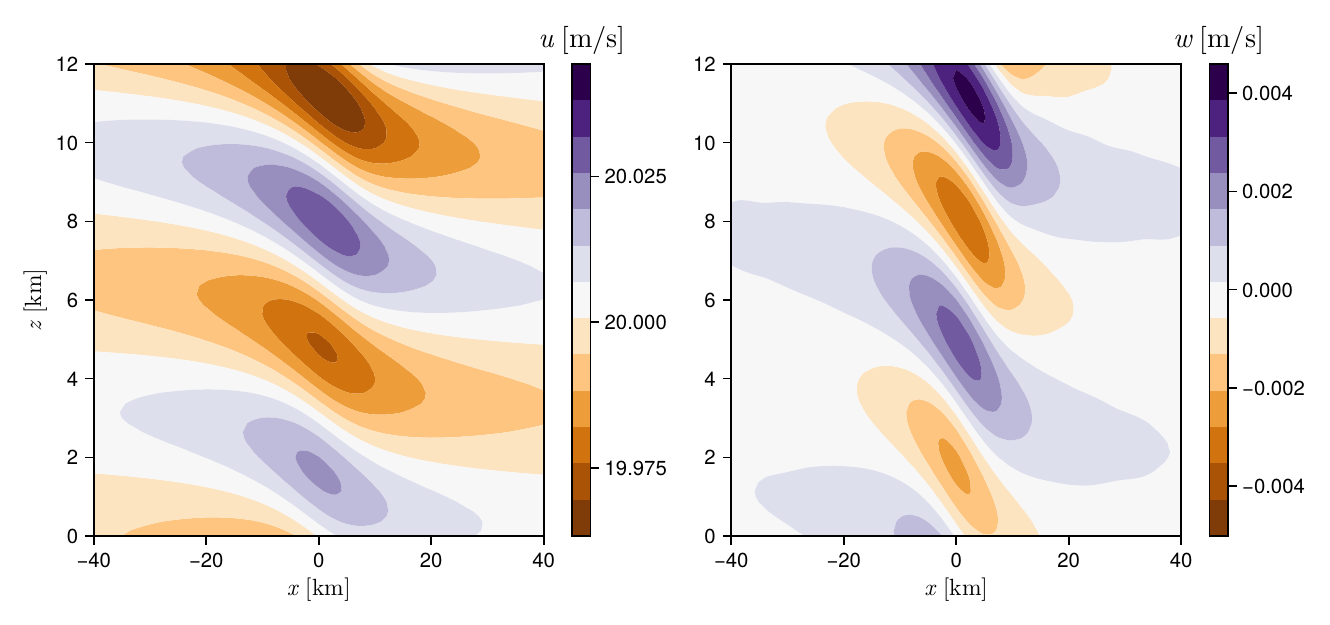}
    \caption{Contours of the horizontal (left) and vertical (right) velocity components for the linear hydrostatic mountain with $100 \times 60$ elements and $p = 3$, for the vector-invariant formulation.}
    \label{fig:hydrostatic_contour}
\end{figure}

\subsubsection{Linear non-hydrostatic mountain}
\label{sec:results_nonhydrostatic}
The second benchmark focuses on linear non-hydrostatic mountain waves~\cite{ACompressibleModelfortheSimulationofMoistMountainWaves,GIRALDO20083849}. The initial state is a uniform horizontal mean flow $\overline{u} = \SI{10}{m/s}$ within a uniformly stratified atmosphere with constant Brunt--V\"ais\"al\"a frequency $\mathcal{N} = \SI{0.01}{s^{-1}}$.
The Exner pressure and potential temperature follow the hydrostatic balance
\begin{equation}
\overline{\Pi}(z) = 1 + \frac{g^2}{c_p \theta_0 \mathcal{N}^2} \left( \exp\left(-\frac{\mathcal{N}^2}{g} z\right) - 1 \right),
\qquad
\theta(z) = \theta_0 \exp\left(\frac{\mathcal{N}^2}{g} z\right),
\end{equation}
with $\theta_0 = \SI{280}{K}$, $p = p_0 \overline{\Pi}^{c_p/R}$, $T = \theta \overline{\Pi}$, and $\overline{\varrho} = p/(R T)$. The mountain profile is again the \textit{versiera di Agnesi}~\eqref{eq:versiera} with $h_c = \SI{1}{m}$ and $a_c = \SI{1}{km}$.
The domain is $[-72, 72] \times [0, 30]~\mathrm{km}^2$ and the simulation spans $t \in [0, 8]$ hours. No-flux conditions are imposed at the bottom and non-reflecting conditions at the top and lateral boundaries, with damping layers in the last \SI{15}{km} vertically and \SI{32}{km} horizontally, and $\alpha = 0.03$. We use $200 \times 50$ elements and $p = 3$.
Since $\mathcal{N} a_c / \overline{u} = 1$, the flow is in the non-hydrostatic regime~\cite{SimpleTestsofaSemiImplicitSemiLagrangianModelon2DMountainWaveProblems, GIRALDO20083849}.
The normalized momentum flux is computed analogously to the hydrostatic case and normalized by $m^{NH} = 0.457\, m^H$~\cite{AnUpperBoundaryConditionPermittingInternalGravityWaveRadiationinNumericalMesoscaleModels}.
Figure~\ref{fig:nonhydrostatic_momentum} shows the evolution of the normalized momentum flux at $T = \SI{2}{h}, \SI{4}{h}, \SI{6}{h}, \SI{8}{h}$. The momentum progressively approaches the expected value of~$1$.
\begin{figure}[tbp]
    \centering
    \includegraphics[width = 0.7\textwidth]{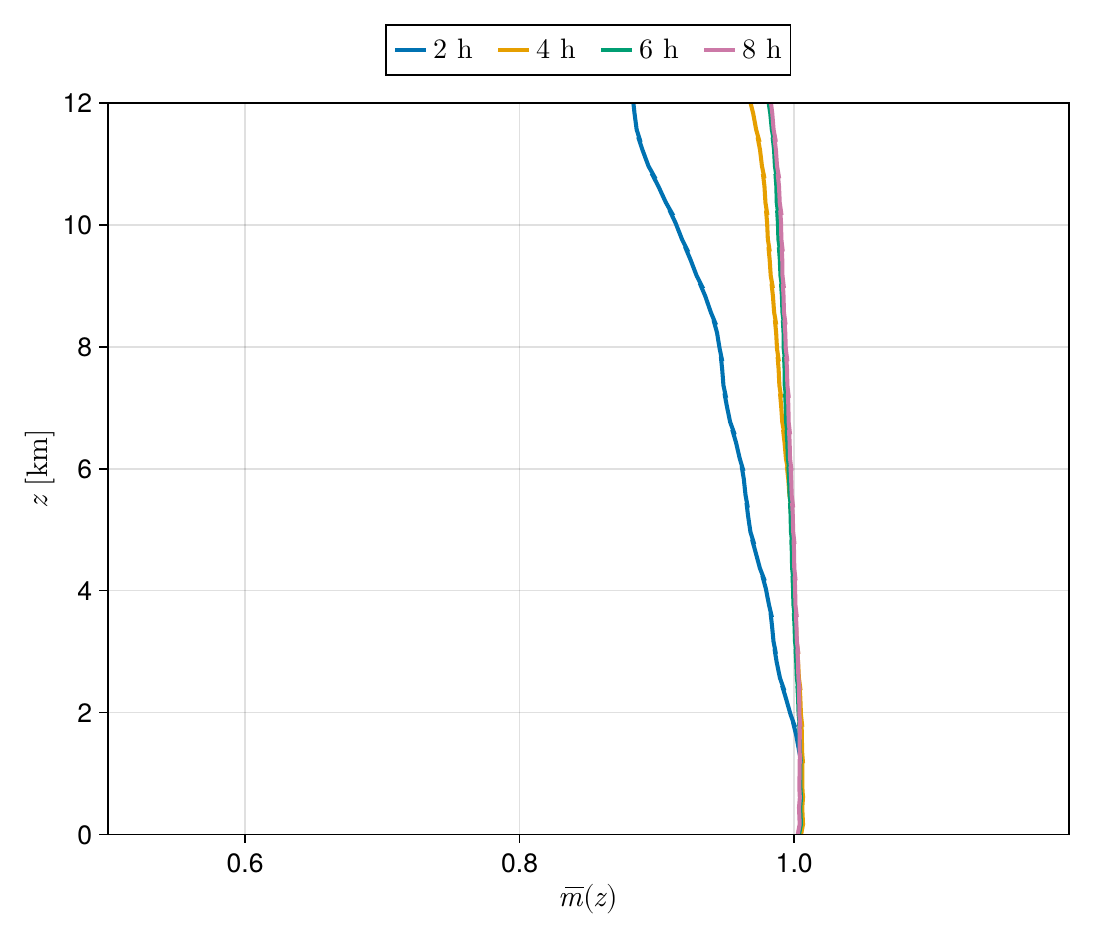}
    \caption{Evolution of the normalized momentum flux for the linear non-hydrostatic mountain with $200 \times 50$ elements and $p = 3$, for the vector-invariant formulation.}
    \label{fig:nonhydrostatic_momentum}
\end{figure}
Figure~\ref{fig:nonhydrostatic_contour} reports the horizontal and vertical velocity components at the final time, in agreement with the reference solutions~\cite{ACompressibleModelfortheSimulationofMoistMountainWaves,GIRALDO20083849}.
\begin{figure}[tbp]
    \centering
    \includegraphics[width = \textwidth]{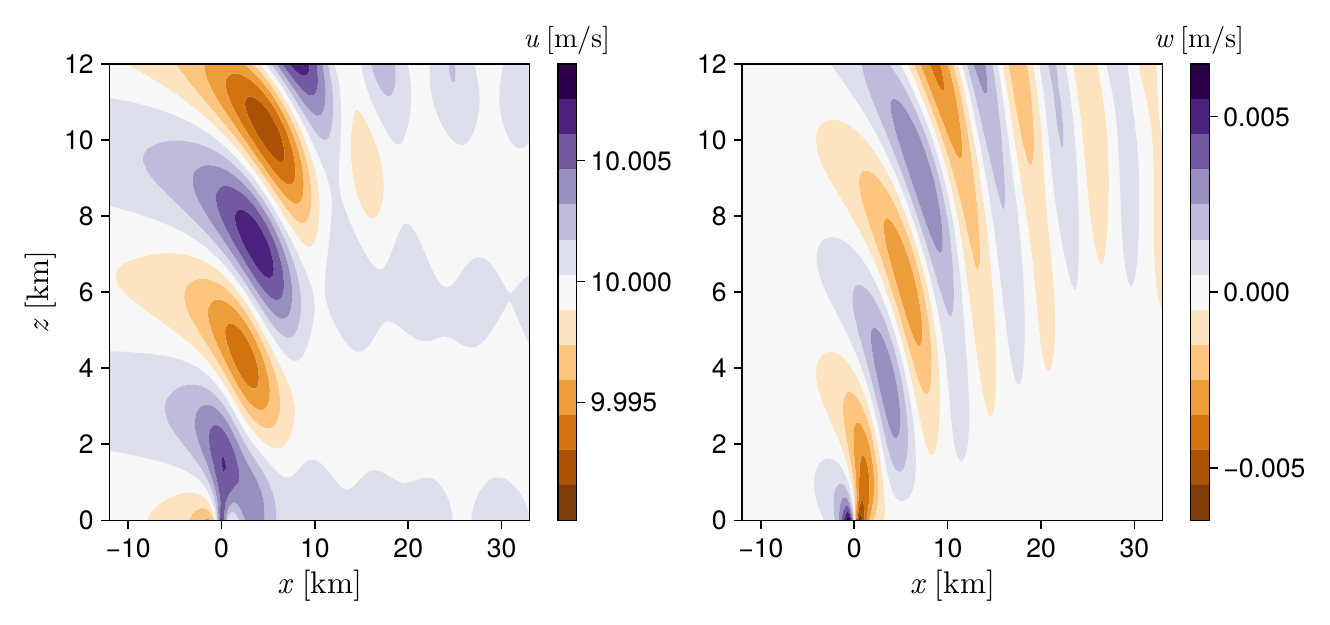}
    \caption{Contours of the horizontal (left) and vertical (right) velocity components for the linear non-hydrostatic mountain with $200 \times 50$ elements and $p = 3$, for the vector-invariant formulation.}
    \label{fig:nonhydrostatic_contour}
\end{figure}

\subsubsection{Breaking gravity waves}
\label{sec:results_breaking}
The third orographic benchmark is the breaking lee-wave test of Klein~\cite{klein2009asymptotics}, originally proposed by Smolarkiewicz and Margolin~\cite{smolarkiewicz1997forward}, which follows a large-amplitude mountain wave from its generation up to the onset of wave breaking.
A constant mean flow $\overline{u} = \SI{10}{m/s}$ is prescribed over an isothermal atmosphere with Brunt--V\"ais\"al\"a frequency $\mathcal{N} = \SI{0.01}{s^{-1}}$ and surface temperature $T_0 = g^2/(c_p \mathcal{N}^2)$, with $g = \SI{10}{m/s^2}$ and the modified ratio of specific heats $\gamma = 1.0697$ adopted in~\cite{klein2009asymptotics}. The orography is the \textit{versiera di Agnesi}~\eqref{eq:versiera} with $h_c = 200\pi \approx \SI{628}{m}$ and $a_c = \SI{1}{km}$, so that the wave amplitude lies well within the nonlinear regime. The domain is $[-60, 60] \times [0, 60]~\mathrm{km}^2$, discretized with $50 \times 30$ elements and polynomial degree $p = 4$; non-reflecting conditions are enforced through Rayleigh damping layers in the upper \SI{20}{km} and in the outer \SI{20}{km} on each lateral side, with $\alpha = 1/600$. The simulation is run up to $T = \SI{3}{h}$.
Figure~\ref{fig:breaking} shows the isentropic lines at $t = \SI{2.5}{h}$ and $t = \SI{3.0}{h}$. At \SI{2.5}{h}, the lee wave propagates upward and steepens with height, in close agreement with the reference solution of Klein~\cite{klein2009asymptotics}. At \SI{3.0}{h}, wave breaking sets in, generating small-scale instabilities. As already pointed out in~\cite{klein2009asymptotics}, beyond this stage the flow is governed by the dissipation of the numerical scheme and a quantitative comparison is no longer meaningful.
\begin{figure}[tbp]
    \centering
    \includegraphics[width = 0.75\textwidth]{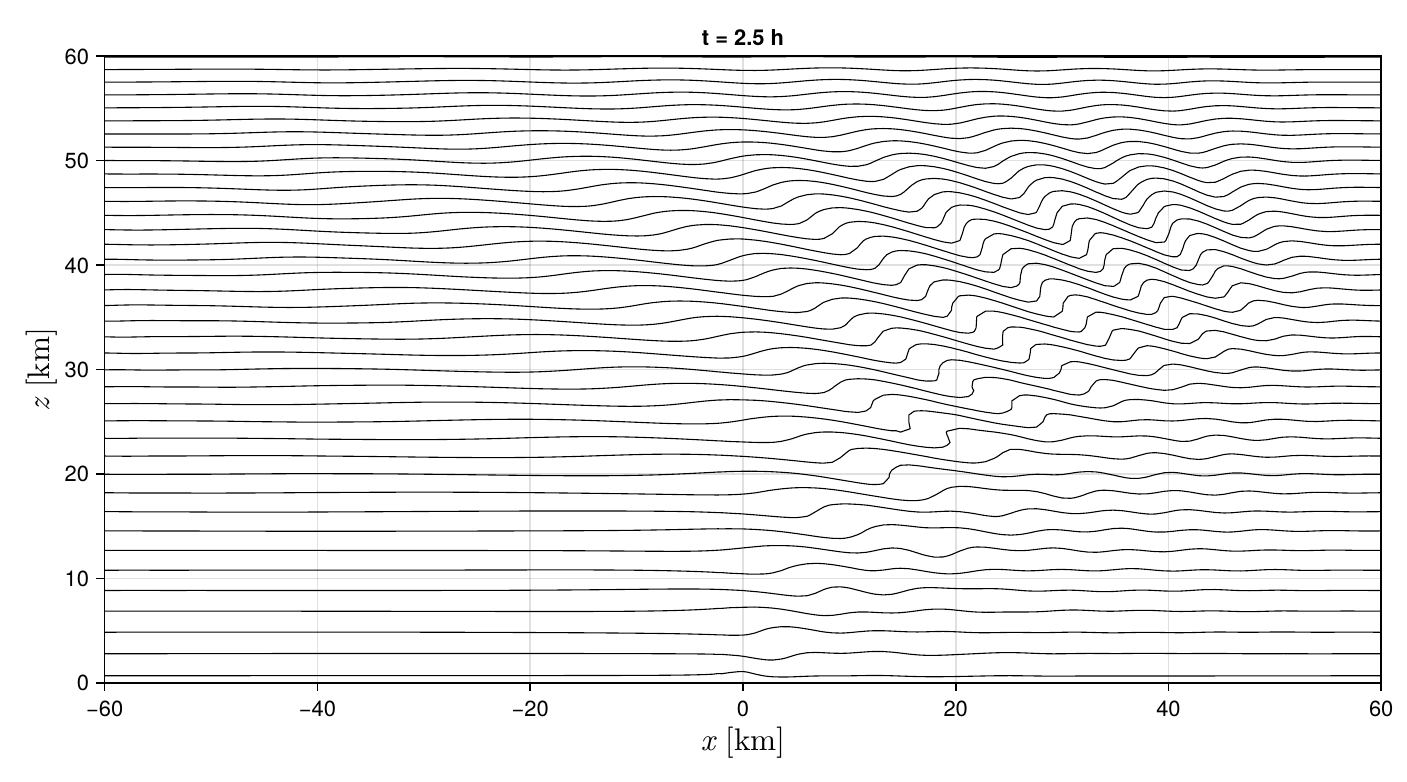}\\
    \includegraphics[width = 0.75\textwidth]{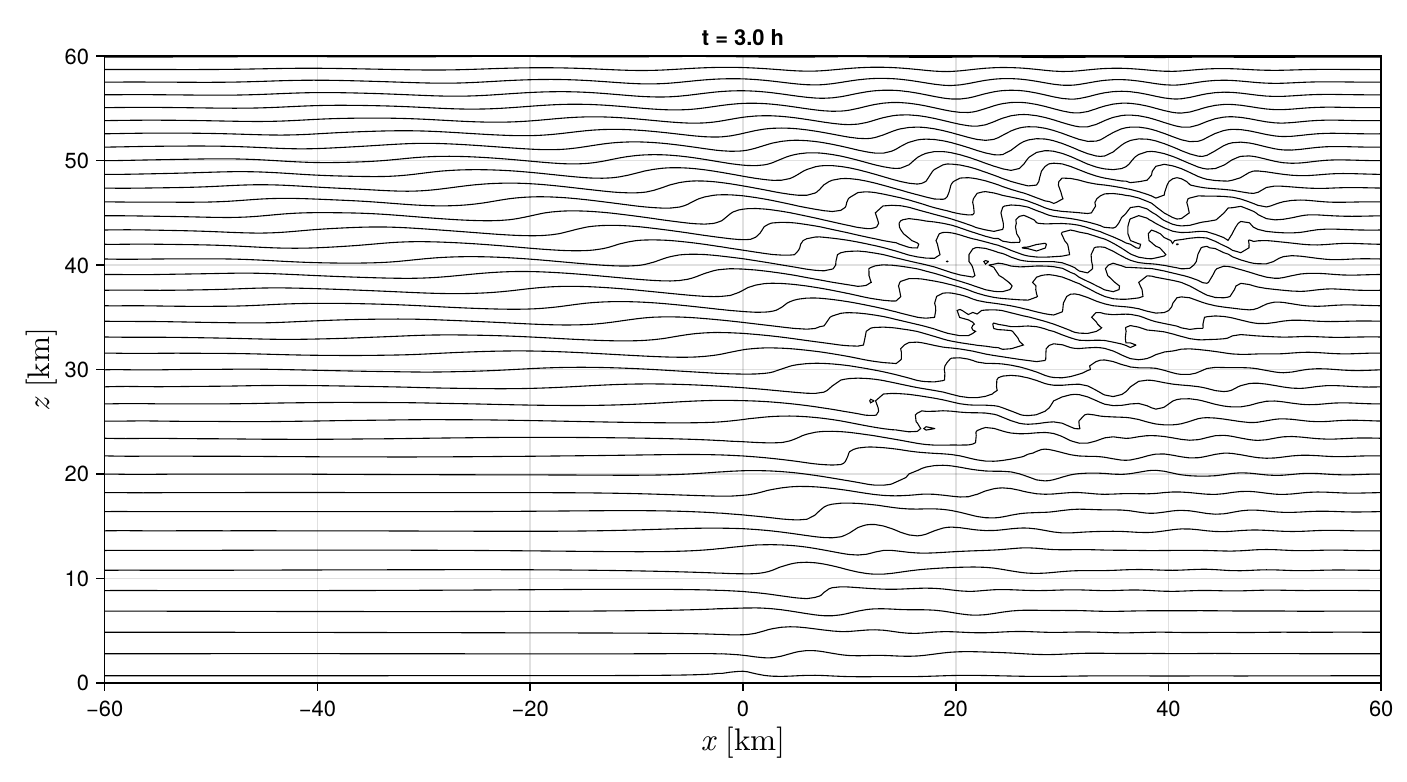}
    \caption{Isentropic lines (potential-temperature contours, with a contour spacing of \SI{5}{K}) for the breaking lee-wave test at $t = \SI{2.5}{h}$ (top) and $t = \SI{3.0}{h}$ (bottom), with $50 \times 30$ elements and $p = 4$, for the vector-invariant formulation.}
    \label{fig:breaking}
\end{figure}

\subsubsection{Critical layer}
\label{sec:results_critical}
The last orographic benchmark is inspired by the idealized study of Spichtinger and Smolarkiewicz~\cite{spichtinger2008critical}, who investigated the impact of the turbulence generated within a critical layer on the formation and evolution of cirrus clouds.
We follow the dynamical setup of their simulations, replacing the ice microphysics by a passive tracer to measure the induced mixing.
A linear wind profile $\overline{u}(z) = u_0 - \Lambda z$, with $u_0 = \SI{10}{m/s}$ and shear $\Lambda = \SI{1e-3}{s^{-1}}$, is prescribed over an atmosphere with constant Brunt--V\"ais\"al\"a frequency $\mathcal{N} = \SI{0.0097}{s^{-1}}$ and surface potential temperature $\theta_s = \SI{288.18}{K}$, so that the background wind vanishes at the critical level $z_c = u_0/\Lambda = \SI{10}{km}$.
The waves are excited by the fixed orography $z_b(x) = a_0\cos(2\pi x/\lambda_x)$, with $a_0 = \SI{20}{m}$ and $\lambda_x = \SI{10}{km}$.
The domain is $[-5, 5] \times [0, 20]~\mathrm{km}^2$, discretized with $80 \times 180$ elements and polynomial degree $p = 3$.
A free-slip condition is used at the bottom and non-reflecting conditions in the upper \SI{5}{km}, with $\alpha = \SI{0.03}{s^{-1}}$.
The passive tracer $\chi$ is prescribed as a Gaussian layer centered at $z_\chi = \SI{8.8}{km}$ with standard deviation $\sigma_\chi = \SI{250}{m}$, and the simulation is run up to $T = \SI{12}{h}$.
For this test, we use the third-order, four-stage SSP Runge--Kutta method~\cite{kraaijevanger1991contractivity} with the embedded method of \cite{conde2022embedded} and the efficient implementation of \cite{Ranocha2022Optimal} available in OrdinaryDiffEq.jl~\cite{rackauckas2017differentialequations}.

Figure~\ref{fig:critical_w} shows the time evolution of the vertical velocity in time intervals of \SI{90}{min}, starting at $t = \SI{4.5}{h}$.
The wave excited by the orography propagates upward and, as it approaches the critical level, its vertical wavelength $\lambda_z = 2\pi\overline{u}(z)/\mathcal{N}$ and its vertical group velocity tend to zero: the wave cannot reach $z_c$ and is absorbed by the mean flow beneath it~\cite{booker1967critical,grubisic1997critical}.
From \SI{7.5}{h} onward, small-scale structures develop beneath the critical level, where the wave breaks and the flow becomes turbulent, as in~\cite{spichtinger2008critical}.
\begin{figure}[tbp]
    \centering
    \includegraphics[width = \textwidth]{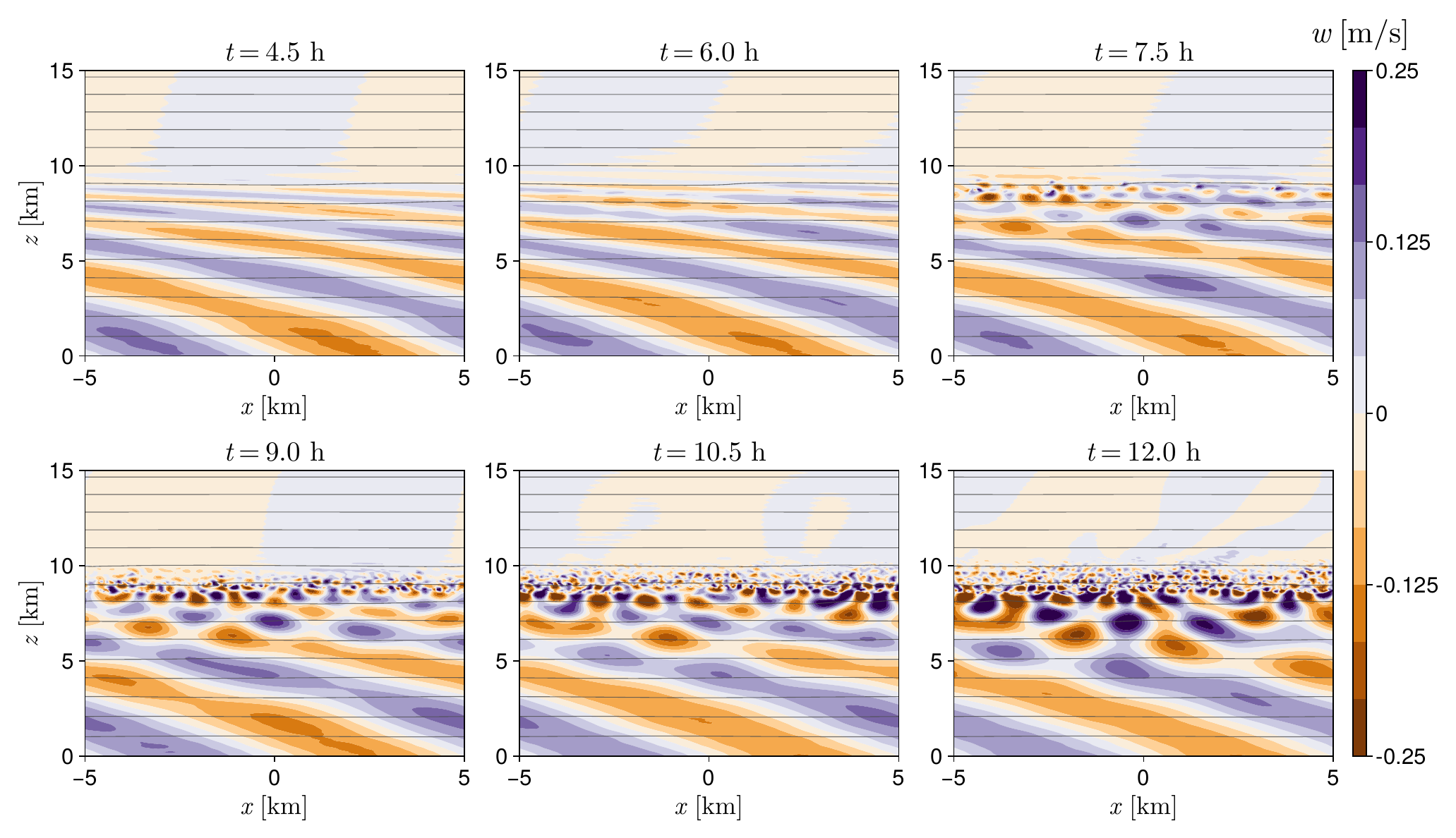}
    \caption{Vertical velocity for the critical-layer test case at $t = 4.5,\, 6,\, 7.5,\, 9,\, 10.5$ and \SI{12}{h}, with $80 \times 180$ elements and $p = 3$, for the vector-invariant formulation. Gray lines indicate the isentropic lines.}
    \label{fig:critical_w}
\end{figure}

Figure~\ref{fig:critical_tracer} shows the evolution of the passive tracer, which follows the wave motion until the breaking and is then stirred by the turbulence.
Figure~\ref{fig:critical_mixing} shows the deviation of the horizontally averaged tracer $\langle\chi\rangle = \int\varrho\chi\,\mathrm{d}x \,/ \int\varrho\,\mathrm{d}x$ from its initial profile: the tracer is transported from the center of the layer toward its edges, so that the layer is broadened by the mixing, which is still developing at $T = \SI{12}{h}$.
\begin{figure}[tbp]
    \centering
    \includegraphics[width = \textwidth]{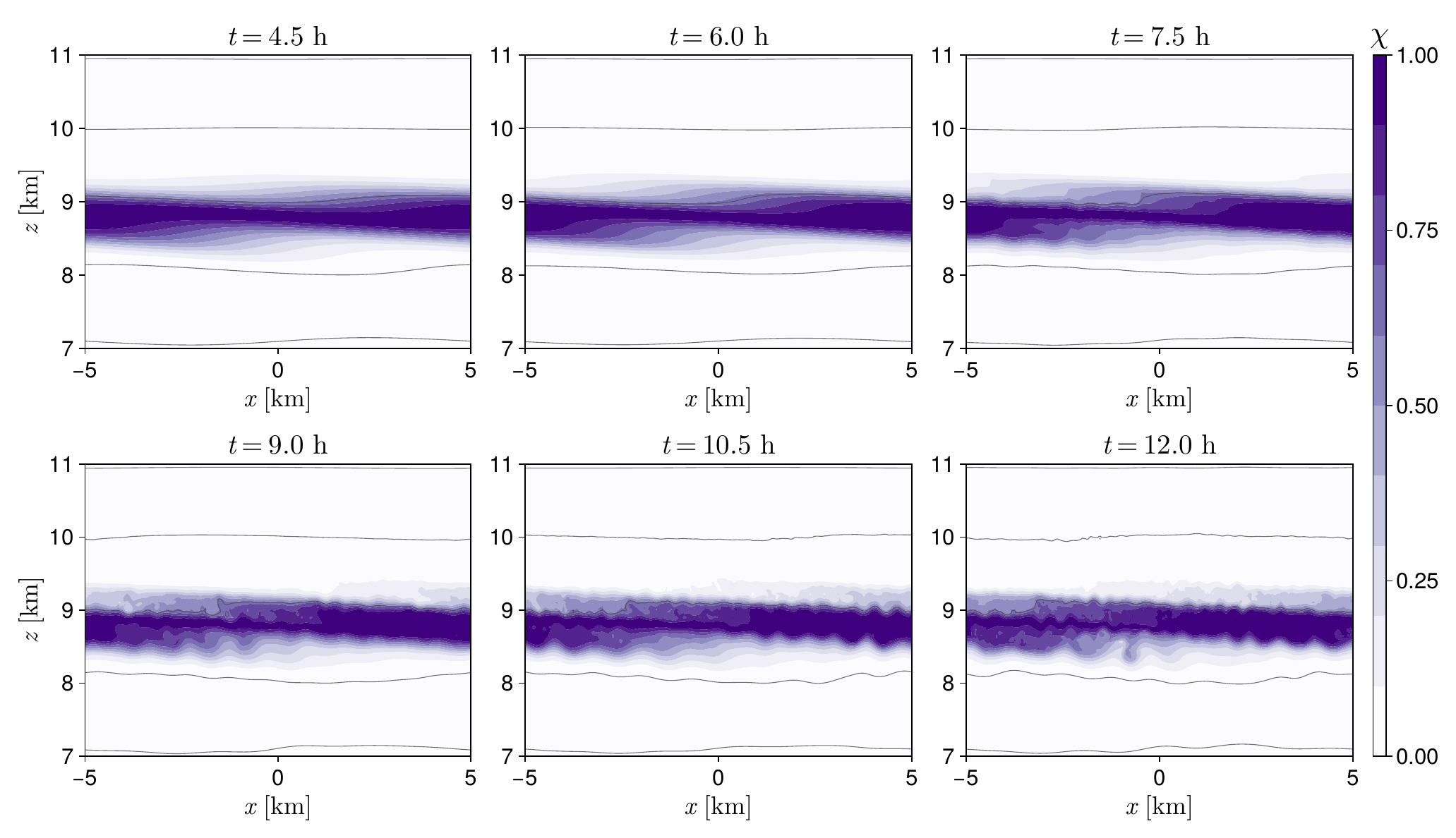}
    \caption{Passive tracer $\chi$ for the critical-layer test case at $t = 4.5,\, 6,\, 7.5,\, 9,\, 10.5$ and \SI{12}{h}, with $80 \times 180$ elements and $p = 3$, for the vector-invariant formulation. Gray lines indicate the isentropic lines.}
    \label{fig:critical_tracer}
\end{figure}

\begin{figure}[tbp]
    \centering
    \includegraphics[width = 0.5\textwidth]{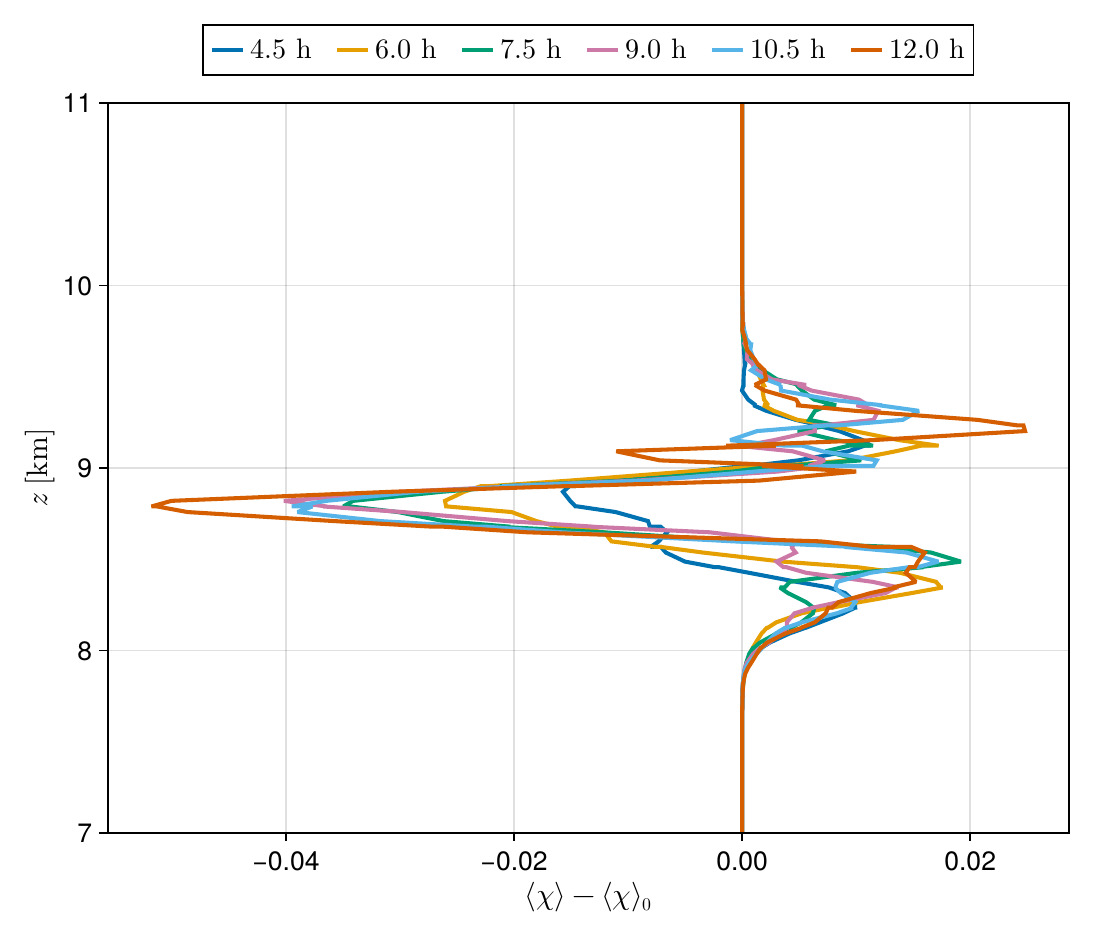}
    \caption{Deviation of the horizontally averaged tracer profile from its initial value for the critical-layer test case at $t = 4.5,\, 6,\, 7.5,\, 9,\, 10.5$ and \SI{12}{h}, with $80 \times 180$ elements and $p = 3$, for the vector-invariant formulation.}
    \label{fig:critical_mixing}
\end{figure}

\subsection{Baroclinic instability}
\label{sec:results_baroclinic}
To further validate the formulation, we employ the baroclinic wave benchmark of Ullrich et al.~\cite{ullrich2014}, adopted in the 2016 Dynamical Core Model Intercomparison Project (DCMIP)~\cite{ullrich2016}.
The initial state is a balanced, axisymmetric solution of the deep-atmosphere equations, perturbed by a localized Gaussian perturbation of the zonal wind in the northern midlatitudes, which triggers the growth of a baroclinic disturbance.
The evolution is predominantly linear up to about day~7 and becomes nonlinear afterward, leading to wave breaking.
The spatial discretization uses an equiangular cubed-sphere grid over six panels in three configurations: a first coarse resolution with $K_h = 8$ elements in the horizontal direction, $K_v = 4$ elements in the vertical direction, and $p = 5$, as in~\cite{artiano2025structurepreservinghighordermethodscompressible}. The remaining two configurations have equal numbers of degrees of freedom, with different polynomial degrees as in~\cite{WARUSZEWSKI2022111507}: $K_h = 30$, $K_v = 8$ elements with $p = 3$, and $K_h = 15$, $K_v = 4$ elements with $p = 7$.
On the right-hand side, a pointwise Coriolis source term $-2\vec{\Omega} \times \vec{V}$ is prescribed.

Time integration uses a horizontally explicit, vertically implicit (HEVI) linear IMEX scheme with the ARK(2,3,2) method of~\cite{giraldo2013numa,GIRALDO2024113275} and $\Delta t = 50,\, 90$ and $180\si{s}$ for $p = 7$, $p = 3$, and $p = 5$, respectively.
Following the approach of Giraldo et al.~\cite{giraldo2013numa, GIRALDO2024113275}, we add and subtract a vertical linear operator.
The Jacobian of the flux-differencing discretization in the vertical direction is computed following the approach of~\cite{CHAN2022110701}.
Figure~\ref{fig:baroclinic_p5} shows contour plots of the surface pressure and of the temperature and relative vorticity at \SI{850}{hPa} on days 8 and 10.
The simulation with the coarse grid gives good results when compared to the more refined grid resolution simulations shown in Figures~\ref{fig:baroclinic_p3} and~\ref{fig:baroclinic_p7}.
The plots are in general in good agreement with the results from the MCore, ENDGame~\cite{ullrich2014}, and MPAS models~\cite{SkamaraockMPAS2021}, as well as the entropy-stable DG scheme from~\cite{WARUSZEWSKI2022111507}, where the vorticity also exhibited grid oscillations due to the minimal numerical dissipation, since the relative vorticity field is more sensitive to numerical dissipation than the pressure and temperature fields~\cite{WARUSZEWSKI2022111507,SkamaraockMPAS2021}.
\begin{figure}[htb]
    \centering
    \includegraphics[width = \textwidth]{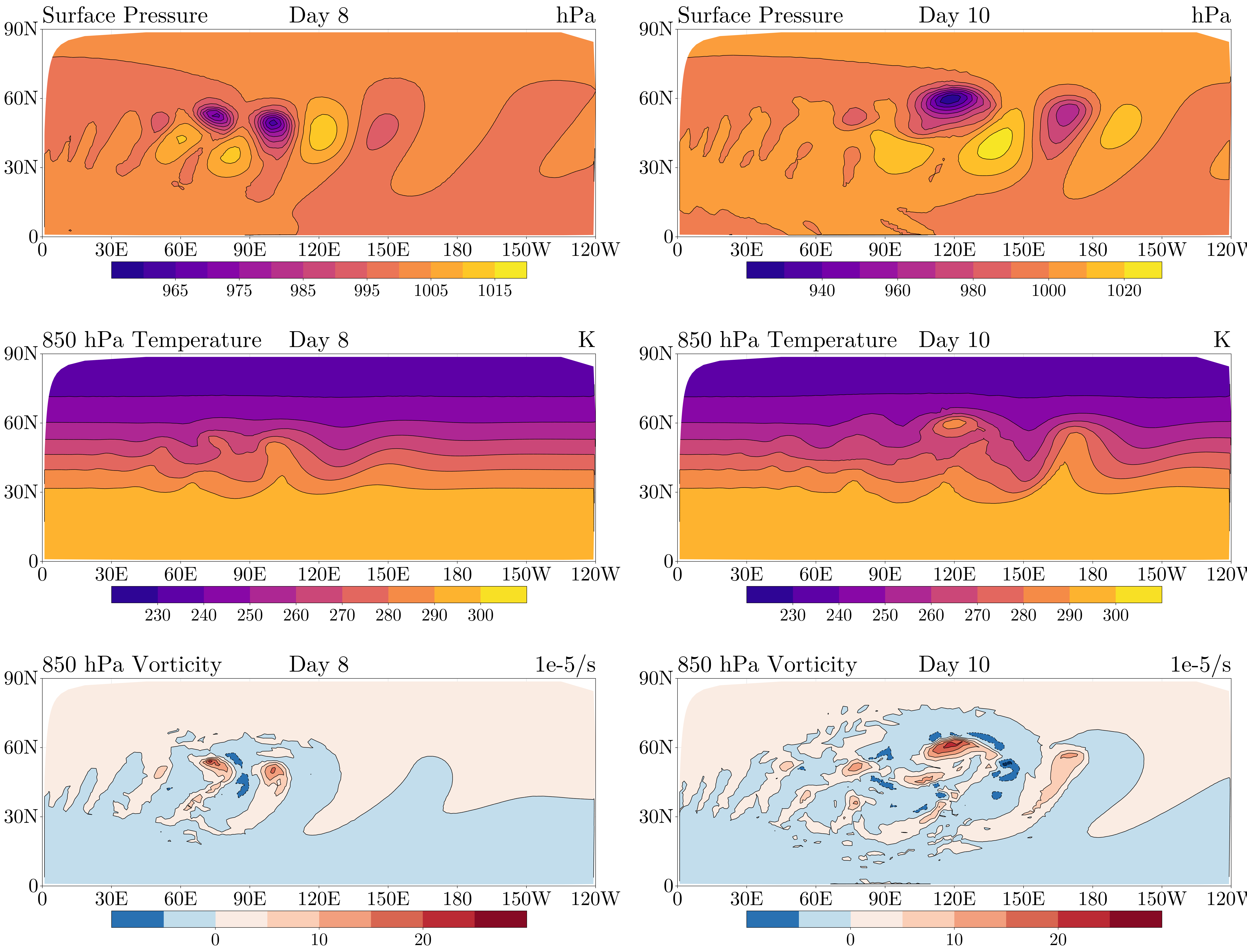}
    \caption{Surface pressure (top), \SI{850}{hPa} temperature (middle) and \SI{850}{hPa} relative vorticity (bottom) at day~8 (left) and day~10 (right) for the baroclinic instability test case, with $K_h = 8$, $K_v = 4$, and polynomial degree $p = 5$.}
    \label{fig:baroclinic_p5}
\end{figure}
\begin{figure}[htb]
    \centering
    \includegraphics[width = \textwidth]{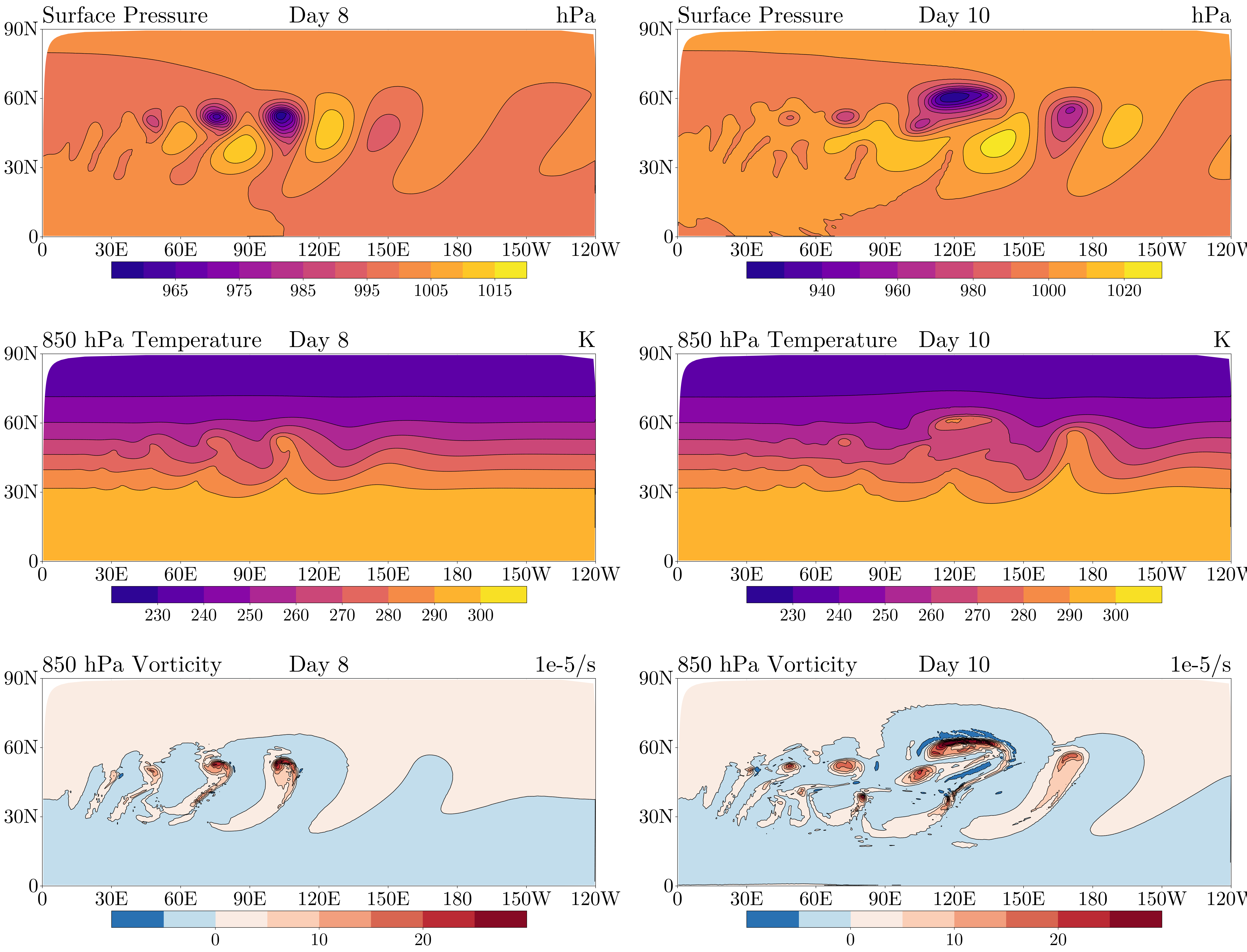}
    \caption{Surface pressure (top), \SI{850}{hPa} temperature (middle) and \SI{850}{hPa} relative vorticity (bottom) at day~8 (left) and day~10 (right) for the baroclinic instability test case, with $K_h = 30$, $K_v = 8$, and polynomial degree $p = 3$.}
    \label{fig:baroclinic_p3}
\end{figure}
\begin{figure}[htb]
    \centering
    \includegraphics[width = \textwidth]{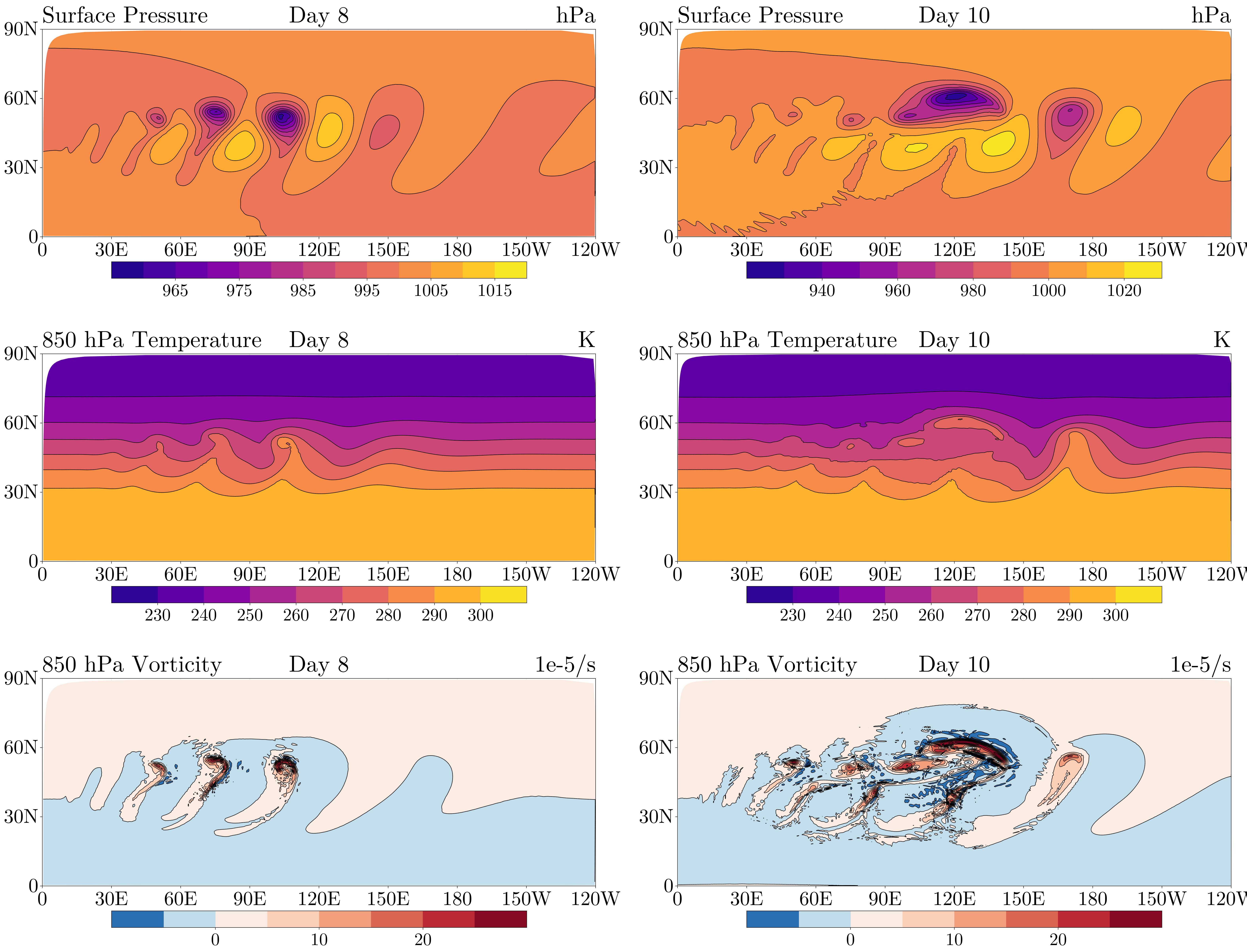}
    \caption{Surface pressure (top), \SI{850}{hPa} temperature (middle) and \SI{850}{hPa} relative vorticity (bottom) at day~8 (left) and day~10 (right) for the baroclinic instability test case, with $K_h = 15$, $K_v = 4$, and polynomial degree $p = 7$.}
    \label{fig:baroclinic_p7}
\end{figure}

Figure~\ref{fig:baroclinic_tseries} shows the evolution of the thermodynamic entropy, the minimum surface pressure and the maximum horizontal wind speed up to day~20, for two discretizations with the same number of degrees of freedom: $K_h = 30$, $K_v = 8$ with $p = 3$ ($\Delta t = \SI{90}{s}$) and $K_h = 15$, $K_v = 4$ with $p = 7$ ($\Delta t = \SI{50}{s}$).
The entropy exhibits minimal dissipation up to day 8 and decreases afterward, when the instability generates underresolved scales, consistent with the entropy stability of the scheme.
At the same number of degrees of freedom, the high-order discretization dissipates less entropy.
The minimum surface pressure and the maximum wind speed agree closely between the two discretizations up to day~15, as well as with the reference solutions in~\cite{WARUSZEWSKI2022111507,SkamaraockMPAS2021,ullrich2014}.
Therefore, the presented new entropy-stable method shows desirable properties of robustness and accuracy also for long-time simulations with very high-order accuracy.
\begin{figure}[htb]
    \centering
    \includegraphics[width = \textwidth]{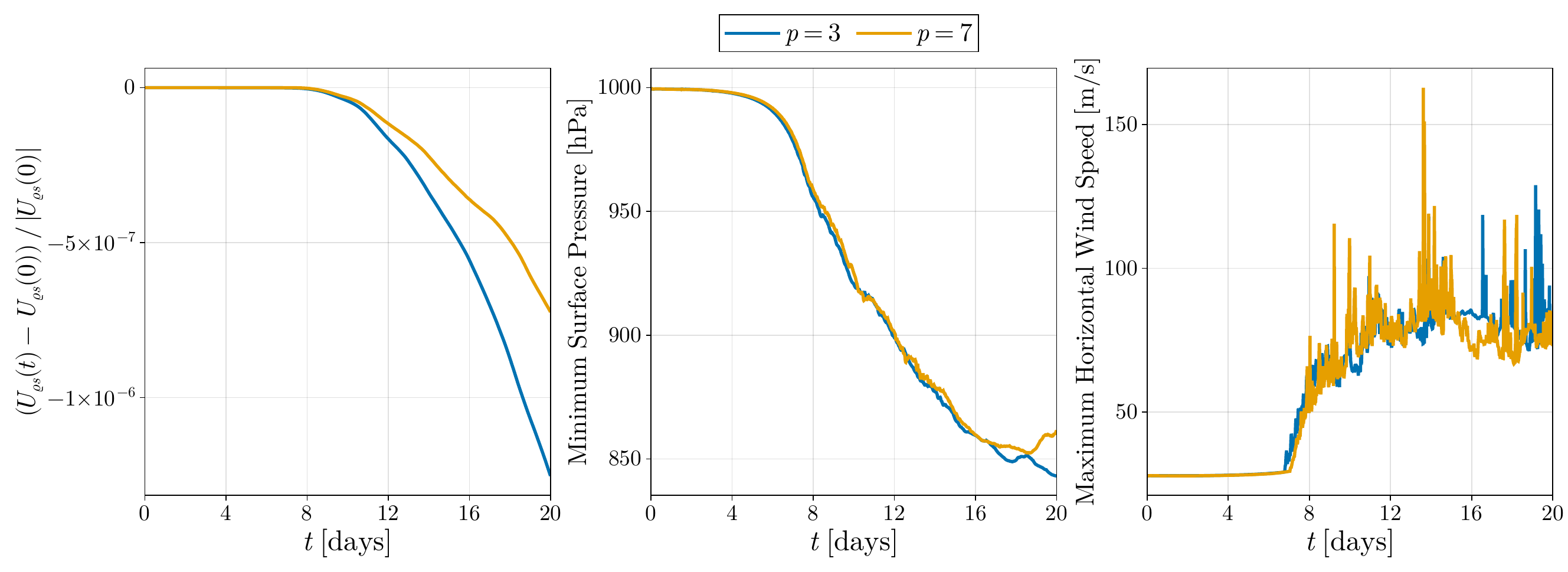}
    \caption{Relative change of the thermodynamic entropy $U_{\varrho s}$ (left), minimum surface pressure (middle) and maximum horizontal wind speed (right) for the baroclinic instability test case up to day~20, with $K_h = 30$, $K_v = 8$, $p = 3$ and $K_h = 15$, $K_v = 4$, $p = 7$.}
    \label{fig:baroclinic_tseries}
\end{figure}

\section{Summary and conclusions}
\label{sec:summary_conclusions}
We presented a structure-preserving flux-differencing discontinuous Galerkin spectral element method (DGSEM) for the compressible Euler equations with gravity in vector-invariant form, using potential temperature as a prognostic variable.
By discretizing the nonconservative vorticity, pressure and geopotential terms as antisymmetric terms, we derived two-point numerical fluxes that are total-energy-conservative (TEC) and entropy-conservative (EC), as well as an entropy-stable (ES) numerical flux that provably preserves both isothermal and isentropic hydrostatic background states.
The ES flux also guarantees the positivity of the density and of the pressure for the first-order finite-volume scheme with strong-stability-preserving Runge--Kutta methods, under a suitable CFL condition.
The main difference with the work of Ricardo et al.~\cite{ricardo2024thermodynamicconsistencystructurepreservationsummation} lies in the choice of the prognostic thermodynamic variable: while they employ the thermodynamic entropy as a linear invariant, we adopt the potential temperature.
As a consequence, their schemes conserve the thermodynamic entropy by construction, and total-energy-conservative/stable fluxes are derived; in our formulation, the potential temperature is conserved by construction, and we derive total-energy-conservative and entropy-conservative/stable fluxes.
Moreover, we proved that the resulting schemes are well-balanced for two families of hydrostatic background states, namely isentropic and isothermal states, on general curvilinear meshes.
Additionally, we have shown that the isentropic background is preserved by any consistent potential-temperature flux, hence by any split form of the pressure gradient written as $c_p \theta \nabla \Pi$. Consequently, the isentropic background is preserved by the scheme of Ricardo et al.~\cite{ricardo2024thermodynamicconsistencystructurepreservationsummation} as recovered in our framework, while the isothermal background requires the inverse logarithmic-mean flux (a kind of generalized nonlinear split-form), recovered via the flux-differencing discretization.
Furthermore, we proved that slip wall boundary conditions are entropy stable for the vector-invariant form.
All these properties carry over, at arbitrary polynomial degree, to high-order SBP methods and, in particular, to the DGSEM considered here on general curvilinear meshes.
The scheme is also efficient, since entropy stability is obtained through a single nonlinear (logarithmic) mean, at a lower cost than total-energy- or potential-temperature-based formulations in conservative form~\cite{WARUSZEWSKI2022111507,artiano2025structurepreservinghighordermethodscompressible}.
To the best of our knowledge, this is the first scheme that is well-balanced for two distinct families of hydrostatic background states simultaneously: the same numerical flux preserves both isothermal and isentropic equilibria, locally and globally, without requiring the background state to be prescribed a priori.
In contrast, existing well-balanced entropy-stable DGSEM schemes preserve a single given background state, selected a priori through the choice of the density mean in the numerical flux~\cite{WARUSZEWSKI2022111507,artiano2025structurepreservinghighordermethodscompressible}.
Numerical experiments, from well-balancedness and inertia-gravity waves to linear hydrostatic/non-hydrostatic mountain flows, as well as breaking waves and critical layer test cases and the baroclinic instability on the cubed sphere, confirmed the theoretical properties and the robustness and accuracy of the method.

\appendix
\section*{Acknowledgments}
MA and HR were supported by the Deutsche Forschungsgemeinschaft
(DFG, German Research Foundation) through project number 528753982
and within the DFG priority program SPP~2410 through project number 526031774.
This work was also supported by the Max Planck Graduate Center with the
Johannes Gutenberg University of Mainz (MPGC) and
the Mainz Institute of Multiscale Modeling (M\textsuperscript{3}ODEL).

\section*{Tool disclosure}
Claude Opus 5 was used to assist with software development and to improve the readability and language.

\section*{Data availability statement}

All Julia source code and data needed to reproduce the numerical results
presented in this paper are available in our reproducibility repository~\cite{artiano2026vectorinvariantRepro}.

%\bibliographystyle{abbrvnat}
%\bibliographystyle{plain}
%\bibliography{refs}
\printbibliography

@misc{artiano2026vectorinvariantRepro,
  title        = {Reproducibility repository for
                  "Entropy-Stable and Well-Balanced Discontinuous Galerkin Methods for the Compressible Euler Equations in Vector-Invariant Form"},
  author       = {Artiano, Marco and Ranocha, Hendrik},
  year         = {2026},
  howpublished = {\url{https://github.com/MarcoArtiano/2026_entropy_stable_vector_invariant}},
  doi          = {10.5281/zenodo.23144288}
}

@misc{artiano2026affordable,
  title         = {On Affordable High-Order Entropy-Conservative/Stable and
                   Well-Balanced Methods for Nonconservative Hyperbolic Systems},
  author        = {Artiano, Marco and Ranocha, Hendrik},
  year          = {2026},
  doi = {10.48550/arXiv.2603.18978},
  archiveprefix = {arXiv},
  primaryclass  = {math.NA},
  eprint = {2603.18978}
}

@misc{waterhouse2026gpuperformanceentropystablediscontinuous,
  title         = {GPU Performance of an Entropy-Stable Discontinuous Galerkin Euler Solver with Non-Conservative Terms},
  author        = {Henry Waterhouse and Maciej Waruszewski and Lucas C. Wilcox and Francis X. Giraldo},
  year          = {2026},
  archiveprefix = {arXiv},
  primaryclass  = {math.NA},
  doi={10.48550/arXiv.2605.16684},
  eprint = {2605.16684}
}

@article{ricardo2024thermodynamicconsistencystructurepreservationsummation,
  author   = {Ricardo, Kieran and Lee, David and Duru, Kenneth},
  title    = {Thermodynamic Consistency and Structure-Preservation in Summation by Parts Methods for the Moist Compressible Euler Equations},
  journal  = {SIAM Journal on Scientific Computing},
  volume   = {48},
  number   = {3},
  pages    = {B517-B543},
  year     = {2026},
  doi      = {10.1137/25M174718X}
}

@article{klose2020assessing,
  title     = {Assessing standard and kinetic energy conserving volume fluxes in
               discontinuous {G}alerkin formulations for marginally resolved
               {N}avier-{S}tokes flows},
  author    = {Klose, Bjoern F and Jacobs, Gustaaf B and Kopriva, David A},
  journal   = {Computers \& Fluids},
  volume    = {205},
  pages     = {104557},
  year      = {2020},
  publisher = {Elsevier},
  doi       = {10.1016/j.compfluid.2020.104557}
}

@article{ullrich2014,
  author   = {Ullrich, Paul A. and Melvin, Thomas and Jablonowski, Christiane and Staniforth, Andrew},
  title    = {A proposed baroclinic wave test case for deep- and shallow-atmosphere dynamical cores},
  journal  = {Quarterly Journal of the Royal Meteorological Society},
  volume   = {140},
  number   = {682},
  pages    = {1590-1602},
  doi      = {10.1002/qj.2241},
  url      = {https://rmets.onlinelibrary.wiley.com/doi/abs/10.1002/qj.2241},
  year     = {2014}
}

@article{Hicken2016_SBP,
  author   = {Hicken, Jason E. and Del Rey Fern\'{a}ndez, David C. and Zingg, David W.},
  title    = {Multidimensional Summation-by-Parts Operators: General Theory and Application to Simplex Elements},
  journal  = {SIAM Journal on Scientific Computing},
  volume   = {38},
  number   = {4},
  pages    = {A1935-A1958},
  year     = {2016},
  doi      = {10.1137/15M1038360},
  url      = {

              10.1137/15M1038360



              }
}

@article{artiano2025structurepreservinghighordermethodscompressible,
title = {Structure-Preserving High-Order Methods for the Compressible Euler Equations in Potential Temperature Formulation for Atmospheric Flows},
journal = {Journal of Computational Physics},
pages = {115221},
year = {2026},
issn = {0021-9991},
doi = {10.1016/j.jcp.2026.115221},
author = {Marco Artiano and Oswald Knoth and Peter Spichtinger and Hendrik Ranocha},
  volume = {566}
}

@article{RENAC20191,
  title    = {Entropy stable {DGSEM} for nonlinear hyperbolic systems in nonconservative form with application to two-phase flows},
  journal  = {Journal of Computational Physics},
  volume   = {382},
  pages    = {1-26},
  year     = {2019},
  issn     = {0021-9991},
  doi      = {10.1016/j.jcp.2018.12.035},
  url      = {https://www.sciencedirect.com/science/article/pii/S002199911930021X},
  author   = {Florent Renac}
}

@article{Castro2013,
  author  = {Castro, Manuel J. and Fjordholm, Ulrik S. and Mishra, Siddhartha and Par{\'e}s, Carlos},
  year    = {2013},
  month   = {05},
  pages   = {1371-1391},
  title   = {Entropy Conservative and Entropy Stable Schemes for Nonconservative Hyperbolic Systems},
  volume  = {51},
  journal = {SIAM Journal on Numerical Analysis},
  doi     = {10.1137/110845379},
  number = {3}
}

@article{Tadmor2003,
  author  = {E. Tadmor},
  title   = {Entropy stability theory for difference approximations of nonlinear conservation laws and related time-dependent problems},
  journal = {Acta Numerica},
  volume  = {12},
  year    = {2003},
  pages   = {451--512},
  doi     = {10.1017/S0962492902000156}
}

@inproceedings{CarpenterBook,
  author    = {Mark H. Carpenter and Matteo Parsani and Eric J. Nielsen and Travis C. Fisher},
  title     = {Towards an Entropy Stable Spectral Element Framework for Computational Fluid Dynamics},
  booktitle = {54th {AIAA} Aerospace Sciences Meeting},
  doi       = {10.2514/6.2016-1058},
  url       = {https://arc.aiaa.org/doi/abs/10.2514/6.2016-1058},
  year = {2016},
  publisher = {American Institute of Aeronautics and Astronautics}
}

@article{winters2018comparative,
  title     = {A comparative study on polynomial dealiasing and split form
               discontinuous {G}alerkin schemes for under-resolved turbulence
               computations},
  author    = {Winters, Andrew R and Moura, Rodrigo C and Mengaldo, Gianmarco and
               Gassner, Gregor J and Walch, Stefanie and Peiro, Joaquim and
               Sherwin, Spencer J},
  journal   = {Journal of Computational Physics},
  volume    = {372},
  pages     = {1--21},
  year      = {2018},
  publisher = {Elsevier},
  doi       = {10.1016/j.jcp.2018.06.016}
}

@article{Carpenter2014,
  author  = {Carpenter, Mark H. and Fisher, Travis C. and Nielsen, Eric J. and Frankel, Steven H.},
  title   = {Entropy Stable Spectral Collocation Schemes for the {N}avier--{S}tokes Equations: Discontinuous Interfaces},
  journal = {SIAM Journal on Scientific Computing},
  volume  = {36},
  number  = {5},
  pages   = {B835-B867},
  year    = {2014},
  doi     = {10.1137/130932193}
}

@article{FISHER2013518,
  title    = {High-order entropy stable finite difference schemes for nonlinear conservation laws: Finite domains},
  journal  = {Journal of Computational Physics},
  volume   = {252},
  pages    = {518-557},
  year     = {2013},
  issn     = {0021-9991},
  doi      = {10.1016/j.jcp.2013.06.014},
  author   = {Travis C. Fisher and Mark H. Carpenter}
}

@article{LIU2025114095,
  title    = {Structure-preserving nodal {DG} method for the {E}uler equations with gravity: well-balanced, entropy stable, and positivity preserving},
  journal  = {Journal of Computational Physics},
  volume   = {537},
  pages    = {114095},
  year     = {2025},
  issn     = {0021-9991},
  doi      = {10.1016/j.jcp.2025.114095},
  author   = {Yuchang Liu and Wei Guo and Yan Jiang and Mengping Zhang}
}

@article{MENGALDO201556,
  title    = {Dealiasing techniques for high-order spectral element methods on regular and irregular grids},
  journal  = {Journal of Computational Physics},
  volume   = {299},
  pages    = {56-81},
  year     = {2015},
  issn     = {0021-9991},
  doi      = {10.1016/j.jcp.2015.06.032},
  url      = {https://www.sciencedirect.com/science/article/pii/S0021999115004301},
  author   = {G. Mengaldo and D. {De Grazia} and D. Moxey and P. E. Vincent and S. J. Sherwin}
}

@article{CiCP-27-5,
  author   = {Yu, Jian and Hesthaven, Jan S.},
  title    = {A Study of Several Artificial Viscosity Models within the Discontinuous {G}alerkin Framework},
  journal  = {Communications in Computational Physics},
  year     = {2020},
  volume   = {27},
  number   = {5},
  pages    = {1309--1343},
  issn     = {1991-7120},
  doi      = {10.4208/cicp.OA-2019-0118},
  url      = {https://global-sci.com/article/79806/a-study-of-several-artificial-viscosity-models-within-the-discontinuous-galerkin-framework}
}

@article{ULLRICH2018427,
  title    = {Impact and importance of hyperdiffusion on the spectral element method: A linear dispersion analysis},
  journal  = {Journal of Computational Physics},
  volume   = {375},
  pages    = {427-446},
  year     = {2018},
  issn     = {0021-9991},
  doi      = {10.1016/j.jcp.2018.06.035},
  url      = {https://www.sciencedirect.com/science/article/pii/S0021999118304133},
  author   = {Paul A. Ullrich and Daniel R. Reynolds and Jorge E. Guerra and Mark A. Taylor}
}

@article{Kopriva2006,
  author   = {David A. Kopriva},
  title    = {Metric Identities and the Discontinuous Spectral Element Method on Curvilinear Meshes},
  journal  = {Journal of Scientific Computing},
  year     = {2006},
  volume   = {26},
  number   = {3},
  pages    = {301--327},
  doi      = {10.1007/s10915-005-9070-8},
  url      = {10.1007/s10915-005-9070-8},
  issn     = {1573-7691}
}

@article{Tadmor1987,
  author  = {E. Tadmor},
  title   = {The numerical viscosity of entropy stable schemes for systems of conservation laws. {I}},
  journal = {Mathematics of Computation},
  volume  = {49},
  number  = {179},
  year    = {1987},
  pages   = {91--103},
  doi     = {10.1090/S0025-5718-1987-0890255-3}
}

@article{DERIGS2017624,
  title    = {A novel averaging technique for discrete entropy-stable dissipation operators for ideal {MHD}},
  journal  = {Journal of Computational Physics},
  volume   = {330},
  pages    = {624-632},
  year     = {2017},
  issn     = {0021-9991},
  doi      = {10.1016/j.jcp.2016.10.055},
  url      = {https://www.sciencedirect.com/science/article/pii/S0021999116305630},
  author   = {Dominik Derigs and Andrew R. Winters and Gregor J. Gassner and Stefanie Walch}
}

@article{ismail2009affordable,
  title     = {Affordable, entropy-consistent {E}uler flux functions {II}:
               {E}ntropy production at shocks},
  author    = {Ismail, Farzad and Roe, Philip L},
  journal   = {Journal of Computational Physics},
  volume    = {228},
  number    = {15},
  pages     = {5410--5436},
  year      = {2009},
  publisher = {Elsevier},
  doi       = {10.1016/j.jcp.2009.04.021}
}

@article{Ranocha2022Optimal,
  author  = {Hendrik Ranocha and Lisandro Dalcin and Matteo Parsani and David I. Ketcheson},
  title   = {Optimized Runge-Kutta Methods with Automatic Step Size Control for Compressible Computational Fluid Dynamics},
  journal = {Communications on Applied Mathematics and Computation},
  year    = {2022},
  volume  = {4},
  number  = {4},
  pages   = {1191--1228},
  doi     = {10.1007/s42967-021-00159-w},
  issn    = {2661-8893}
}

@article{gassner2016split,
  title     = {Split Form Nodal Discontinuous {G}alerkin Schemes with
               Summation-By-Parts Property for the Compressible {E}uler
               Equations},
  author    = {Gassner, Gregor Josef and Winters, Andrew Ross and
               Kopriva, David A},
  journal   = {Journal of Computational Physics},
  volume    = {327},
  pages     = {39--66},
  year      = {2016},
  publisher = {Elsevier},
  doi       = {10.1016/j.jcp.2016.09.013}
}

@article{Souza,
  author   = {Souza, A. N. and He, J. and Bischoff, T. and Waruszewski, M. and Novak, L. and Barra, V. and Gibson, T. and Sridhar, A. and Kandala, S. and Byrne, S. and Wilcox, L. C. and Kozdon, J. and Giraldo, F. X. and Knoth, O. and Marshall, J. and Ferrari, R. and Schneider, T.},
  title    = {The Flux-Differencing Discontinuous {G}alerkin Method Applied to an Idealized Fully Compressible Nonhydrostatic Dry Atmosphere},
  journal  = {Journal of Advances in Modeling Earth Systems},
  volume   = {15},
  number   = {4},
  pages    = {e2022MS003527},
  doi      = {10.1029/2022MS003527},
  url      = {https://agupubs.onlinelibrary.wiley.com/doi/abs/10.1029/2022MS003527},
  year     = {2023}
}

@article{SkamaraockMPAS2021,
  author    = {William C. Skamarock and Hing Ong and Joseph B. Klemp},
  title     = {A Fully Compressible Nonhydrostatic Deep-Atmosphere Equations Solver for MPAS},
  journal   = {Monthly Weather Review},
  year      = {2021},
  publisher = {American Meteorological Society},
  address   = {Boston MA, USA},
  volume    = {149},
  number    = {2},
  doi       = {10.1175/MWR-D-20-0286.1},
  pages     = {571--583},
  url       = {https://journals.ametsoc.org/view/journals/mwre/149/2/MWR-D-20-0286.1.xml}
}

@article{WARUSZEWSKI2022111507,
  title    = {Entropy stable discontinuous {G}alerkin methods for balance laws in non-conservative form: Applications to the {E}uler equations with gravity},
  journal  = {Journal of Computational Physics},
  volume   = {468},
  pages    = {111507},
  year     = {2022},
  issn     = {0021-9991},
  doi      = {10.1016/j.jcp.2022.111507},
  url      = {https://www.sciencedirect.com/science/article/pii/S0021999122005691},
  author   = {Maciej Waruszewski and Jeremy E. Kozdon and Lucas C. Wilcox and Thomas H. Gibson and Francis X. Giraldo}
}

@article{taylor2019,
  author   = {Taylor, Mark A. and Guba, Oksana and Steyer, Andrew and Ullrich, Paul A. and Hall, David M. and Eldred, Christopher},
  title    = {An Energy Consistent Discretization of the Nonhydrostatic Equations in Primitive Variables},
  journal  = {Journal of Advances in Modeling Earth Systems},
  volume   = {12},
  number   = {1},
  pages    = {e2019MS001783},
  doi      = {10.1029/2019MS001783},
  year     = {2020}
}

@article{baldauf2013,
  author   = {Baldauf, Michael and Brdar, Slavko},
  title    = {An analytic solution for linear gravity waves in a channel as a test for numerical models using the non-hydrostatic, compressible {E}uler equations},
  journal  = {Quarterly Journal of the Royal Meteorological Society},
  volume   = {139},
  number   = {677},
  pages    = {1977-1989},
  doi      = {10.1002/qj.2105},
  url      = {https://rmets.onlinelibrary.wiley.com/doi/abs/10.1002/qj.2105},
  year     = {2013}
}

@article{blaise2016,
  author   = {Blaise, Sébastien and Lambrechts, Jonathan and Deleersnijder, Eric},
  title    = {A stabilization for three-dimensional discontinuous {G}alerkin discretizations applied to nonhydrostatic atmospheric simulations},
  journal  = {International Journal for Numerical Methods in Fluids},
  volume   = {81},
  number   = {9},
  pages    = {558-585},
  doi      = {10.1002/fld.4197},
  url      = {https://onlinelibrary.wiley.com/doi/abs/10.1002/fld.4197},
  year     = {2016}
}

@article{BALDAUF2021110635,
  title    = {A horizontally explicit, vertically implicit ({HEVI}) discontinuous {G}alerkin scheme for the 2-dimensional {E}uler and {N}avier-{S}tokes equations using terrain-following coordinates},
  journal  = {Journal of Computational Physics},
  volume   = {446},
  pages    = {110635},
  year     = {2021},
  issn     = {0021-9991},
  doi      = {10.1016/j.jcp.2021.110635},
  url      = {https://www.sciencedirect.com/science/article/pii/S0021999121005301},
  author   = {Michael Baldauf}
}

@article{GIRALDO20083849,
  title    = {A study of spectral element and discontinuous {G}alerkin methods for the Navier–Stokes equations in nonhydrostatic mesoscale atmospheric modeling: Equation sets and test cases},
  journal  = {Journal of Computational Physics},
  volume   = {227},
  number   = {8},
  pages    = {3849-3877},
  year     = {2008},
  issn     = {0021-9991},
  doi      = {10.1016/j.jcp.2007.12.009},
  url      = {https://www.sciencedirect.com/science/article/pii/S0021999107005384},
  author   = {F. X. Giraldo and M. Restelli}
}

@article{ORLANDO2023115124,
  title    = {An {IMEX}-{DG} solver for atmospheric dynamics simulations with adaptive mesh refinement},
  journal  = {Journal of Computational and Applied Mathematics},
  volume   = {427},
  pages    = {115124},
  year     = {2023},
  issn     = {0377-0427},
  doi      = {10.1016/j.cam.2023.115124},
  url      = {https://www.sciencedirect.com/science/article/pii/S0377042723000687},
  author   = {Giuseppe Orlando and Tommaso Benacchio and Luca Bonaventura}
}

@article{ChandrashekarWB,
  author   = {Chandrashekar, Praveen and Klingenberg, Christian},
  title    = {A Second Order Well-Balanced Finite Volume Scheme for {E}uler Equations with Gravity},
  journal  = {SIAM Journal on Scientific Computing},
  volume   = {37},
  number   = {3},
  pages    = {B382-B402},
  year     = {2015},
  doi      = {10.1137/140984373}
}

@article{GIRFOGLIO2025106510,
  title    = {A comparative computational study of different formulations of the compressible {E}uler equations for mesoscale atmospheric flows in a finite volume framework},
  journal  = {Computers \& Fluids},
  volume   = {288},
  pages    = {106510},
  year     = {2025},
  issn     = {0045-7930},
  doi      = {10.1016/j.compfluid.2024.106510},
  author   = {M. Girfoglio and A. Quaini and G. Rozza}
}

@article{SimpleTestsofaSemiImplicitSemiLagrangianModelon2DMountainWaveProblems,
  author    = {Jean-Pierre  Pinty and Robert  Benoit and Evelyne  Richard and René  Laprise},
  title     = {Simple Tests of a Semi-Implicit Semi-{L}agrangian Model on 2D Mountain Wave Problems},
  journal   = {Monthly Weather Review},
  year      = {1995},
  publisher = {American Meteorological Society},
  address   = {Boston MA, USA},
  volume    = {123},
  number    = {10},
  doi       = {10.1175/1520-0493(1995)123<3042:STOASI>2.0.CO;2},
  pages     = {3042--3058},
  url       = {https://journals.ametsoc.org/view/journals/mwre/123/10/1520-0493_1995_123_3042_stoasi_2_0_co_2.xml}
}

@article{Smith1979TheIO,
  title   = {The Influence of Mountains on the Atmosphere},
  author  = {Ronald B. Smith},
  journal = {Advances in Geophysics},
  year    = {1979},
  volume  = {21},
  pages   = {87-230},
  doi = {10.1016/S0065-2687(08)60262-9}
}

@article{AnUpperBoundaryConditionPermittingInternalGravityWaveRadiationinNumericalMesoscaleModels,
  author    = {Joseph B.  Klemp and Dale R.  Durran},
  title     = {An Upper Boundary Condition Permitting Internal Gravity Wave Radiation in Numerical Mesoscale Models},
  journal   = {Monthly Weather Review},
  year      = {1983},
  publisher = {American Meteorological Society},
  address   = {Boston MA, USA},
  volume    = {111},
  number    = {3},
  doi       = {10.1175/1520-0493(1983)111<0430:AUBCPI>2.0.CO;2},
  pages     = {430--444},
  url       = {https://journals.ametsoc.org/view/journals/mwre/111/3/1520-0493_1983_111_0430_aubcpi_2_0_co_2.xml}
}

@article{gassner2013skew,
  title     = {A Skew-Symmetric Discontinuous {G}alerkin Spectral Element
               Discretization and Its Relation to {SBP}-{SAT} Finite Difference
               Methods},
  author    = {Gassner, Gregor Josef},
  journal   = {SIAM Journal on Scientific Computing},
  volume    = {35},
  number    = {3},
  pages     = {A1233--A1253},
  year      = {2013},
  publisher = {Society for Industrial and Applied Mathematics},
  doi       = {10.1137/120890144}
}

@article{fjordholm2012arbitrarily,
  title={Arbitrarily High-Order Accurate Entropy Stable Essentially
         Nonoscillatory Schemes for Systems of Conservation Laws},
  author={Fjordholm, Ulrik Skre and Mishra, Siddhartha and
          Tadmor, Eitan},
  journal={SIAM Journal on Numerical Analysis},
  volume={50},
  number={2},
  pages={544--573},
  year={2012},
  publisher={Society for Industrial and Applied Mathematics},
  doi={10.1137/110836961}
}

@article{VisbalGaitonde1999,
  author  = {Visbal, Miguel R. and Gaitonde, Datta V.},
  title   = {High-Order-Accurate Methods for Complex Unsteady Subsonic Flows},
  journal = {AIAA Journal},
  volume  = {37},
  number  = {10},
  pages   = {1231-1239},
  year    = {1999},
  doi     = {10.2514/2.591}
}

@inbook{Kopriva2009,
  author    = {Kopriva, David A.},
  title     = {Spectral Element Methods},
  booktitle = {Implementing Spectral Methods for Partial Differential Equations: Algorithms for Scientists and Engineers},
  year      = {2009},
  publisher = {Springer Netherlands},
  address   = {Dordrecht},
  pages     = {293--354},
  isbn      = {978-90-481-2261-5},
  doi       = {10.1007/978-90-481-2261-5_8}
}

@article{WINTERS2017274,
  title    = {A uniquely defined entropy stable matrix dissipation operator for high Mach number ideal {MHD} and compressible {E}uler simulations},
  journal  = {Journal of Computational Physics},
  volume   = {332},
  pages    = {274-289},
  year     = {2017},
  issn     = {0021-9991},
  doi      = {10.1016/j.jcp.2016.12.006},
  author   = {Andrew R. Winters and Dominik Derigs and Gregor J. Gassner and Stefanie Walch}
}

@misc{ullrich2016,
  author       = {Ullrich, Paul A. and Jablonowski, Christiane and Reed, Kevin A. and Zarzycki, Colin and Lauritzen, Peter H. and Nair, Ramachandran D. and Kent, James and Verlet-Banide, Antonin},
  title        = {Dynamical Core Model Intercomparison Project ({DCMIP2016}) Test Case Document},
  year         = {2016},
  howpublished = {Technical report. \url{https://github.com/ClimateGlobalChange/DCMIP2016}},
  url = {https://github.com/ClimateGlobalChange/DCMIP2016/blob/master/DCMIP2016-TestCaseDocument_v1.pdf},
  note = {DCMIP Summer School, June 2016}
}

@article{Winters2020,
  author   = {Andrew R. Winters and Christof Czernik and Moritz B. Schily and Gregor J. Gassner},
  title    = {Entropy stable numerical approximations for the isothermal and polytropic {E}uler equations},
  journal  = {BIT Numerical Mathematics},
  year     = {2020},
  volume   = {60},
  number   = {3},
  pages    = {791--824},
  doi      = {10.1007/s10543-019-00789-w},
  url      = {10.1007/s10543-019-00789-w}
}

@article{AControlVolumeModeloftheCompressibleEulerEquationswithaVerticalLagrangianCoordinate,
  author    = {Chen, Xi and Andronova, Natalia and {Van Leer}, Bram and Penner, Joyce E. and Boyd, John P. and Jablonowski, Christiane and Lin, Shian-Jiann},
  title     = {A Control-Volume Model of the Compressible {E}uler Equations with a Vertical {L}agrangian Coordinate},
  journal   = {Monthly Weather Review},
  year      = {2013},
  publisher = {American Meteorological Society},
  address   = {Boston MA, USA},
  volume    = {141},
  number    = {7},
  doi       = {10.1175/MWR-D-12-00129.1},
  pages     = {2526--2544},
  url       = {https://journals.ametsoc.org/view/journals/mwre/141/7/mwr-d-12-00129.1.xml}
}

@article{chandrashekar2013,
  author  = {Praveen Chandrashekar},
  title   = {Kinetic Energy Preserving and Entropy Stable Finite Volume Schemes for Compressible Euler and Navier-Stokes Equations},
  journal = {Communications in Computational Physics},
  year    = {2013},
  volume  = {14},
  number  = {5},
  pages   = {1252-1286},
  doi     = {10.4208/cicp.170712.010313a},
  url     = {10.4208/cicp.170712.010313a}
}

@article{CHAN2022110701,
title = {Efficient computation of Jacobian matrices for entropy stable summation-by-parts schemes},
journal = {Journal of Computational Physics},
volume = {448},
pages = {110701},
year = {2022},
issn = {0021-9991},
doi = {10.1016/j.jcp.2021.110701},
url = {https://www.sciencedirect.com/science/article/pii/S0021999121005969},
author = {Jesse Chan and Christina G. Taylor}
}

@article{GIRALDO2024113275,
  title    = {A performance study of horizontally explicit vertically implicit (HEVI) time-integrators for non-hydrostatic atmospheric models},
  journal  = {Journal of Computational Physics},
  volume   = {515},
  pages    = {113275},
  year     = {2024},
  issn     = {0021-9991},
  doi      = {10.1016/j.jcp.2024.113275},
  url      = {https://www.sciencedirect.com/science/article/pii/S0021999124005230},
  author   = {Francis X. Giraldo and Felipe A. V. {de Bragança Alves} and James F. Kelly and Soonpil Kang and P. Alex Reinecke}
}

@article{giraldo2013numa,
  author   = {Giraldo, F. X. and Kelly, J. F. and Constantinescu, E. M.},
  title    = {Implicit-Explicit Formulations of a Three-Dimensional Nonhydrostatic Unified Model of the Atmosphere (NUMA)},
  journal  = {SIAM Journal on Scientific Computing},
  volume   = {35},
  number   = {5},
  pages    = {B1162-B1194},
  year     = {2013},
  doi      = {10.1137/120876034}
}

@article{NUMA,
  author   = {Andreas Müller and Michal A Kopera and Simone Marras and Lucas C Wilcox and Tobin Isaac and Francis X Giraldo},
  title    = {Strong scaling for numerical weather prediction at petascale with the atmospheric model {NUMA}},
  journal  = {The International Journal of High Performance Computing Applications},
  volume   = {33},
  number   = {2},
  pages    = {411-426},
  year     = {2019},
  doi      = {10.1177/1094342018763966}
}

@article{ACompressibleModelfortheSimulationofMoistMountainWaves,
  author    = {Dale R.  Durran and Joseph B.  Klemp},
  title     = {A Compressible Model for the Simulation of Moist Mountain Waves},
  journal   = {Monthly Weather Review},
  year      = {1983},
  publisher = {American Meteorological Society},
  address   = {Boston MA, USA},
  volume    = {111},
  number    = {12},
  doi       = {10.1175/1520-0493(1983)111<2341:ACMFTS>2.0.CO;2},
  pages     = {2341--2361},
  url       = {https://journals.ametsoc.org/view/journals/mwre/111/12/1520-0493_1983_111_2341_acmfts_2_0_co_2.xml}
}

@article{MeansGenerated,
  issn      = {0025-570X},
  url       = {http://www.jstor.org/stable/30044201},
  author    = {Hongwei Chen},
  journal   = {Mathematics Magazine},
  number    = {5},
  pages     = {397--399},
  publisher = {Mathematical Association of America},
  title     = {Means Generated by an Integral},
  urldate   = {2025-08-18},
  volume    = {78},
  year      = {2005},
  doi       = {10.2307/30044201}
}

@book{Gottliebs2011,
  title     = {Strong stability preserving {R}unge-{K}utta and multistep time
               discretizations},
  author    = {Gottlieb, Sigal and Ketcheson, David I and Shu, Chi-Wang},
  year      = {2011},
  publisher = {World Scientific},
  address   = {Singapore},
  doi = {10.1142/7498}
}

@inbook{Barth1999,
  author    = {Barth, Timothy J.},
  editor    = {Kr{\"o}ner, Dietmar
               and Ohlberger, Mario
               and Rohde, Christian},
  title     = {Numerical Methods for Gasdynamic Systems on Unstructured Meshes},
  booktitle = {An Introduction to Recent Developments in Theory and Numerics for Conservation Laws: Proceedings of the International School on Theory and Numerics for Conservation Laws, Freiburg/Littenweiler, October 20--24, 1997},
  year      = {1999},
  publisher = {Springer Berlin Heidelberg},
  address   = {Berlin, Heidelberg},
  pages     = {195--285},
  isbn      = {978-3-642-58535-7},
  doi       = {10.1007/978-3-642-58535-7_5}
}

@article{sjogreen2018high,
  title     = {High order entropy conservative central schemes for wide ranges of
               compressible gas dynamics and {MHD} flows},
  author    = {Sj{\"o}green, Bj{\"o}rn and Yee, H. C.},
  journal   = {Journal of Computational Physics},
  volume    = {364},
  pages     = {153--185},
  year      = {2018},
  publisher = {Elsevier},
  doi       = {10.1016/j.jcp.2018.02.003}
}

@article{ranocha2023efficient,
  title   = {Efficient implementation of modern entropy stable and kinetic energy
             preserving discontinuous {G}alerkin methods for conservation laws},
  author  = {Ranocha, Hendrik and Schlottke-Lakemper, Michael and Chan, Jesse and Rueda-Ram{\'i}rez, Andr{\'e}s M. and Winters, Andrew R. and Hindenlang, Florian and Gassner, Gregor J.},
  journal = {ACM Transactions on Mathematical Software},
  year    = {2023},
  month   = {12},
  volume  = {49},
  doi     = {10.1145/3625559},
  number = {4},
  pages = {1--30}
}

@article{choi2016,
  author        = {Choi, Suk-Jin and Hong, Song-You},
  doi           = {10.1007/s13143-016-0005-0},
  id            = {Choi2016},
  journal       = {Asia-Pacific Journal of Atmospheric Sciences},
  number        = {3},
  pages         = {291--307},
  title         = {A global non-hydrostatic dynamical core using the spectral element method on a cubed-sphere grid},
  volume        = {52},
  year          = {2016},
  issn = {1976-7951}
}

@article{baldaufbridge2026,
  author  = {Baldauf, Michael and Prill, Florian},
  title   = {Comparison of two {E}uler equation sets in a Discontinuous {G}alerkin solver for atmospheric modelling ({BRIDGE} v0.9)},
  journal = {Geoscientific Model Development},
  volume  = {19},
  year    = {2026},
  pages   = {7013--7040},
  doi     = {10.5194/gmd-19-7013-2026},
  number = {14}
}

@article{kawai2025,
  author  = {Kawai, Y. and Tomita, H.},
  title   = {Development of a high-order global dynamical core using the discontinuous Galerkin method for an atmospheric large-eddy simulation (LES)
             and proposal of test cases: SCALE-DG v0.8.0},
  journal = {Geoscientific Model Development},
  volume  = {18},
  year    = {2025},
  number  = {3},
  pages   = {725--762},
  url     = {https://gmd.copernicus.org/articles/18/725/2025/},
  doi     = {10.5194/gmd-18-725-2025}
}

@article{yatunin_clima2016,
  author   = {Yatunin, Dennis and Byrne, Simon and Kawczynski, Charles and Kandala, Sriharsha and Bozzola, Gabriele and Sridhar, Akshay and Shen, Zhaoyi and Jaruga, Anna and Sloan, Julia and He, Jia and Huang, Daniel Zhengyu and Barra, Valeria and Chew, Ray and Boral, Anudhyan and Chen, Yi-fan and Knoth, Oswald and Ullrich, Paul and Mbengue, Cheikh and Schneider, Tapio},
  title    = {The Climate Modeling Alliance Atmosphere Dynamical Core: Concepts, Numerics, and Scaling},
  journal  = {Journal of Advances in Modeling Earth Systems},
  volume   = {18},
  number   = {3},
  pages    = {e2025MS005014},
  doi      = {10.1029/2025MS005014},
  year     = {2026}
}

@article{CAM-SE2018,
  author   = {Lauritzen, P. H. and Nair, R. D. and Herrington, A. R. and Callaghan, P. and Goldhaber, S. and Dennis, J. M. and Bacmeister, J. T. and Eaton, B. E. and Zarzycki, C. M. and Taylor, Mark A. and Ullrich, P. A. and Dubos, T. and Gettelman, A. and Neale, R. B. and Dobbins, B. and Reed, K. A. and Hannay, C. and Medeiros, B. and Benedict, J. J. and Tribbia, J. J.},
  title    = {NCAR Release of CAM-SE in CESM2.0: A Reformulation of the Spectral Element Dynamical Core in Dry-Mass Vertical Coordinates With Comprehensive Treatment of Condensates and Energy},
  journal  = {Journal of Advances in Modeling Earth Systems},
  volume   = {10},
  number   = {7},
  pages    = {1537-1570},
  doi      = {10.1029/2017MS001257},
  url      = {https://agupubs.onlinelibrary.wiley.com/doi/abs/10.1029/2017MS001257},
  year     = {2018}
}

@article{ranocha2018generalised,
  title       = {Generalised Summation-by-Parts Operators and Variable Coefficients},
  author      = {Ranocha, Hendrik},
  journal     = {Journal of Computational Physics},
  volume      = {362},
  pages       = {20--48},
  year        = {2018},
  month       = {06},
  publisher   = {Elsevier},
  doi         = {10.1016/j.jcp.2018.02.021},
  eprinttype  = {arxiv},
  eprintclass = {math.NA}
}

@article{bezanson2017julia,
  title     = {Julia: {A} Fresh Approach to Numerical Computing},
  author    = {Bezanson, Jeff and Edelman, Alan and Karpinski, Stefan and
               Shah, Viral B},
  journal   = {SIAM Review},
  volume    = {59},
  number    = {1},
  pages     = {65--98},
  year      = {2017},
  publisher = {SIAM},
  doi       = {10.1137/141000671}
}

@article{ranocha2022adaptive,
  title     = {Adaptive numerical simulations with {T}rixi.jl:
               {A} case study of {J}ulia for scientific computing},
  author    = {Ranocha, Hendrik and Schlottke-Lakemper, Michael and
               Winters, Andrew Ross and Faulhaber, Erik and Chan, Jesse and
               Gassner, Gregor J},
  journal   = {Proceedings of the JuliaCon Conferences},
  volume    = {1},
  number    = {1},
  pages     = {77},
  year      = {2022},
  month     = {01},
  publisher = {The Open Journal},
  doi       = {10.21105/jcon.00077}
}

@article{schlottkelakemper2021purely,
  title     = {A purely hyperbolic discontinuous {G}alerkin approach for
               self-gravitating gas dynamics},
  author    = {Schlottke-Lakemper, Michael and Winters, Andrew R and
               Ranocha, Hendrik and Gassner, Gregor J},
  journal   = {Journal of Computational Physics},
  pages     = {110467},
  year      = {2021},
  month     = {10},
  volume    = {442},
  publisher = {Elsevier},
  doi       = {10.1016/j.jcp.2021.110467}
}

@article{kraaijevanger1991contractivity,
  title     = {Contractivity of {R}unge-{K}utta methods},
  author    = {Kraaijevanger, Johannes Franciscus Bernardus Maria},
  journal   = {BIT Numerical Mathematics},
  volume    = {31},
  number    = {3},
  pages     = {482--528},
  year      = {1991},
  publisher = {Springer},
  doi       = {10.1007/BF01933264}
}

@article{rackauckas2017differentialequations,
  title     = {{DifferentialEquations.jl} {--} {A} Performant and Feature-Rich
               Ecosystem for Solving Differential Equations in {J}ulia},
  author    = {Rackauckas, Christopher and Nie, Qing},
  journal   = {Journal of Open Research Software},
  volume    = {5},
  number    = {1},
  pages     = {15},
  year      = {2017},
  publisher = {Ubiquity Press},
  doi       = {10.5334/jors.151}
}

@article{danisch2021makie,
  title   = {Makie.jl: Flexible high-performance data visualization for {J}ulia},
  author  = {Danisch, Simon and Krumbiegel, Julius},
  journal = {Journal of Open Source Software},
  volume  = {6},
  number  = {65},
  pages   = {3349},
  year    = {2021},
  doi     = {10.21105/joss.03349}
}

@article{rojas2021robustness,
  title       = {On the robustness and performance of entropy stable collocated discontinuous {G}alerkin methods},
  author      = {Rojas, Diego and Boukharfane, Radouan and Dalcin, Lisandro and Del Rey Fern{\'a}ndez, David C. and Ranocha, Hendrik and Keyes, David E. and Parsani, Matteo},
  journal     = {Journal of Computational Physics},
  volume      = {426},
  pages       = {109891},
  year        = {2021},
  month       = {2},
  publisher   = {Elsevier},
  doi         = {10.1016/j.jcp.2020.109891}
}

@article{chan2018discretely,
  title     = {On discretely entropy conservative and entropy stable
               discontinuous {G}alerkin methods},
  author    = {Chan, Jesse},
  journal   = {Journal of Computational Physics},
  volume    = {362},
  pages     = {346--374},
  year      = {2018},
  publisher = {Elsevier},
  doi       = {10.1016/j.jcp.2018.02.033}
}

@article{chan2019skew,
  title     = {Skew-symmetric entropy stable modal discontinuous {G}alerkin
               formulations},
  author    = {Chan, Jesse},
  journal   = {Journal of Scientific Computing},
  volume    = {81},
  number    = {1},
  pages     = {459--485},
  year      = {2019},
  publisher = {Springer},
  doi       = {10.1007/s10915-019-01026-w}
}

@article{svard2021boundary,
  title     = {Entropy stable boundary conditions for the {E}uler equations},
  author    = {Sv{\"a}rd, Magnus},
  journal   = {Journal of Computational Physics},
  volume    = {426},
  pages     = {109947},
  year      = {2021},
  publisher = {Elsevier},
  doi       = {10.1016/j.jcp.2020.109947}
}

@inbook{hindenlang2020wall,
  author    = {Hindenlang, Florian J. and Gassner, Gregor J. and Kopriva, David A.},
  title     = {Stability of Wall Boundary Condition Procedures for Discontinuous
               {G}alerkin Spectral Element Approximations of the Compressible
               {E}uler Equations},
  booktitle = {Spectral and High Order Methods for Partial Differential Equations
               ICOSAHOM 2018},
  editor    = {Sherwin, Spencer J. and Moxey, David and Peir{\'o}, Joaquim and
               Vincent, Peter E. and Schwab, Christoph},
  series    = {Lecture Notes in Computational Science and Engineering},
  volume    = {134},
  pages     = {3--19},
  publisher = {Springer},
  year      = {2020},
  doi       = {10.1007/978-3-030-39647-3_1}
}

@article{klein2009asymptotics,
  author  = {Klein, Rupert},
  title   = {Asymptotics, structure, and integration of sound-proof atmospheric flow equations},
  journal = {Theoretical and Computational Fluid Dynamics},
  volume  = {23},
  number  = {3},
  pages   = {161--195},
  year    = {2009},
  doi     = {10.1007/s00162-009-0104-y}
}

@article{booker1967critical,
  author  = {Booker, John R. and Bretherton, Francis P.},
  title   = {The critical layer for internal gravity waves in a shear flow},
  journal = {Journal of Fluid Mechanics},
  volume  = {27},
  number  = {3},
  pages   = {513--539},
  year    = {1967},
  doi     = {10.1017/S0022112067000515}
}

@article{ChandrashekarWBDG,
	author = {Chandrashekar, Praveen and Zenk, Markus},
	doi = {10.1007/s10915-016-0339-x},
	id = {Chandrashekar2017},
	journal = {Journal of Scientific Computing},
	number = {3},
	pages = {1062--1093},
	title = {Well-Balanced Nodal Discontinuous Galerkin Method for Euler Equations with Gravity},
	volume = {71},
	year = {2017},
  issn = {1573-7691}
}

@article{MICHELDANSAC2025106853,
title = {Towards a fully well-balanced and entropy-stable scheme for the Euler equations with gravity: General equations of state},
journal = {Computers \& Fluids},
volume = {303},
pages = {106853},
year = {2025},
issn = {0045-7930},
doi = {10.1016/j.compfluid.2025.106853},
author = {Victor Michel-Dansac and Andrea Thomann}
}

@article{KAPPELI2014199,
title = {Well-balanced schemes for the Euler equations with gravitation},
journal = {Journal of Computational Physics},
volume = {259},
pages = {199-219},
year = {2014},
issn = {0021-9991},
doi = {10.1016/j.jcp.2013.11.028},
author = {R. Käppeli and S. Mishra}
}

@article{smolarkiewicz1997forward,
author = {Piotr K. Smolarkiewicz and Len G. Margolin},
title = {On Forward-in-Time Differencing for Fluids: an Eulerian/Semi-Lagrangian Non-Hydrostatic Model for Stratified Flows},
journal = {Atmosphere-Ocean},
volume = {35},
number = {sup1},
pages = {127--152},
year = {1997},
publisher = {Taylor \& Francis},
doi = {10.1080/07055900.1997.9687345}
}

@inproceedings{spichtinger2008critical,
  author       = {Spichtinger, Peter and Smolarkiewicz, Piotr K.},
  title        = {Turbulence in cirrus clouds -- impact of critical layers},
  booktitle    = {Proceedings of the 15th International Conference on Clouds and Precipitation},
  eventtitle   = {15th International Conference on Clouds and Precipitation ({ICCP})},
  eventdate    = {2008-07-07/2008-07-11},
  venue        = {Canc\'{u}n, Mexico},
  note         = {Extended abstract},
  year         = {2008}
}

@article{grubisic1997critical,
  author  = {Grubi\v{s}i\'{c}, Vanda and Smolarkiewicz, Piotr K.},
  title   = {The Effect of Critical Levels on {3D} Orographic Flows: Linear Regime},
  journal = {Journal of the Atmospheric Sciences},
  volume  = {54},
  number  = {15},
  pages   = {1943--1960},
  year    = {1997},
  doi     = {10.1175/1520-0469(1997)054<1943:TEOCLO>2.0.CO;2}
}

@article{wang2013high,
  title={High-order {CFD} methods: current status and perspective},
  author={Wang, Zhijian J and Fidkowski, Krzysztof and
          Abgrall, R{\'e}mi and Bassi, Francesco and Caraeni, Doru
          and Cary, Andrew and Deconinck, Herman and Hartmann, Ralf
          and Hillewaert, Koen and Huynh, Hung T and Kroll, Norbert
          and May, Georg and Persson, Per-Olof and van Leer, Bram
          and Visbal, Miguel},
  journal={International Journal for Numerical Methods in Fluids},
  volume={72},
  number={8},
  pages={811--845},
  year={2013},
  publisher={Wiley Online Library},
  doi={10.1002/fld.3767}
}

@article{CREAN2018410,
title = {Entropy-stable summation-by-parts discretization of the Euler equations on general curved elements},
journal = {Journal of Computational Physics},
volume = {356},
pages = {410-438},
year = {2018},
issn = {0021-9991},
doi = {10.1016/j.jcp.2017.12.015},
url = {https://www.sciencedirect.com/science/article/pii/S0021999117308999},
author = {Jared Crean and Jason E. Hicken and David C. {Del Rey Fernández} and David W. Zingg and Mark H. Carpenter}
}

@article{NOLASCO2020104555,
title = {Optimized geometrical metrics satisfying free-stream preservation},
journal = {Computers \& Fluids},
volume = {207},
pages = {104555},
year = {2020},
issn = {0045-7930},
doi = {10.1016/j.compfluid.2020.104555},
url = {https://www.sciencedirect.com/science/article/pii/S0045793020301274},
author = {Irving Reyna Nolasco and Lisandro Dalcin and David C. {Del Rey Fernández} and Stefano Zampini and Matteo Parsani}
}

@article{ThomasLombard,
author = {Thomas, P. D. and Lombard, C. K.},
title = {Geometric Conservation Law and Its Application to Flow Computations on Moving Grids},
journal = {AIAA Journal},
volume = {17},
number = {10},
pages = {1030-1037},
year = {1979},
doi = {10.2514/3.61273}
}

@article{ZHANG20108918,
title = {On positivity-preserving high order discontinuous Galerkin schemes for compressible Euler equations on rectangular meshes},
journal = {Journal of Computational Physics},
volume = {229},
number = {23},
pages = {8918-8934},
year = {2010},
issn = {0021-9991},
doi = {10.1016/j.jcp.2010.08.016},
url = {https://www.sciencedirect.com/science/article/pii/S0021999110004535},
author = {Xiangxiong Zhang and Chi-Wang Shu}
}

@article{zhang2011maximum,
  title={Maximum-principle-satisfying and positivity-preserving
         high-order schemes for conservation laws: survey and new
         developments},
  author={Zhang, Xiangxiong and Shu, Chi-Wang},
  journal={Proceedings of the Royal Society of London A: Mathematical,
           Physical and Engineering Sciences},
  volume={467},
  number={2134},
  pages={2752--2776},
  year={2011},
  publisher={The Royal Society},
  doi={10.1098/rspa.2011.0153}
}

@misc{candelaresi2026massivelyparallelnumericalsimulations,
      title={Massively parallel numerical simulations with Julia},
      author={Simon Candelaresi and Benedict Geihe and Marco Artiano and Lars Christmann and Valentin Churavy and Andrés Rueda-Ramírez and Hendrik Ranocha and Gregor J. Gassner and Michael Schlottke-Lakemper},
      year={2026},
      archivePrefix={arXiv},
      primaryClass={math.NA},
      doi={10.48550/arXiv.2607.25866},
  eprint = {2607.25866}
}

@article{gassmann2013,
  author  = {Gassmann, Almut},
  title   = {A global hexagonal {C}-grid non-hydrostatic dynamical core ({ICON-IAP})
             designed for energetic consistency},
  journal = {Quarterly Journal of the Royal Meteorological Society},
  volume  = {139},
  pages   = {152--175},
  year    = {2013},
  doi     = {10.1002/qj.1960},
  number = {670}
}

@article{marras2016review,
  author  = {Marras, Simone and Kelly, James F. and Moragues, Margarida and
             M{\"u}ller, Andreas and Kopera, Michal A. and V{\'a}zquez, Mariano and
             Giraldo, Francis X. and Houzeaux, Guillaume and Jorba, Oriol},
  title   = {A Review of Element-Based {G}alerkin Methods for Numerical Weather
             Prediction: Finite Elements, Spectral Elements, and Discontinuous {G}alerkin},
  journal = {Archives of Computational Methods in Engineering},
  volume  = {23},
  pages   = {673--722},
  year    = {2016},
  doi     = {10.1007/s11831-015-9152-1},
  number = {4}
}

@article{gassner2018br1,
  author  = {Gassner, Gregor J. and Winters, Andrew R. and Hindenlang, Florian J. and
             Kopriva, David A.},
  title   = {The {BR1} Scheme is Stable for the Compressible {N}avier-{S}tokes Equations},
  journal = {Journal of Scientific Computing},
  volume  = {77},
  pages   = {154--200},
  year    = {2018},
  doi     = {10.1007/s10915-018-0702-1},
  number = {1}
}

@article{lukacova2022convergence,
  title={Convergence of Discontinuous {G}alerkin Schemes for the {E}uler
         Equations via Dissipative Weak Solutions},
  author={Luk{\'a}cov{\'a}-Medvidov{\'a} and {\"O}ffner, Philipp},
  journal={Applied Mathematics and Computation},
  volume={436},
  pages={127508},
  year={2023},
  publisher={Elsevier},
  doi={10.1016/j.amc.2022.127508},
  eprint={2202.10043},
  eprinttype={arxiv},
  eprintclass={math.NA}
}

@online{ranocha2026convergence,
  title={Convergence of entropy-conservative summation-by-parts
         discretizations to smooth solutions of hyperbolic
         conservation laws},
  author={Ranocha, Hendrik},
  year={2026},
  month={07},
  eprint={2607.27049},
  eprinttype={arxiv},
  eprintclass={math.NA},
  doi={10.48550/arXiv.2607.27049}
}

@online{chen2026error,
  title={The error estimate of entropy-stable discontinuous {G}alerkin
         methods for hyperbolic conservation laws},
  author={Chen, Xu-Kun and Liu, Yong and Shu, Chi-Wang},
  year={2026},
  month={08},
  eprint={2608.06642},
  eprinttype={arxiv},
  eprintclass={math.NA},
  doi={10.48550/arXiv.2608.06642}
}

@inproceedings{MultiFloatsJL,
  author = {Zhang, David Kai and Aiken, Alex},
  title = {High-Performance Branch-Free Algorithms for Extended-Precision Floating-Point Arithmetic},
  year = {2025},
  isbn = {9798400714665},
  publisher = {Association for Computing Machinery},
  address = {New York, NY, USA},
  url = {https://doi.org/10.1145/3712285.3759876},
  doi = {10.1145/3712285.3759876},
  booktitle = {Proceedings of the International Conference for High Performance Computing, Networking, Storage and Analysis},
  pages = {695–710},
  numpages = {16},
  location = {St. Louis, MO, USA},
  series = {SC '25}
}

@article{conde2022embedded,
  title={Embedded pairs for optimal explicit strong stability preserving
         {R}unge-{K}utta methods},
  author={Conde, Sidafa and Fekete, Imre and Shadid, John N},
  journal={Journal of Computational and Applied Mathematics},
  volume={412},
  pages={114325},
  year={2022},
  doi={10.1016/j.cam.2022.114325},
  eprint={1806.08693},
  eprinttype={arxiv},
  eprintclass={math.NA}
}
\end{document}